\documentclass[preprint,11pt]{elsarticle}
\usepackage{lipsum}

\newcommand\blfootnote[1]{%
  \begingroup
  \renewcommand\thefootnote{}\footnote{#1}%
  \addtocounter{footnote}{-1}%
  \endgroup
}

\usepackage[right=2.5cm,left=2.5cm,bottom=2.5cm,top=2.5cm]{geometry}
\usepackage[reqno]{amsmath}
\usepackage[utf8]{inputenc}
\usepackage{amssymb}
\usepackage{amsmath}
\usepackage{adjustbox}
\usepackage{amsthm}
\usepackage{graphicx}
\usepackage{hyperref}
\hypersetup{
  colorlinks = false,
  bookmarks  = true,
  breaklinks = true,
  linkcolor  = red,
  citecolor  = red,
  urlcolor   = green,
}
\graphicspath{{./images/}{./FIGURES/}{}} 
\usepackage{xcolor}
\usepackage{subcaption}
\usepackage{booktabs}
\usepackage{rotating}
\usepackage{multirow}
\usepackage{comment}
\usepackage{setspace}
\usepackage{rotating}
\usepackage{url}
\usepackage{paralist}
\usepackage{cleveref}
\usepackage{parskip}
\usepackage{mathrsfs}
\usepackage{graphicx}
\usepackage{caption}
\usepackage[algoruled,linesnumbered]{algorithm2e}
\usepackage{microtype}
\usepackage{natbib}
\usepackage{soul}
\usepackage{stmaryrd} 

\usepackage{blkarray}

\newcommand{\BPP}{{BPP}}
\newcommand{\CBPP}{{CBPP}}
\newcommand{\BPPMCF}{{BPPMCF}}
\newcommand{\BPPS}{{BPPS}}

\newcommand{\fone}{{\textnormal{ILP}_\textnormal{N}}}
\newcommand{\ftwo}{{\textnormal{ILP}^\dag_\textnormal{N}}}
\newcommand{\ftwobis}{{\textnormal{ILP}^\ddag_\textnormal{N}}}
\newcommand{\fthree}{{\textnormal{ILP}^{\star}_\textnormal{N}}}

\newcommand{\afone}{{\textnormal{ILP}_\textnormal{AF}}}
\newcommand{\aftwo}{{\textnormal{ILP}^\dag_\textnormal{AF}}}
\newcommand{\afthree}{{\textnormal{ILP}^{\star}_\textnormal{AF}}}

\newcommand{\afoneLP}{{\textnormal{LP}_\textnormal{AF}}}
\newcommand{\aftwoLP}{{\textnormal{LP}^\dag_\textnormal{AF}}}
\newcommand{\afthreeLP}{{\textnormal{LP}^{\star}_\textnormal{AF}}}

\newcommand{\foneLP}{{\textnormal{LP}_\textnormal{N}}}
\newcommand{\ftwoLP}{{\textnormal{LP}^\dag_\textnormal{N}}}
\newcommand{\ftwoLPbis}{{\textnormal{LP}^{\ddag}_\textnormal{N}}}
\newcommand{\fthreeLP}{{\textnormal{LP}^{\star}_\textnormal{N}}}

\def \comma{,}

\newcommand{\blue}[1]{#1}

\newcommand{\Rev}[1]{#1}

\newtheorem{proposition}{Proposition}
\newtheorem{lemma}{Lemma}
\theoremstyle{definition}

\usepackage[acronym]{glossaries}
\newacronym{milp}{MILP}{Mixed Integer Linear Programming}
\newacronym{mip}{MIP}{Mixed Integer Programming}
\newacronym{ilp}{ILP}{Integer Linear Programming}

\newcommand{\items}{{\mathcal{I}}}
\newcommand{\classes}{{\mathcal{C}}}
\newcommand{\bins}{{\mathcal{B}}}
\newcommand{\opt}{{OPT}}
\newcommand{\UB}{{UB}}

\newcommand{\bincost}{{r}}
\newcommand{\binnumberUB}{k}
\newcommand{\binnumberLB}{\check{k}}
\newcommand{\binnumberLBclass}{\gamma}

\newcommand{\itemarcs}{\mathcal{A}_{\items}(i)}
\newcommand{\classarcs}{\mathcal{A}_{\classes}(c)}

\newcommand{\ratio}{{\varrho}}

\newcommand{\solclassopt}{{\mathcal{S}^\star(I_c)}}

\makeatletter
\patchcmd{\pprintMaketitle}
  {\footnotesize\itshape\elsaddress\par\vskip36pt}
  {\footnotesize\itshape\elsaddress\par\vskip18pt}
  {}{\PackageWarning{mypreamble}{Patch di pprintMaketitle non riuscita}}
\makeatother

\usepackage{array}

\newcolumntype{L}[1]{>{\raggedright\arraybackslash}p{#1}}  
\newcolumntype{R}[1]{>{\raggedleft\arraybackslash}p{#1}}   

\usepackage{todonotes}

\allowdisplaybreaks

\journal{ }

\begin{document}

\begin{frontmatter}
  \title{\LARGE The Bin Packing Problem with Setups: Formulation\Rev{s}, Structural Properties and Computational Insights}

\author{Roberto Baldacci\textsuperscript{a}} 
\author{Fabio Ciccarelli\textsuperscript{b,$*$}}
\author{Valerio Dose\textsuperscript{b}} 
\author{Stefano Coniglio\textsuperscript{c}} 
\author{Fabio Furini\textsuperscript{b}}

\let\comma,

\affiliation{College of Science and Engineering. Hamad Bin Khalifa University. Qatar Foundation. Doha. Qatar}

\affiliation{Department of Computer, Control and Management Engineering Antonio Ruberti. Sapienza University of Rome. Rome. Italy}

\affiliation{Department of Economics. University of Bergamo. Bergamo. Italy}

\begin{abstract}
We introduce the \emph{Bin Packing Problem with Setups} (BPPS), a generalization of the classical Bin Packing Problem with applications in production planning and logistics. In this problem, the items are partitioned into classes, and packing items of a class in a bin incurs a setup weight and cost.
We propose a natural \textit{Integer Linear Programming} (ILP) formulation for the BPPS and analyze its \textit{Linear Programming} relaxation. We show that the resulting lower bound can be arbitrarily weak and introduce the \emph{Minimum Classes Inequalities} (MCIs), which guarantee a worst-case ratio of $1/2$ with respect to the optimal objective function value of the \BPPS{}. We also derive the \emph{Minimum Bins Inequality} (MBI) and an upper bound on the number of bins in any optimal solution, substantially reducing the formulation size.
\Rev{We further develop an arc-flow formulation for the BPPS based on a tailored graph construction and compression procedure. Its LP relaxation dominates that of the natural formulation, and both the MCIs and the MBI are extended to the arc-flow model.
Finally, we introduce a benchmark comprising $576$ randomly generated instances and $36$ real-world instances derived from a vehicle-routing application, and conduct extensive computational experiments. Results show that the natural formulation performs best on instances with small or medium item weights, whereas the arc-flow formulation is more effective for large item weights and on the real-world testbed.
}
\end{abstract}

\end{frontmatter}

\blfootnote{*Corresponding author. E-mail address: {\tt f.ciccarelli@uniroma1.it}}

\section{Introduction}\label{sec:intro}

Given an infinite number of identical \emph{bins} with a positive integer \emph{capacity} \( d \in \mathbb{Z}_{\ge 1} \), and a set \( \items = \{1,2,\ldots,n\} \) of \( n \) \emph{items}, where each item \( i \in \items \) has a positive integer \emph{weight} \( w_i \in \mathbb{Z}_{\ge 1} \), the \emph{Bin Packing Problem} (\BPP{}) consists of determining a partition of the items into the minimum number of \emph{packing patterns}, where a packing pattern is a subset of items \(S\subseteq\items\) whose total item weight does not exceed the bin capacity.
This classical packing problem is fundamental in Operations Research and has been extensively studied from both theoretical and algorithmic perspectives; see~\citet{coffman2013bin} and~\citet{DELORME20161} for comprehensive surveys.

In this paper, we introduce and study a generalization of the \BPP{}, which we call the \emph{Bin Packing Problem with Setups} (\BPPS{}).  
In this new problem, the item set \( \items \) is partitioned into \( m \) \emph{item classes}, collectively denoted by \( \mathcal{P} = \{\items_1,\items_2,\dots,\items_m\} \), where, for each class \( c \in \classes  =  \{1,2,\ldots,m\}\), \( \items_c \subseteq \items \) is the subset of items belonging to class \( c \). Since the classes form a partition of \( \items \), we have
\[
\items = \bigcup_{c \in \classes} \items_c \quad \text{ and }
\quad 
\items_{c} \cap \items_{g} = \emptyset 
\quad \text{for all } c, g \in \classes \text{ with } c \ne g,
\]
that is, each item belongs to exactly one class.
For each class \( c \in \classes \), we are also given a nonnegative integer \emph{setup weight} \( s_c \in \mathbb{Z}_{\ge 0} \) and a nonnegative integer \emph{setup cost} \( f_c \in \mathbb{Z}_{\ge 0} \). Whenever at least one item of class \( c \) is packed into a bin, we incur both a setup cost \( f_c \) and a reduction in available bin capacity equal to the setup weight \( s_c \). Each setup cost and weight is incurred only once per bin--class pair, regardless of how many items of that class are placed in the same bin. Throughout this paper, a class is said to be \emph{active} in a bin if and only if the bin contains at least one of its items.
Finally, we are given a positive integer \emph{bin cost} \( \bincost \in \mathbb{Z}_{\ge 1} \), which is incurred once for every used bin.

\Rev{For any subset of items $S\subseteq\items$, let $\mathcal{C}(S)=\{c\in\classes:S\cap\items_c\ne\emptyset\}$ denote its active classes. In the \BPPS{}, a packing pattern is a subset $S\subseteq\items$ whose total item weight plus the setup weights of its active classes does not exceed the bin capacity:}
{\color{black}
\[
\sum_{i \in S} w_i + \sum_{c \in \mathcal{C}(S)} s_c \le d.
\]
}
\Rev{By construction, every packing pattern can be assigned to a single bin. A partition $\mathcal{S} \subseteq 2^{\items}$ of the item set into packing patterns has cost }
\begin{equation} \label{eq:UB}
\Rev{\sum_{S\in\mathcal S} \bigg(\bincost + \sum_{c\in\mathcal C(S)} f_c \bigg)= \bincost\,|\mathcal S| + \sum_{S\in\mathcal S}\sum_{c\in\mathcal C(S)} f_c,}
\end{equation}
\Rev{that is, the sum, over all packing patterns $S \in \mathcal{S}$, of the cost of the bin containing the pattern plus the setup costs of the classes active in it.
The \BPPS{} consists in determining a partition into packing patterns of minimum cost. We denote by $\mathcal{S}^\star \subseteq 2^{\items}$ an optimal partition, and by $\opt$ the corresponding optimal objective function value, i.e., the minimum cost required to pack all the items. The value of objective function~\eqref{eq:UB} of any \BPPS{} solution $\mathcal{S}$ provides an upper bound on the \BPPS{} optimum, and is therefore denoted by $\UB(\mathcal{S})$.}

An instance of the \BPPS{} is defined by the tuple
$
(n, \boldsymbol{w}, d, m, \mathcal{P}, \boldsymbol{s}, \boldsymbol{f}, \bincost),
$
where \( \boldsymbol{w} \in \mathbb{Z}_{\ge 1}^n \) is the vector of item weights, \( \boldsymbol{s} \in \mathbb{Z}_{\ge 0}^m \) is the vector of setup weights, and \( \boldsymbol{f} \in \mathbb{Z}_{\ge 0}^m \) is the vector of setup costs of the classes. 
To ensure the feasibility of a \BPPS{} instance, we assume that, for each class \( c \in \classes \), the combined weight of any item \( i \in \items_c \) and its setup weight does not exceed the bin capacity, i.e., \( w_i + s_c \le d \). Moreover, we exclude the trivial case where all items can be packed in a single bin.

If $f_c = s_c = 0$ for all $c \in \classes$, the \BPPS{} is equivalent to the classical \BPP.  
Since the \BPP{} is strongly $\mathcal{NP}$-hard, the same complexity result applies to the \BPPS{}. It is worth noticing that, when $\bincost = 0$, the objective coincides with the total setup cost. In this case, the problem is equivalent to solving an instance of the \BPP{} for each class $c \in \classes$, where the capacity of the bin is reduced to $d - s_c$ and the item set corresponds to $\items_c$.

The bin cost~\( \bincost \) plays an important role in shaping the structure of optimal \BPPS{} solutions.  
A high value of~\( \bincost \) discourages the use of multiple bins, promoting solutions that employ fewer bins packed as tightly as possible---even at the expense of activating multiple classes within the same bin, which increases the total setup cost.  
Conversely, when~\( \bincost \) is small, it becomes advantageous to use more bins to avoid incurring high setup costs from combining items of different classes in the same bin.  
\Rev{This also implies that, unlike the classical \BPP{}, where the number of used bins in an optimal solution coincides with the minimum number of bins required for packing all the items, in the \BPPS{} using more bins than the minimum may lead to a better solution.}
This trade-off is illustrated by the two optimal \BPPS{} solutions depicted in Figure~\ref{fig:1}.

\begin{figure}[h]
\begin{center}
\ifx\JPicScale\undefined\def\JPicScale{6.5}\fi
\unitlength \JPicScale mm
\begin{tikzpicture}[x=21,y=\unitlength,inner sep=0pt]

\draw[fill=black!20!white] (1,0) rectangle (9,6); 

\draw[thick,-] (1,-0.3)--(1,6); 
\draw[thick,-] (0.7,0)--(9.5,0); 

\foreach \x in {1,3,5,7,9} {
  \draw[-] (\x,0)--(\x,-0.3);
}
\foreach \y in {1,2,3,4,5,6} {
  \draw[-] (0.7,\y)--(1,\y);
}

\draw (1,0) node[left=0.5cm] {$0$};
 \draw (1,3) node[left=0.5cm] {$3$};
\draw (1,6) node[left=0.5cm] {$d=6$};

\draw (2,0) node[below=0.5cm] {bin $1$};
\draw (4,0) node[below=0.5cm] {bin $2$};
\draw (6,0) node[below=0.5cm] {bin $3$};
\draw (8,0) node[below=0.5cm] {bin $4$};

\draw (5,-1) node[below=0.5cm] {(a) an optimal BPPS solution with $\bincost=10$ };

\draw[fill=white] (1,0) rectangle (3,3);  \draw (1,0) node[above right=0.2cm] {\small $w_1$};
\draw[fill=white] (1,3) rectangle (3,4);  \draw (1,3) node[above right=0.2cm] {\small $w_5$};
\draw[fill=white] (1,4) rectangle (3,5);  \draw (1,4) node[above right=0.2cm] {\small $s_1$};
\draw[fill=white] (1,5) rectangle (3,6);  \draw (1,5) node[above right=0.2cm] {\small $s_2$};

\draw[fill=white] (3,0) rectangle (5,3);  \draw (3,0) node[above right=0.2cm] {\small $w_2$};
\draw[fill=white] (3,3) rectangle (5,4);  \draw (3,3) node[above right=0.2cm] {\small $w_6$};
\draw[fill=white] (3,4) rectangle (5,5);  \draw (3,4) node[above right=0.2cm] {\small $s_1$};
\draw[fill=white] (3,5) rectangle (5,6);  \draw (3,5) node[above right=0.2cm] {\small $s_2$};

\draw[fill=white] (5,0) rectangle (7,3);  \draw (5,0) node[above right=0.2cm] {\small $w_3$};
\draw[fill=white] (5,3) rectangle (7,4);  \draw (5,3) node[above right=0.2cm] {\small $w_7$};
\draw[fill=white] (5,4) rectangle (7,5);  \draw (5,4) node[above right=0.2cm] {\small $s_1$};
\draw[fill=white] (5,5) rectangle (7,6);  \draw (5,5) node[above right=0.2cm] {\small $s_2$};

\draw[fill=white] (7,0) rectangle (9,3);  \draw (7,0) node[above right=0.2cm] {\small $w_4$};
\draw[fill=white] (7,3) rectangle (9,4);  \draw (7,3) node[above right=0.2cm] {\small $w_8$};
\draw[fill=white] (7,4) rectangle (9,5);  \draw (7,4) node[above right=0.2cm] {\small $s_1$};
\draw[fill=white] (7,5) rectangle (9,6);  \draw (7,5) node[above right=0.2cm] {\small $s_2$};

\draw[fill=black!20!white] (1+10,0) rectangle (11+10,6); 

\draw[thick,-] (1+10,-0.3)--(1+10,6);
\draw[thick,-] (0.7+10,0)--(11.5+10,0);

\foreach \x in {1,3,5,7,9,11} {
  \draw[-] (\x+10,0)--(\x+10,-0.3);
}
\foreach \y in {1,2,3,4,5,6} {
  \draw[-] (0.7+10,\y)--(1+10,\y);
}

\draw (2+10,0) node[below=0.5cm] {bin $1$};
\draw (4+10,0) node[below=0.5cm] {bin $2$};
\draw (6+10,0) node[below=0.5cm] {bin $3$};
\draw (8+10,0) node[below=0.5cm] {bin $4$};
\draw (10+10,0) node[below=0.5cm] {bin $5$};

\draw (6+10,-1) node[below=0.5cm] {(b) an optimal BPPS solution with $\bincost=1$};

\draw[fill=white] (1+10,0) rectangle (3+10,3); \draw (1+10,0) node[above right=0.2cm] {\small $w_1$};
\draw[fill=white] (1+10,3) rectangle (3+10,4); \draw (1+10,3) node[above right=0.2cm] {\small $s_1$};

\draw[fill=white] (3+10,0) rectangle (5+10,3); \draw (3+10,0) node[above right=0.2cm] {\small $w_2$};
\draw[fill=white] (3+10,3) rectangle (5+10,4); \draw (3+10,3) node[above right=0.2cm] {\small $s_1$};

\draw[fill=white] (5+10,0) rectangle (7+10,3); \draw (5+10,0) node[above right=0.2cm] {\small $w_3$};
\draw[fill=white] (5+10,3) rectangle (7+10,4); \draw (5+10,3) node[above right=0.2cm] {\small $s_1$};

\draw[fill=white] (7+10,0) rectangle (9+10,3); \draw (7+10,0) node[above right=0.2cm] {\small $w_4$};
\draw[fill=white] (7+10,3) rectangle (9+10,4); \draw (7+10,3) node[above right=0.2cm] {\small $s_1$};

\draw[fill=white] (9+10,0) rectangle (11+10,1); \draw (9+10,0) node[above right=0.2cm] {\small $w_5$};
\draw[fill=white] (9+10,1) rectangle (11+10,2); \draw (9+10,1) node[above right=0.2cm] {\small $w_6$};
\draw[fill=white] (9+10,2) rectangle (11+10,3); \draw (9+10,2) node[above right=0.2cm] {\small $w_7$};
\draw[fill=white] (9+10,3) rectangle (11+10,4); \draw (9+10,3) node[above right=0.2cm] {\small $w_8$};
\draw[fill=white] (9+10,4) rectangle (11+10,5); \draw (9+10,4) node[above right=0.2cm] {\small $s_2$};

\end{tikzpicture}
\end{center}

\caption{
  The figure considers a \BPPS{} instance with bin capacity \( d = 6 \), number of items \( n = 8 \), and number of classes \( m = 2 \).  
  The item weights are \( w_1 = w_2 = w_3 = w_4 = 3 \) and \( w_5 = w_6 = w_7 = w_8 = 1 \), while the setup weights and costs are \( s_1 = s_2 = 1 \), \( f_1 = 2 \) and \( f_2 = 3 \).  
  The items are partitioned into the classes \( \items_1 = \{1,2,3,4\} \) and \( \items_2 = \{5,6,7,8\} \).  
  The figure shows two optimal \BPPS{} solutions for this instance for two different values of the bin cost \( \bincost \).  
  Used bins are shown along the horizontal axis, while the vertical axis represents the bin capacity \( d \). For each bin, the figure displays the total weight of the packed items and the setup weight associated with their classes. The height of each item and setup block corresponds to its weight, visually illustrating bin utilization. The remaining space in each bin is shown in gray.
  In part~(a), with \( \bincost = 10 \), the optimal solution consists of four used bins: \( \big\{\{1,5\}, \{2,6\}, \{3,7\}, \{4,8\}\big\} \), yielding a total minimum cost of \( \opt = 60 \).  
  In part~(b), with \( \bincost = 1 \), the optimal solution consists of five used bins: \( \big\{\{1\}, \{2\}, \{3\}, \{4\}, \{5,6,7,8\}\big\} \), with a total minimum cost of \( \opt = 16 \).
}
\label{fig:1}
\end{figure}

\subsection{Applications and motivations}\label{sec:applications}

To the best of our knowledge, the \BPPS{} has not yet been addressed in the literature, despite its high practical relevance.
Setup costs and setup times are critical in many production planning applications where production lines or plants must be configured for specific activities; see, e.g., the books by~\citet{PochetBOOK} and~\citet{SawikBOOK} and the references therein. In such settings, products (items) are typically grouped into families (classes), and specific setup operations are required before manufacturing products of each family.
A typical application of the BPPS in production planning is as follows. Consider a set of identical production lines, each available for a fixed amount of time (corresponding to the bin capacity), and a fixed cost incurred whenever a line is activated. A set of products must be manufactured, with each product requiring a given processing time (item weight) and belonging to a single product class. Manufacturing items of a class on a production line requires a setup operation that consumes a predefined amount of time (setup weight) and incurs a monetary setup cost. These operations may, for example, involve the intervention of specialized technicians or the reconfiguration of equipment for a particular product family.
In this context, the \BPPS{} models the problem of determining a minimum-cost production plan. Consider, for instance, a production setting involving bottles, flasks, and pots, as in~\citet{CHEBIL201540}, where machinery requires setup operations when switching from one product type to another; another example is a series of sewing production lines, where garments of different shapes and colors are produced, and color changes require specific setup actions; see, e.g.,~\citet{FocacciFGGS16}.

Other important applications arise in logistics contexts where goods are distributed by vehicles that require special handling equipment, such as forklifts or conveyors, or by multi-compartment vehicles that transport different types of goods in separate compartments~\citep{iori_routing_2010,PollarisBCJL15,BRECHT2019591}.
In these applications, items represent customer orders to be delivered and are partitioned into classes, each requiring a specific piece of handling equipment. Bins correspond to the vehicles used for delivery. Assigning an order to a vehicle requires that the corresponding equipment be available on board, which reduces capacity and incurs equipment (setup) costs — e.g., from renting a forklift or installing specialized loading systems.
Similarly, in the chemical and pharmaceutical industries~\citep{bishara2006cold} and in the distribution of perishable goods~\citep{BRECHT2019591}, certain products require refrigerated or frozen storage. This reduces the available space in the vehicle (e.g., due to the installation of a bulkhead separating frozen and dry goods) and incurs additional costs (e.g., due to refrigeration energy consumption). \Rev{In this paper, we consider a set of real-world instances of exactly this type, arising from the distribution of fresh and frozen grocery products by a major international retailer.}
In such scenarios, the \BPPS{} models the problem of determining the optimal fleet composition and assigning customer orders to vehicles to minimize total distribution costs.

\subsection{Literature review on related problems}
\label{sec:lit}

The most closely related problem to the \BPPS{} is the \textit{Knapsack Problem with Setups} (KPS), introduced by~\citet{CG94} and~\citet{L98}. In this problem, items are partitioned into classes, and a fixed setup cost is incurred for each class whose items are packed in the knapsack. The objective is to select a subset of items that maximizes the total profit, defined as the sum of item profits minus the setup costs of the activated classes. As in the \BPPS{}, the setup associated with each active class also consumes part of the knapsack capacity, thereby creating a trade-off between item selection and class activation.
Several exact and heuristic algorithms have been proposed for the KPS, including MILP formulations, decomposition-based methods, and various combinatorial approaches~
\citep{A06,ARB08,CHEBIL201540,furini2018exact,DSS17,MPV09,pferschy2018improved,Yang2009}. 
Despite these similarities, the KPS differs fundamentally from the \BPPS{} in that it involves a single knapsack rather than multiple bins, and does not require packing all items.

Our study also connects to the broader literature on generalizations of the classical Bin Packing Problem (\BPP{}), aimed at capturing features that arise in complex real-world applications---see, e.g.,~\citep{baldacci_numerically_2024,ceselli2008optimization,DellAmicoDI12,DellAmicoFI20,MartinovicSCF23,sadykov2013bin,wei_new_2020}.
Among the most relevant and related variants,  the \textit{Class-Constrained Bin Packing Problem} (\CBPP{})~\citep{BORGES2020106455,Shachnai2001313,E2006class} assumes that items are partitioned into classes and imposes a limit on the number of distinct classes allowed per bin. The \textit{Bin Packing Problem with Minimum Color Fragmentation} (\BPPMCF{})~\citep{BOLOGNA25,bergman2019binary,casazza2014mathematical,mehrani2022models} aims to minimize the number of bins used for each color, encouraging the grouping of items of the same type, under the constraint that only a fixed number of bins is available.
While these models incorporate class-based constraints, they do not account for class-dependent setup costs or setup weights, nor do they model the trade-off between bin usage and class activation. To the best of our knowledge, the \BPPS{} is the first problem to address this setting by accounting for setup operations that are both capacity-consuming and cost-inducing, depending on the classes of items packed into each bin.
\Rev{Finally, it is worth mentioning the following two recent \BPP{} generalizations. \citet{Baldi2019} introduce the \textit{Generalized Bin Packing Problem with Bin-dependent Item Profits}, in which bins may have different costs and capacities, items may be compulsory or optional, and the profit generated by an item depends on the bin to which it is assigned. The objective is to minimize the cost of the selected bins minus the profits of the packed items.}
\Rev{\citet{Crainic2021} study \textit{Multi- and Single-Period Variable Cost and Size Bin Packing Problems with Assignment Costs}, which combine bin-selection costs with item-to-bin assignment costs; the multi-period variant also accounts for the availability of items and bins over time, while items that are not assigned to any bin incur an additional cost.}

\Rev{To position the \BPPS{} within the nomenclature introduced by~\citet{Crainic2021}, we use their labeling scheme $D/C/B/K/T[\cdot]$. Here, $D$ denotes the number of physical dimensions; $C\in\{U,V\}$ and $B\in\{U,V\}$ indicate whether bin costs and sizes, respectively, are unique or variable; $K\in\{S,M\}$ distinguishes between single- and multiple-type item settings; $T\in\{1,N,C\}$ denotes a single-period, multi-period, or continuous-time representation; and $[\cdot]$ lists additional problem-specific attributes. Accordingly, the \BPPS{} can be labeled
as $1/U/U/M=m/1[\text{Setup}]$: it is a one-dimensional, single-period problem with identical bins and $m$ item classes, in which activating a class in a bin incurs both a setup cost and an
additional capacity consumption. The attribute $[\text{Setup}]$ identifies a mechanism not covered by the attributes explicitly listed in~\citet{Crainic2021}.}

\subsection{Contributions and structure of the paper}

In Section~\ref{sec:ILP}, we propose a natural \textit{Integer Linear Programming} (ILP) formulation for the \BPPS{}, extending the classical formulation of the BPP.  
In Section~\ref{sec:strength1}, we study the structural properties of the \textit{Linear Programming} (LP) relaxation of the ILP formulation, deriving closed-form expressions for its optimal solution and showing that the associated lower bound can be arbitrarily poor in the worst case.  
\Rev{
In Section~\ref{sec:MCI}, we introduce an effective family of valid inequalities, called the \textit{Minimum Classes Inequalities} (MCIs), that strengthen the LP relaxation of the natural ILP model. We also analyze this strengthened relaxation and derive closed-form expressions for its optimal solution.
In Section~\ref{sec:strength2}, we prove that incorporating the MCIs yields a lower bound with worst-case performance ratio $1/2$ relative to the optimal objective function value of the \BPPS{}.}
In Section~\ref{sec:MBI}, we present an additional valid inequality, called the \textit{Minimum Bins Inequality} (MBI), which further strengthens the LP relaxation of the ILP model. In Section~\ref{sec:strength3}, we derive closed-form expressions for an optimal solution to the relaxation with the MBI.  
Finally, in Section~\ref{sec:bin_upper_bound}, we establish an upper bound on the number of bins used in any optimal \BPPS{} solution. This result enables a significant reduction in the number of variables and constraints of the ILP formulation.  Most of the proofs of these theoretical results are provided in Section~E.2 of the electronic companion due to space limitations. %
\Rev{In Section~\ref{sec:arcflow}, we present an \emph{arc-flow formulation} for the \BPPS{}, extending the construction of~\citet{Brand16} to handle item classes, setup weights, and setup costs. We show that the LP relaxation of this formulation dominates that of the natural ILP formulation, and we extend the MCIs and the MBI to this new formulation.}
Section~\ref{sec:computationals} presents the computational study. Since the \BPPS{} is a novel problem, we first introduce, in Section~\ref{sec:instances}, a benchmark library of \Rev{both real-world and} synthetic instances that captures its main features, including variations in item weights, class numbers, setup costs, and setup weights.  
\Rev{In Section~\ref{sec:performance}, we assess the computational performance of both the natural and the arc-flow ILP formulations for the \BPPS{}, evaluate the impact of the proposed enhancements on the effectiveness of such models, and compare the two formulation families against each other to identify the instance features that drive their performance.}
\Rev{In Section~\ref{sec:comp_LP}, we assess the strength of the LP relaxations of the natural and arc-flow formulation families, confirming that the arc-flow relaxation is consistently tighter, provide insights into the main structural features of optimal BPPS solutions, and discuss why the stronger relaxation of the arc-flow formulation does not always translate into superior computational performance.}
The paper concludes with Section~\ref{sec:conclusions}, which summarizes the contributions and outlines directions for future research.

\section{A natural ILP formulation for the BPPS}\label{sec:ILP}

In this section, we extend the classical ILP formulation for the \BPP{} (see, e.g., \citep{DELORME20161}) to the \BPPS{}. 
To this end, let $\binnumberUB \in \mathbb{Z}_{\ge 1}$ denote an upper bound on the number of bins used in any optimal \BPPS{} solution, and let $\bins = \{1,2,\dots,\binnumberUB\}$ be the set of candidate bins. The way valid values of $\binnumberUB$ can be obtained is discussed in Section~\ref{sec:bin_upper_bound}.  

For each item \( i \in \items \) and bin \( b \in \bins \), let \( x_{ib} \) be a binary variable equal to~1 if and only if item \( i \) is packed into bin \( b \).  
For each class \( c \in \classes \) and bin \( b \in \bins \), let \( y_{cb} \) be a binary variable equal to~1 if at least one item of class \( c \) is packed into bin \( b \).  
For each bin \( b \in \bins \), let \( z_b \) be a binary variable equal to~1 if bin \( b \) is used.  
The variables \( \boldsymbol{x} \in \{0,1\}^{n \, \binnumberUB} \) define the partition of the items into bins,  
the variables \( \boldsymbol{y} \in \{0,1\}^{m \, \binnumberUB} \) identify the active classes in each bin,  
and the variables \( \boldsymbol{z} \in \{0,1\}^{\binnumberUB} \) indicate which bins are used.  
Using these natural binary variables, we introduce the following ILP formulation for the \BPPS{}, which we refer to as~\(\fone\), where the subscript ``N'' stands for natural.
\begin{subequations}\label{form:kantorovich}
\begin{align}\label{con:kanto-obj}
&(\fone) &  \min_{\boldsymbol{x},\boldsymbol{y},\boldsymbol{z}} \quad  \sum_{b \in \bins} \bigg(  \bincost \, z_{b} + \sum_{c \in \classes} f_c \, y_{cb} \bigg) \\[1 ex]
\label{con:kanto-ass}
&& \sum_{b \in \bins} x_{ib}   & = 1,     &   i \in \items, \\[1 ex]
\label{con:kanto-cap}
&&    \sum_{i \in \items} w_i \, x_{ib} + \sum_{c \in \classes} s_c \; y_{cb}  & \leq d \, z_b,  &   b \in \bins, \\[1 ex]
\label{con:kanto-alpha}
&&    x_{ib}  & \leq   y_{cb},    &  c \in \classes, \,  i \in \items_c, \, b \in \bins.
\end{align} 
\end{subequations}
The objective function~\eqref{con:kanto-obj} to be minimized coincides with the sum of the fixed cost of the used bins and the total setup costs of the classes active in them.  
The \textit{assignment constraints}~\eqref{con:kanto-ass} require each item to be packed into exactly one bin.  
The \textit{capacity constraints}~\eqref{con:kanto-cap} ensure that the total weight of the items assigned to a bin, plus the setup weights of their active classes, does not exceed the bin capacity; they also guarantee that no items are placed in an unused bin.  
For each bin $b\in\bins$ and item $i\in\items$, the \textit{linking constraints}~\eqref{con:kanto-alpha} enforce that the class of item~$i$ is declared active in bin~$b$ whenever $x_{ib}=1$, thereby forcing the corresponding variable $y_{cb}$ to take value~1. 
Overall, formulation~$\fone$ involves $(n+m+1)\,\binnumberUB$ binary variables and $(n+1)\,\binnumberUB+n$ linear constraints.

\subsection{Strength of the LP relaxations of $\,\fone$} \label{sec:strength1}

Let us denote by $\foneLP$ the LP relaxation of $\fone$, obtained by replacing the binary variables with the following ones:
\begin{equation}\label{RELAX}
 0 \le x_{ib}  \le 1,~  i \in \items, \, b \in \bins,~~~~~ 
 0 \le y_{cb}  \le 1,~  c \in \classes, \, b \in \bins,~~~~~ 
 0 \le z_{b}  \le 1,~  b \in \bins. 
\end{equation} 
We denote by $\zeta(\foneLP)$ the optimal objective function value of $\foneLP$, which is a lower bound on the optimal objective function value~\(\opt\) of the \BPPS{}.  
The following proposition provides closed-form expressions for both an optimal solution to $\foneLP$ and its optimal objective function value.
\begin{proposition}\label{prop:lprelax}
For every \BPPS{} instance, an optimal solution to $\foneLP$ is
\begin{equation} \label{sol:LP1}
      x_{ib} = \frac 1 \binnumberUB, ~ i \in \items, \; b \in \bins, 
      ~~~~~ y_{cb} = \frac 1 \binnumberUB, ~  c \in \classes, \; b \in \bins, 
      ~~~~~  z_{b} = \frac{\sum_{i \in \items} w_i + \sum_{c \in \classes} s_c }{\binnumberUB \, d} ,  ~  b \in \bins, 
\end{equation}
and its optimal objective function value is
\begin{equation}\label{eq:LPoptval}
 \zeta(\foneLP) =   \frac{\bincost}{d} \left(\sum_{i \in \items} w_i + \sum_{c \in \classes} s_c \right) + \sum_{c \in \classes} f_c.
\end{equation}
\end{proposition}

Proposition~\ref{prop:lprelax} shows that an optimal solution to~$\foneLP$ is obtained by evenly distributing both the items and the setup weights across the available bins.  
The fractional value of each $z$-variable represents the proportion of the bin that is actually utilized.  

The optimal objective function value of $\foneLP$ is obtained by summing the setup costs of all classes (each counted once) and adding the bin cost multiplied by the ratio between the total item weight plus the total setup weights (each counted once) and the bin capacity.  

\blue{
Following the notation used, e.g., in \cite{MT90} and \cite{ManuelEJOR}, let $LB$ be a lower bound for a minimization problem.
For any instance $I$ of the problem, denote by $LB(I)$ the value of the lower bound $LB$ and by $OPT(I)$ the optimal objective function value (assumed positive).
The \textit{worst-case performance ratio} of $LB$ is defined as the largest real number $\ratio(LB)\in[0,1]$ such that
\begin{equation*}
    \frac{LB(I)}{OPT(I)} \ge \ratio(LB) ~~\text{for all instances}~I.
\end{equation*}
The following proposition shows that the worst-case performance ratio of the lower bound $\zeta(\foneLP)$\,\eqref{eq:LPoptval} is zero and, consequently, this lower bound can be arbitrarily weak.
The proof constructs a family of instances parameterized by the number of items.
In each instance, all items belong to a single class, and each bin can contain at most one item due to a very large setup weight.
In this case, an optimal \BPPS{} solution requires one bin per item. In contrast, in an optimal fractional solution of~$\foneLP$, both the setup and item weights can be fractionally distributed across all bins, yielding a much smaller optimal objective function value.
The claim then follows by letting the number of items tend to infinity.
\begin{proposition}\label{prop:ratio}
$$\ratio(\zeta(\foneLP)) = 0.$$
\end{proposition}
}

\begin{proof}
Consider the family of \BPPS{} instances, defined for each $n \ge 2$, with a single class ($m=1$, hence $\items_1=\items$), where $w_i = 1, i \in \items$, 
and $s_1 = n-1.$
Let $\binnumberUB=n$ and $d=n$, and assume that the setup cost of the class can take 
any non-negative integer value, i.e., $f_1 \in \mathbb{Z}_{\ge 0}$.
Since $w_i+s_1 = n$ for all items $i \in \items$, each bin can contain at most one item.  
Therefore, any optimal BPPS solution uses one bin per item, and the optimal objective function value of the \BPPS{} is $\opt = n \, (\bincost+f_1)$.
By Proposition~\ref{prop:lprelax}, the optimal objective function value of $\foneLP$ is
\[
\zeta(\foneLP) 
= \frac{\bincost}{n} \, \big(n + (n-1)\big)+f_1
= \frac{\bincost}{n} \, (2n-1)+f_1
= 2\bincost + f_1 - \frac{\bincost}{n} .
\]
Taking the limit of the ratio as $n \to \infty$ gives
\[
\lim_{n\to\infty}\;\frac{\zeta(\foneLP)}{\opt} 
= \lim_{n\to\infty}\; \frac{2\bincost+f_1-\frac{\bincost}{n}}{n\,(\bincost+f_1)}
= \lim_{n\to\infty}\;\left(\frac{2\bincost+f_1}{n(\bincost+f_1)}-\frac{\bincost}{n^2(\bincost+f_1)}\right)
= 0.
\]
\end{proof}

Due to Proposition~\ref{prop:lprelax}, in the family of instances considered in the proof of Proposition~\ref{prop:ratio}, all $z$-variables take the value
\[
 \frac{1}{n}\left( \frac{n}{n} + \frac{n-1}{n} \right)
= \frac{2n-1}{n^2},
\]
representing the fraction of each bin that is actually utilized in an optimal solution to $\foneLP$.  
Multiplying this value by the bin capacity $d=n$, we obtain that each bin is filled up to the value $(2n-1)/n$.  
Moreover, in this case, the lower bound on the number of used bins, given by the sum of the $z$-variables, is also $(2n-1)/n$, which converges to~2 as $n \to +\infty$.  
In contrast, any optimal \BPPS{} solution always requires exactly $n$ bins.  
An illustration of the difference between \BPPS{} and $\foneLP$ optimal solutions for the family of worst-case instances used in the proof of Proposition~\ref{prop:ratio} is provided in Figure~\ref{fig:ratio}.

\begin{figure}[h]
    \begin{center}
\ifx\JPicScale\undefined\def\JPicScale{5.5}\fi
\unitlength \JPicScale mm
\begin{tikzpicture}[x=18,y=\unitlength,inner sep=0pt]

\draw (2,4.5) node[below =0.7cm] {$\vdots$};
\draw (4,4.5) node[below =0.7cm] {$\vdots$};

\draw (8,4.5) node[below =0.7cm] {$\vdots$};
\draw (10,4.5) node[below =0.7cm] {$\vdots$};

\draw (2,9.5) node[below =0.9cm] {$///$};
\draw (4,9.5) node[below =0.9cm] {$///$};

\draw (8,9.5) node[below =0.9cm] {$///$};
\draw (10,9.5) node[below =0.9cm] {$///$};

\draw (6,3) node[below =0.7cm] {$\cdots$};
\draw (6,7) node[below =0.7cm] {$\cdots$};

\draw (1,0) node[left=0.25cm] {\small $0$};
\draw (1,6.5) node[left=0.25cm] {\small $\frac{2n-1}{n}$};
\draw (1,10) node[left=0.25cm] {\small $n$};

\draw (2,0) node[below =0.5cm] {bin $1$};
\draw (4,0) node[below =0.5cm] {bin $2$};
\draw (8,0) node[below =0.5cm] {bin $n \, \text{-} 1$};
\draw (10,0) node[below =0.5cm] {bin $n$};

\draw (6,-1.5) node[below =0.5cm] {(a) an optimal $\foneLP$ solution};

\draw[fill=white] (1,0) rectangle (3,1);
\draw (1,0) node[above right =0.1cm] {\small $ \frac{1}{n} w_1$};

\draw[fill=white] (1,1) rectangle (3,2);
\draw (1,1) node[above right =0.1cm] {\small $ \frac{1}{n} w_2$};

\draw[fill=white] (1,3) rectangle (3,4);
\draw (1,3) node[above right =0.1cm] {\small $ \frac{1}{n} w_{n-1}$};

\draw[fill=white] (1,4) rectangle (3,5);
\draw (1,4) node[above right =0.1cm] {\small $ \frac{1}{n} w_{n}$};

\draw[fill=white] (1,5) rectangle (3,6.5);
\draw (1,5) node[above right =0.2cm] {\small $ \frac{1}{n} s_1$};
\draw[fill=black!20!white] (1,6.5) rectangle (3, 7);

\draw[fill=black!20!white] (1,8) rectangle (3,10);

\draw[fill=white] (3,0) rectangle (5,1);
\draw (3,0) node[above right =0.1cm] {\small $ \frac{1}{n} w_1$};

\draw[fill=white] (3,1) rectangle (5,2);
\draw (3,1) node[above right =0.1cm] {\small $ \frac{1}{n} w_2$};

\draw[fill=white] (3,3) rectangle (5,4);
\draw (3,3) node[above right =0.1cm] {\small $ \frac{1}{n} w_{n-1}$};

\draw[fill=white] (3,4) rectangle (5,5);
\draw (3,4) node[above right =0.1cm] {\small $ \frac{1}{n} w_{n}$};

\draw[fill=white] (3,5) rectangle (5,6.5);
\draw (3,5) node[above right =0.2cm] {\small $ \frac{1}{n} s_1$};

\draw[fill=black!20!white] (3,6.5) rectangle (5,7);

\draw[fill=black!20!white] (3,8) rectangle (5,10);

\draw[fill=white] (7,0) rectangle (9,1);
\draw (7,0) node[above right =0.1cm] {\small $ \frac{1}{n} w_1$};

\draw[fill=white] (7,1) rectangle (9,2);
\draw (7,1) node[above right =0.1cm] {\small $ \frac{1}{n} w_2$};

\draw[fill=white] (7,3) rectangle (9,4);
\draw (7,3) node[above right =0.1cm] {\small $ \frac{1}{n} w_{n-1}$};

\draw[fill=white] (7,4) rectangle (9,5);
\draw (7,4) node[above right =0.1cm] {\small $ \frac{1}{n} w_{n}$};

\draw[fill=white] (7,5) rectangle (9,6.5);
\draw (7,5) node[above right =0.2cm] {\small $ \frac{1}{n} s_1$};

\draw[fill=black!20!white] (7,6.5) rectangle (9,7);

\draw[fill=black!20!white] (7,8) rectangle (9,10);

\draw[fill=white] (9,0) rectangle (11,1);
\draw (9,0) node[above right =0.1cm] {\small $ \frac{1}{n} w_1$};

\draw[fill=white] (9,1) rectangle (11,2);
\draw (9,1) node[above right =0.1cm] {\small $ \frac{1}{n} w_2$};

\draw[fill=white] (9,3) rectangle (11,4);
\draw (9,3) node[above right =0.1cm] {\small $ \frac{1}{n} w_{n-1}$};

\draw[fill=white] (9,4) rectangle (11,5);
\draw (9,4) node[above right =0.1cm] {\small $ \frac{1}{n} w_{n}$};

\draw[fill=white] (9,5) rectangle (11,6.5);
\draw (9,5) node[above right =0.2cm] {\small $ \frac{1}{n} s_1$};

\draw[fill=black!20!white] (9,6.5) rectangle (11,7);
\draw[fill=black!20!white] (9,8) rectangle (11,10);

\draw (2+13,9.5) node[below =0.9cm] {$///$};
\draw (4+13,9.5) node[below =0.9cm] {$///$};
\draw (8+13,9.5) node[below =0.9cm] {$///$};
\draw (10+13,9.5) node[below =0.9cm] {$///$};

\draw (6+13,3) node[below =0.7cm] {$\cdots$};

\draw (6+13,7) node[below =0.7cm] {$\cdots$};

\draw[fill=white] (1+13,3.5) rectangle (3+13,7);
\draw (1+13,3.5) node[above right =0.2cm] {\small $s_1$};
\draw[fill=white] (1+13,8) rectangle (3+13,10);

\draw[fill=white] (1+13,0) rectangle (3+13,3.5);
\draw (1+13,0) node[above right =0.2cm] {\small $w_1$};

\draw[fill=white] (3+13,3.5) rectangle (5+13,7);
\draw (3+13,3.5) node[above right =0.2cm] {\small $s_{1}$};
\draw[fill=white] (3+13,8) rectangle (5+13,10);

\draw[fill=white] (3+13,0) rectangle (5+13,3.5);
\draw (3+13,0) node[above right =0.2cm] {\small $w_{2}$};

\draw[fill=white] (7+13,3.5) rectangle (9+13,7);
\draw (7+13,3.5) node[above right =0.2cm] {\small $s_{1}$};
\draw[fill=white] (7+13,8) rectangle (9+13,10);

\draw[fill=white] (7+13,0) rectangle (9+13,3.5);
\draw (7+13,0) node[above right =0.2cm] {\small $w_{n-1}$};

\draw[fill=white] (9+13,3.5) rectangle (11+13,7);
\draw (9+13,3.5) node[above right =0.2cm] {\small $s_{1}$};
\draw[fill=white] (9+13,8) rectangle (11+13,10);

\draw[fill=white] (9+13,0) rectangle (11+13,3.5);
\draw (9+13,0) node[above right =0.2cm] {\small $w_{n}$};

\draw (2+13,0) node[below =0.5cm] {bin $1$};
\draw (4+13,0) node[below =0.5cm] {bin $2$};
\draw (8+13,0) node[below =0.5cm] {bin $n \, \text{-} 1$};
\draw (10+13,0) node[below =0.5cm] {bin $n$};

\draw (6+13,-1.5) node[below =0.5cm] {(b) an optimal BPPS solution};

\end{tikzpicture}
\end{center}

        \centering
\caption{The left part of the figure illustrates an optimal solution to~$\foneLP$ for the worst-case family of instances used in the proof of Proposition~\ref{prop:ratio}.  
Each bin is filled up to the value $\frac{2n-1}{n}$. The blocks represent the fraction of the item weights and of the class setup weight assigned to each bin.
The right part of the figure illustrates an optimal \BPPS{} solution for the same instance, in which each bin of capacity $n$ is completely filled by the weight of one item plus the class's setup weight.  
}
    \label{fig:ratio}
\end{figure}

\subsection{Minimum Classes Inequalities} 
\label{sec:MCI}

To strengthen $\foneLP$, we introduce the following family of inequalities, called the \emph{Minimum Classes Inequalities} (MCIs):
\begin{equation}\label{MCI}
  \sum_{b \in \bins} y_{cb} \;\geq\; 
  \underbrace{\left\lceil \frac{\sum_{i \in \items_c} w_i}{d - s_c}\right\rceil}_{=\binnumberLBclass_c},
  \qquad c \in \classes,
\end{equation}
where, for each class $c \in \classes$, the right-hand side~$\binnumberLBclass_c$ is a lower bound on the minimum number of bins required to pack all its items in any optimal BPPS solution.  
It is defined as the ceiling of the ratio between the total weight of the items in class~$c$ and the residual bin capacity
$
d - s_c,
$
namely, the capacity available after accounting for the setup weight of class~$c$.  
These inequalities guarantee that each class is activated in a sufficient number of bins, thereby excluding fractional solutions of~$\foneLP$ that would otherwise underestimate the actual packing requirements.  
The right-hand sides of the MCIs~\eqref{MCI} can be efficiently computed in linear time with respect to the number of items~$n$ of a \BPPS{} instance.
The following proposition states that the MCIs in~\eqref{MCI} are valid inequalities for~$\fone$.

\begin{proposition}\label{prop:mcoi}
For every \BPPS{} instance, the Minimum Classes Inequalities~\eqref{MCI} are satisfied by all feasible solutions of~$\fone$.
\end{proposition}

The proof uses the constraints of $\foneLP$, together with the integrality of the variables, to derive the MCIs for all integer-feasible solutions.
We denote by $\ftwo$ the formulation obtained by adding the MCIs in~\eqref{MCI} to $\fone$, and by $\ftwoLP$ its LP relaxation.  
We let $\zeta(\ftwoLP)$ denote the optimal objective function value of $\ftwoLP$, which provides a lower bound on the optimal objective function value~\(\opt\) of the \BPPS{}, and is
at least as strong as~$\zeta(\foneLP)$.  

Not only $\foneLP$, but also $\ftwoLP$ admits closed-form expressions for both an optimal solution to $\ftwoLP$ and its optimal objective function value, as stated by the following proposition.

\begin{proposition}\label{prop:lprelax_2}
For every \BPPS{} instance, an optimal solution to $\ftwoLP$ is 
\begin{equation}\label{sol:LP2}
 x_{ib} = \frac 1 \binnumberUB, ~ i \in \items, \; b \in \bins, 
 ~~~~~
 y_{cb} = \frac{\binnumberLBclass_c}{\binnumberUB}, ~ c \in \classes, \; b \in \bins, 
 ~~~~~
 z_{b} = \frac{\sum_{i \in \items} w_i + \sum_{c \in \classes} \binnumberLBclass_c \, s_c }{\binnumberUB \, d} ,  ~  b \in \bins,
\end{equation}
and its optimal objective function value is
\begin{equation}\label{eq:LPoptval_2}
\zeta(\ftwoLP) \;=\; \frac{\bincost}{d} \left(\sum_{i \in \items} w_i + \sum_{c \in \classes} \binnumberLBclass_c \, s_c \right) \;+\; \sum_{c \in \classes} \binnumberLBclass_c \, f_c.
\end{equation}
\end{proposition}

Proposition~\ref{prop:lprelax_2} shows that an optimal solution to~$\ftwoLP$ is again obtained by evenly distributing the items across the available bins, but now the setup weights of each class~$c \in \classes$ are distributed proportionally to~$\binnumberLBclass_c$.  
The fractional value of each $z$-variable represents the proportion of the bin that is actually utilized. These values are strictly larger than those in~\eqref{sol:LP1} whenever there exists a class $c\in\classes$ such that $\binnumberLBclass_c>1$ and $s_c>0$.
The optimal objective function value of $\ftwoLP$ is obtained by summing the setup costs of all classes, each counted $\binnumberLBclass_c$ times, and adding the bin cost multiplied by the ratio between the total item weight plus the setup weights of all classes, each counted $\binnumberLBclass_c$ times, and the bin capacity.

\subsection{Strength of the LP relaxations of $\,\ftwo$} \label{sec:strength2}

The inclusion of the MCIs~\eqref{MCI} yields the following main theoretical result of the paper: the lower bound $\zeta(\ftwoLP)$\,\eqref{eq:LPoptval_2} is always strictly greater than one half of the optimal objective function value of the \BPPS{}, $\opt$.
This result shows that, unlike $\foneLP$, the relaxation $\ftwoLP$ cannot yield arbitrarily weak lower bounds.

\begin{proposition}\label{prop:ratio_MCI}
For every \BPPS{} instance, we have
\[
\frac{\zeta(\ftwoLP)}{\opt} \;>\; \frac{1}{2}.
\]
\end{proposition}

The proof of this proposition is based on the construction of a heuristic solution to the BPPS, obtained by solving several BPPs defined accounting for the setup costs of the classes in different ways: either by reducing the capacity, or by increasing the item weights, or by combining the two ideas, depending on the class. This allows the use of classical results on the BPP to show the inequalities sought. 
The proof is somewhat surprising because it requires a careful control of the heuristic solution’s structure. In contrast, simpler constructions that would intuitively yield a better value do not make the result any easier to prove.

\blue{
The following proposition shows that the worst-case performance ratio of the lower bound $\zeta(\ftwoLP)$\,\eqref{eq:LPoptval_2} is $1/2$. 
The proof constructs a family of instances parameterized by the bin capacity, inspired by the classic worst-case examples for the \BPP{}, see, e.g., \cite[Section~8.3.2]{MT90}.
For these instances, the ratio $\zeta(\ftwoLP)/\opt$ can be made arbitrarily close to $1/2$.
The claim then follows by letting the bin capacity tend to infinity.
\begin{proposition}
\label{prop:worstcase}
$$\ratio(\zeta(\ftwoLP)) = \frac{1}{2}.$$
\end{proposition}
}
\vspace{-10mm}
\blue{
\begin{proof}
Consider the family of \BPPS{} instances with a single class ($m=1$, hence $\items_1=\items$), defined for each $\vartheta \in \mathbb{Z}_{\ge 1}$ by $d = 2\,\vartheta, \, w_i = \vartheta, i \in \items$, $s_1 = 1$, and $f_1 = 0$.
For each pair of distinct items $i, j \in \items$, we have that $w_i + w_j +s_1 = 2\, \vartheta +1 > d$. Hence, each bin can contain at most one item.
Therefore, any optimal BPPS solution uses one bin per item, and the optimal objective function value of the \BPPS{} is $\opt = n \, \bincost$.
Moreover, we have $\binnumberLBclass_1 \;=\; \left\lceil \frac{n \,\vartheta}{2 \, \vartheta-1} \right\rceil$
and, by Proposition~\ref{prop:lprelax_2} and letting $k = n$, the optimal objective function value of $\ftwoLP$ is
\[
\zeta(\ftwoLP)
= 
 \frac{\bincost}{2 \,\vartheta}\!\left(n \, \vartheta + \left\lceil \frac{n \, \vartheta}{2 \, \vartheta-1} \right\rceil \right)
= n \, \frac{\bincost}{2}
+ \frac{\bincost}{2 \,\vartheta}\left\lceil \frac{n \, \vartheta}{2 \, \vartheta-1} \right\rceil.
\]
Taking the limit of the ratio as $\vartheta \to \infty$ gives
\[
\lim_{\vartheta\to\infty}\;\frac{\zeta(\ftwoLP)}{\opt}
= \lim_{\vartheta\to\infty}\; \frac{n \, \frac{\bincost}{2}
+ \frac{\bincost}{2 \,\vartheta}\left\lceil \frac{n \, \vartheta}{2 \, \vartheta-1}\right\rceil}{n \, \bincost}
= \frac{1}{2}
+ \lim_{\vartheta\to\infty}\;\frac{1}{2n\vartheta}\left\lceil \frac{n \, \vartheta}{2 \, \vartheta-1}\right\rceil
= \frac{1}{2},
\]
where the last equality follows from the fact that
$\left\lceil \frac{n\vartheta}{2\vartheta-1}\right\rceil \le n$ for all $\vartheta\ge 1$,
and hence
$\frac{1}{2n\vartheta}\left\lceil \frac{n\vartheta}{2\vartheta-1}\right\rceil \le \frac{1}{2\vartheta}\to 0$ as $\vartheta \to \infty$. The result follows by combining the above limit with Proposition~\ref{prop:ratio_MCI}.
\end{proof}
}

Due to Proposition~\ref{prop:lprelax_2}, in the family of instances considered in the proof of Proposition~\ref{prop:worstcase}, all $z$-variables take the value
\[
 \frac{1}{2 \, \vartheta}\left( \frac{n \, \vartheta}{n} + \frac{\binnumberLBclass_1}{n} \right)
= \frac{\vartheta+\frac{\binnumberLBclass_1}{n}}{2 \, \vartheta}
\]
representing the fraction of each bin that is actually utilized in an optimal solution to $\ftwoLP$.  
Multiplying this value by the bin capacity $d=2\vartheta$, we obtain that each bin is filled up to the value $\vartheta+\frac{\binnumberLBclass_1}{n}$.  
Moreover, in this case, the lower bound on the number of used bins, given by the sum of the $z$-variables, is $(n\vartheta+\binnumberLBclass_1)/(2\vartheta)$, which converges to $n/2$ as $\vartheta \to +\infty$.  
In contrast, any optimal \BPPS{} solution always requires exactly $n$ bins.  
An illustration of the difference between \BPPS{} and $\ftwoLP$ optimal solutions for the worst-case instance used in the proof of Proposition~\ref{prop:worstcase} is provided in Figure~\ref{fig:example_bound_LC+}. 

\begin{figure}[h]
 
   \begin{center}
\ifx\JPicScale\undefined\def\JPicScale{5.5}\fi
\unitlength \JPicScale mm
\begin{tikzpicture}[x=18,y=\unitlength,inner sep=0pt]

\draw (2,5.5) node[below =0.7cm] {$\vdots$};
\draw (4,5.5) node[below =0.7cm] {$\vdots$};

\draw (8,5.5) node[below =0.7cm] {$\vdots$};
\draw (10,5.5) node[below =0.7cm] {$\vdots$};

\draw (2,9.5) node[below =0.9cm] {$///$};
\draw (4,9.5) node[below =0.9cm] {$///$};

\draw (8,9.5) node[below =0.9cm] {$///$};
\draw (10,9.5) node[below =0.9cm] {$///$};

\draw (6,3) node[below =0.7cm] {$\cdots$};
\draw (6,7) node[below =0.7cm] {$\cdots$};

\draw (1,0) node[left=0.25cm] {\small  $0$};
\draw (1,6) node[left=0.25cm] {\small{$\vartheta + \frac{\gamma_1}{n}$}};
\draw (1,10) node[left=0.25cm] {\small $2\,\vartheta$};

\draw (2,0) node[below =0.5cm] {bin $1$};
\draw (4,0) node[below =0.5cm] {bin $2$};
\draw (8,0) node[below =0.5cm] {bin $n \, \text{-} 1$};
\draw (10,0) node[below =0.5cm] {bin $n$};

\draw (6,-1.5) node[below =0.5cm] {(a) an optimal $\ftwoLP$ solution};

\draw[fill=white] (1,0) rectangle (3,1);
\draw (1,0) node[above right =0.1cm] {\small $ \frac{\gamma_1}{n} s_1$};

\draw[fill=white] (1,1) rectangle (3,2);
\draw (1,1) node[above right =0.1cm] {\small $ \frac{1}{n} w_1$};

\draw[fill=white] (1,2) rectangle (3,3);
\draw (1,2) node[above right =0.1cm] {\small $ \frac{1}{n} w_2$};

\draw[fill=white] (1,4) rectangle (3,5);
\draw (1,4) node[above right =0.1cm] {\small $ \frac{1}{n} w_{n-1}$};

\draw[fill=white] (1,5) rectangle (3,6);
\draw (1,5) node[above right =0.1cm] {\small $ \frac{1}{n} w_{n}$};

\draw[fill=black!20!white] (1,6) rectangle (5,7);
\draw[fill=black!20!white] (1,8) rectangle (5,10);

\draw[fill=white] (3,0) rectangle (5,1);
\draw (3,0) node[above right =0.1cm] {\small $ \frac{\gamma_1}{n} s_1$};

\draw[fill=white] (3,1) rectangle (5,2);
\draw (3,1) node[above right =0.1cm] {\small $ \frac{1}{n} w_1$};

\draw[fill=white] (3,2) rectangle (5,3);
\draw (3,2) node[above right =0.1cm] {\small $ \frac{1}{n} w_2$};

\draw[fill=white] (3,4) rectangle (5,5);
\draw (3,4) node[above right =0.1cm] {\small $ \frac{1}{n} w_{n-1}$};

\draw[fill=white] (3,5) rectangle (5,6);
\draw (3,5) node[above right =0.1cm] {\small $ \frac{1}{n} w_{n}$};

\draw[fill=black!20!white] (3,6) rectangle (5,7);
\draw[fill=black!20!white] (3,8) rectangle (5,10);

\draw[fill=white] (7,0) rectangle (9,1);
\draw (7,0) node[above right =0.1cm] {\small $ \frac{\gamma_1}{n} s_1$};

\draw[fill=white] (7,1) rectangle (9,2);
\draw (7,1) node[above right =0.1cm] {\small $ \frac{1}{n} w_1$};

\draw[fill=white] (7,2) rectangle (9,3);
\draw (7,2) node[above right =0.1cm] {\small $ \frac{1}{n} w_2$};

\draw[fill=white] (7,4) rectangle (9,5);
\draw (7,4) node[above right =0.1cm] {\small $ \frac{1}{n} w_{n-1}$};

\draw[fill=white] (7,5) rectangle (9,6);
\draw (7,5) node[above right =0.1cm] {\small $ \frac{1}{n} w_{n}$};

\draw[fill=black!20!white] (7,6) rectangle (9,7);
\draw[fill=black!20!white] (7,8) rectangle (9,10);

\draw[fill=white] (9,0) rectangle (11,1);
\draw (9,0) node[above right =0.1cm] {\small $ \frac{\gamma_1}{n} s_1$};

\draw[fill=white] (9,1) rectangle (11,2);
\draw (9,1) node[above right =0.1cm] {\small $ \frac{1}{n} w_1$};

\draw[fill=white] (9,2) rectangle (11,3);
\draw (9,2) node[above right =0.1cm] {\small $ \frac{1}{n} w_2$};

\draw[fill=white] (9,4) rectangle (11,5);
\draw (9,4) node[above right =0.1cm] {\small $ \frac{1}{n} w_{n-1}$};

\draw[fill=white] (9,5) rectangle (11,6);
\draw (9,5) node[above right =0.1cm] {\small $ \frac{1}{n} w_{n}$};

\draw[fill=black!20!white] (9,6) rectangle (11,7);
\draw[fill=black!20!white] (9,8) rectangle (11,10);

\draw[dashed] (7,2)--(7,3);
\draw[dashed] (9,2)--(9,3);
\draw[dashed] (11,2)--(11,3);

\draw (6+13,3) node[below =0.7cm] {$\cdots$};
\draw (6+13,7) node[below =0.7cm] {$\cdots$};

\draw[fill=white] (1+13,1.5) rectangle (3+13,7);
\draw (1+13,1.5) node[above right =0.2cm] {\small $w_1$};

\draw[fill=white] (1+13,0) rectangle (3+13,1.5);
\draw (1+13,0) node[above right =0.2cm] {\small $s_1$};

\draw[fill=black!20!white] (1+13,7) rectangle (3+13,8);
\draw[fill=black!20!white] (1+13,9) rectangle (3+13,10);

\draw[fill=white] (3+13,1.5) rectangle (5+13,7);
\draw (3+13,1.5) node[above right =0.2cm] {\small $w_{2}$};

\draw[fill=white] (3+13,0) rectangle (5+13,1.5);
\draw (3+13,0) node[above right =0.2cm] {\small $s_{1}$};

\draw[fill=black!20!white] (3+13,7) rectangle (5+13,8);
\draw[fill=black!20!white] (3+13,9) rectangle (5+13,10);

\draw[fill=white] (7+13,1.5) rectangle (9+13,7);
\draw (7+13,1.5) node[above right =0.2cm] {\small $w_{n-1}$};

\draw[fill=white] (7+13,0) rectangle (9+13,1.5);
\draw (7+13,0) node[above right =0.2cm] {\small $s_{1}$};

\draw[fill=black!20!white] (7+13,7) rectangle (9+13,8);
\draw[fill=black!20!white] (7+13,9) rectangle (9+13,10);

\draw[fill=white] (9+13,1.5) rectangle (11+13,7);
\draw (9+13,1.5) node[above right =0.2cm] {\small $w_{n}$};

\draw[fill=white] (9+13,0) rectangle (11+13,1.5);
\draw (9+13,0) node[above right =0.2cm] {\small $s_{1}$};

\draw[fill=black!20!white] (9+13,7) rectangle (11+13,8);
\draw[fill=black!20!white] (9+13,9) rectangle (11+13,10);

\draw (2+13,10.45) node[below =0.9cm] {$///$};
\draw (4+13,10.45) node[below =0.9cm] {$///$};

\draw (8+13,10.45) node[below =0.9cm] {$///$};
\draw (10+13,10.45) node[below =0.9cm] {$///$};

\draw (14.1,7) node[left=0.25cm] {\small{$\vartheta + 1$}};

\draw (2+13,0) node[below =0.5cm] {bin $1$};
\draw (4+13,0) node[below =0.5cm] {bin $2$};
\draw (8+13,0) node[below =0.5cm] {bin $n \, \text{-} 1$};
\draw (10+13,0) node[below =0.5cm] {bin $n$};

\draw (6+13,-1.5) node[below =0.5cm] {(b) an optimal BPPS solution};

\end{tikzpicture}
\end{center}

\caption{The left part of the figure illustrates an optimal solution to~$\ftwoLP$ for the worst-case family of instances used in the proof of Proposition~\ref{prop:worstcase}.  
Each bin is filled up to the value $\vartheta + \frac{\binnumberLBclass_1}{n}$.  
The blocks represent the fraction of the item weights and of the class setup weight assigned to each bin.  
The right part of the figure illustrates an optimal \BPPS{} solution to the same instance, in which at most one item can be packed in each bin of capacity $2\;\vartheta$, since two items together with the setup weight do not fit.
}
   \label{fig:example_bound_LC+}
\end{figure}

Propositions \ref{prop:ratio_MCI} and \ref{prop:worstcase} show that, with the inclusion of the MCIs, we can recover, for the BPPS, the classical result of a tight worst-case performance ratio of $1/2$ for the natural BPP formulation.

\subsection{Minimum Bins Inequality} 
\label{sec:MBI}

To further strengthen $\ftwoLP$, we introduce the following inequality, called the \emph{Minimum Bins Inequality} (MBI):
\begin{equation}\label{MBI}
  \sum_{b \in \bins} z_{b} \;\geq\; 
  \underbrace{\Biggl\lceil\frac{\sum_{i \in \items} w_i \;+\; \sum_{c \in \classes} \binnumberLBclass_c \, s_c}{d} \Biggr\rceil}_{=\binnumberLB},
\end{equation}
where the right-hand side $\binnumberLB$ is a lower bound on the minimum number of bins required to pack all items in any optimal BPPS solution while also accounting for setup weights.  
It is defined as the ceiling of the ratio between the sum of all item weights plus, for each class $c \in \classes$, its setup weight $s_c$ counted $\binnumberLBclass_c$ times and the bin capacity $d$.  
This inequality enforces that the total number of opened bins is sufficient to accommodate both item weights and the class-induced setup weights, thereby excluding fractional solutions of~$\ftwoLP$ that would otherwise underestimate the overall capacity requirement.  
The right-hand side of the MBI~\eqref{MBI} can be efficiently computed in linear time with respect to the number of items~$n$  of a \BPPS{} instance.
The following proposition shows that the MBI in~\eqref{MBI} is also a valid inequality for~$\fone$.

\begin{proposition}\label{prop:mbi}
For every \BPPS{} instance, the Minimum Bins Inequality~\eqref{MBI} is satisfied by every feasible solution of~$\fone$.
\end{proposition}

The proof is based on using the constraints of $\foneLP$,  together with the MCIs and the integrality of the $z$ variables, to derive the MBI for all integer feasible solutions.
We denote by $\ftwobis$ the formulation obtained by adding the MCIs~\eqref{MCI} and the MBI~\eqref{MBI}~to $\fone$, and by $\ftwoLPbis$ its LP relaxation.  
We let $\zeta(\ftwoLPbis)$ denote the optimal objective function value of $\ftwoLPbis$, which provides a lower bound on the optimal objective function value~\(\opt\) of the \BPPS{}, and is
at least as strong as~$\zeta(\ftwoLP)$.

\subsection{Strength of the LP relaxations of $\,\ftwobis$} \label{sec:strength3}

Not only $\foneLP$ and $\ftwoLP$, but also $\ftwoLPbis$ admits closed-form expressions for both an optimal solution to $\ftwoLPbis$ and its optimal objective function value, as shown by the following proposition.

\begin{proposition}\label{prop:lprelax_3}
For every \BPPS{} instance, an optimal solution to $\ftwoLPbis$ is
\begin{equation}\label{sol:LP2bis}
 x_{ib} = \frac 1 \binnumberUB, ~ i \in \items, \; b \in \bins, 
 ~~~~~
 y_{cb} = \frac{\binnumberLBclass_c}{\binnumberUB}, ~ c \in \classes, \; b \in \bins, 
 ~~~~~
 z_{b} = \frac \binnumberLB \binnumberUB ,  ~  b \in \bins, 
\end{equation}
and its optimal objective function value is
\begin{equation}\label{eq:LPoptval_2bis}
\zeta(\ftwoLPbis) \;=\; 
\bincost \, \binnumberLB 
\;+\; \sum_{c \in \classes} \binnumberLBclass_c \, f_c.
\end{equation}
\end{proposition}

While the MBI~\eqref{MBI} provides a global lower bound on the total number of bins that must be used, the MCIs~\eqref{MCI} operate at the class level by enforcing that each class $c \in \classes$ is activated in at least $\binnumberLBclass_c$ bins. These two families of inequalities are therefore complementary: the MBI~\eqref{MBI} constrains only the aggregate number of bins, without distinguishing which classes are responsible for their activation, whereas the MCIs~\eqref{MCI} ensure that the contribution of each class is explicitly enforced. 

\blue{The following proposition shows that the worst-case performance ratio of the lower bound $\zeta(\ftwoLPbis)$\,\eqref{eq:LPoptval_2bis} is $1/2$\Rev{, equal to that of $\zeta(\ftwoLP)$}.
\begin{proposition}\label{prop:worstcasebis}
$$\ratio(\zeta(\ftwoLPbis)) = \frac{1}{2}.$$
\end{proposition}
}

The proof employs the same instance of the BPPS used in Proposition \ref{prop:worstcase}, thereby showing the analogous result for the model $\ftwoLPbis$.
In our computational experiments of Section \ref{sec:computationals}, we empirically measure the actual improvement of the lower bound provided by $\ftwoLP$ and $\ftwoLPbis$ over $\foneLP$ across the tested instance classes.

\subsection{An upper bound to the number of bins used in any optimal BPPS solution} \label{sec:bin_upper_bound}

An obvious valid choice for the upper bound~$\binnumberUB$ on the number of bins in any optimal \BPPS{} solution is the number of items~$n$, since no solution can use more than one bin per item. This bound, however, is often weak in practice, as the number of bins used is typically much smaller than~$n$. Moreover, the value of~$\binnumberUB$ directly determines the number of variables and constraints in the formulation and thus strongly impacts computational performance. To address this issue, we develop a stronger, instance-dependent upper bound based on class-wise considerations. Specifically, for each class $c \in \classes$, let $\overline{\beta}_c$ denote an upper bound on the minimum number of bins of capacity $d-s_c$ required to pack all its items.

\begin{proposition} \label{upper_bound_bins}
For every \BPPS{} instance, the value 
\begin{equation}\label{eq:bin_UB}    
\hat{\binnumberUB} = \sum_{c \in \classes} \overline{\beta}_c,
\end{equation}
provides an upper bound on the number of bins used in any optimal \BPPS{} solution.
\end{proposition}

The proof is based on comparing the value of an optimal solution to the BPPS with the value of the first heuristic solution also considered in the proof of Proposition \ref{prop:ratio_MCI}, which is obtained by solving a BPP of capacity $d-s_c$ for every class $c\in\classes$. 

The values $\overline{\beta}_c$, with $c \in \classes$, can be determined by applying any algorithm that provides 
a feasible solution to this  \BPP{} associated with class~$c$, since such an algorithm would yield an upper bound on the number of bins required. 
In practice, both exact and heuristic algorithms for the \BPP{} can be employed for this 
purpose. The specific choices adopted to compute these values in our computational 
experiments will be discussed in Section~\ref{sec:computationals}.

{\color{black}
\section{An arc-flow ILP formulation for the BPPS}
\label{sec:arcflow}

In this section, we introduce an arc-flow formulation for the \BPPS{} by extending the approach of~\citet{Brand16}.
Originally introduced for the BPP by~\citet{deCarvalho1999}, arc-flow formulations model feasible packing patterns as source-to-sink paths in a \emph{directed acyclic graph} (DAG), using variables that represent flows on individual arcs of the DAG.
We refer to~\citet{DELIMA20223} for a recent survey on the dynamic-programming foundations of arc-flow formulations.

In the classical bin packing setting, a path encodes a packing pattern whose total weight does not exceed the bin capacity.
In the \BPPS{}, however, the feasibility and cost of a packing pattern also depend on the classes of the items it contains, since each active class consumes its setup weight and contributes its setup cost.
We encode these class activations directly in the DAG by introducing, for each class, a setup arc that must be traversed before any item of that class can be packed.
Thus, each feasible packing pattern, including both the packed items and the classes activated in the bin, is represented by a source-to-sink path in the DAG.
The associated arc-flow formulation then selects a minimum-cost collection of feasible packing patterns, while assignment constraints ensure that every item is packed exactly once.

\subsection{DAG construction}
\label{sec:afConstruction}

To derive the arc-flow formulation for the \BPPS{}, we construct a DAG that is denoted by $\mathcal{G} = (\mathcal{V}, \mathcal{A})$, with a \emph{source} $\sigma$ and a \emph{sink} $\tau$. Every source-to-sink path in this DAG encodes a feasible \BPPS{} packing pattern, whose total weight of packed items and active classes does not exceed the bin capacity.
Following the arc-flow paradigm based on dynamic programming~\citep{DELIMA20223}, each vertex of $\mathcal{G}$ corresponds to a dynamic programming \emph{state}.
The structure of the DAG is described by two dimensions: the \emph{stage} dimension, which represents the sequence of setup, item, and class-transition decisions, and the \emph{load} dimension, which records the current total weight in the bin.
Specifically, given a fixed ordering of the items, in the standard DAG construction for the \BPP{} a state is identified by a pair $(\ell, t)$, where $\ell \in \{0,1,\dots,d\}$ is the accumulated load, i.e., the total item weight packed so far, and $t$ is the stage, i.e., the position in the item ordering after the first $t$ items have been considered.
In the DAG for the \BPPS{}, we consider an ordering of the classes and of the items within each class. The sequence of stages is then enriched by adding setup stages and class-transition stages. Furthermore, in the \BPPS{} DAG, the accumulated load accounts for both the packed item weights and the active classes.
To account for class setups, we add one state component: two states with the same stage and load may lead to different subsequent decisions depending on whether the current class setup has already been activated and whether at least one item of that class has already been packed.
We therefore extend each state with an \emph{active flag} $e \in \{-1, 0, 1\}$, so that every vertex $v \in \mathcal{V}$ is characterized by a triple $(\ell, t, e)$.
The active flag~$e$ reflects the status of the current class along the path: $e = -1$ marks an entry/exit vertex, $e = 0$ marks a vertex at which the setup weight of the current class has been reserved but no item of that class has yet been packed, and $e = 1$ marks a vertex at which at least one item of the current class has been packed.

All vertices sharing the same stage index~$t$ form \emph{layer}~$t$.
These layers are of two types.
For each class $c \in \classes$, the \emph{class layer} of~$c$ consists of $|\items_c| + 1$ consecutive layers: one \emph{setup layer}, in which the setup weight of class~$c$ is reserved, followed by $|\items_c|$ \emph{item layers}, one per item of class~$c$.
Separating consecutive class layers are single \emph{entry/exit layers}.
These layers contain the boundary vertices from which a class can either be entered through its setup arc or bypassed entirely.
Setup and item arcs connect consecutive layers, whereas class-bypass arcs connect directly the entry and exit layers of a class.
The source is $\sigma = (0, 0, -1)$.
Each arc $a \in \mathcal{A}$ has a \emph{weight} $\omega_a \in \mathbb{Z}_{\ge 0}$ and a \emph{cost} $\lambda_a \in \mathbb{Z}_{\ge 0}$.
The weight $\omega_a$ represents the load consumed by traversing arc~$a$, while the cost $\lambda_a$ represents the contribution of arc~$a$ to the objective function.
In particular, setup arcs which model the activation of class $c \in \classes$ have weight $s_c$ and cost $f_c$, item arcs which model the packing of item~$i \in \items$ have weight $w_i$ and zero cost, bypass arcs have zero weight and zero cost, and loss arcs to the sink have zero weight and cost~$\bincost$.
For each item $i \in \items$, we denote by $\itemarcs \subseteq \mathcal{A}$ the set of item arcs that pack item~$i$.
Similarly, for each class $c \in \classes$, we denote by $\classarcs \subseteq \mathcal{A}$ the set of setup arcs that activate class~$c$.

The DAG is built by processing classes $c \in \classes$ in the order induced by the input instance. Let $t_{\mathrm{prev}}$ denote the stage index of the current entry/exit layer, initialized to $t_{\mathrm{prev}} = 0$. For each class~$c$, the construction proceeds as follows.

A \emph{bypass arc} (weight $0$, cost $0$) from each entry/exit vertex $(\ell, t_{\mathrm{prev}}, -1)$ to $(\ell, t_{\mathrm{exit}}, -1)$, where $t_{\mathrm{exit}} = t_{\mathrm{prev}} + |\items_c| + 2$, encodes packing patterns in which class~$c$ is absent.
For each entry/exit vertex $(\ell, t_{\mathrm{prev}}, -1)$ with $\ell + s_c \le d$, a \emph{setup arc} (weight $s_c$, cost $f_c$) to $(\ell + s_c,\; t_{\mathrm{prev}} + 1,\; 0)$ reserves the setup weight of class~$c$ and activates it in the current bin.

Items of class~$c$ are then processed in non-increasing order of weight, starting from stage $t = t_{\mathrm{prev}} + 1$. For each item $i \in \items_c$ and every vertex $(\ell, t, e)$ at stage~$t$, two arcs are added: a bypass arc to $(\ell, t+1, e)$, skipping item~$i$; and, if $\ell + w_i \le d$, an \emph{item arc} (weight $w_i$, cost $0$) to $(\ell + w_i,\; t+1,\; 1)$, packing item~$i$. After processing all items of class~$c$, a bypass arc from each vertex $(\ell, t_{\mathrm{exit}} - 1, 1)$ transitions to the exit layer $(\ell, t_{\mathrm{exit}}, -1)$. Vertices with flag $e = 0$ at stage $t_{\mathrm{exit}} - 1$---representing states in which class~$c$ was activated but none of its items was packed---are dominated and receive no outgoing arc.

After processing class~$c$, its exit layer becomes the entry layer for the next class, and $t_{\mathrm{prev}}$ is updated to $t_{\mathrm{exit}}$.
After the last class, a \emph{loss arc} (weight $0$, cost $\bincost$) connects each vertex of the final entry/exit layer with load $\ell \ge 1$ to~$\tau$.
Thus, the bin-opening cost is charged exactly once for every non-empty source-to-sink path, independently of the first class activated along the path.

Finally, a backward breadth-first search from~$\tau$ marks all vertices from which the sink is reachable; all unmarked vertices and their incident arcs are then removed. This step eliminates residual vertices at which a class was activated but the DAG contains no forward path reaching~$\tau$ with at least one item of that class packed. As a consequence, in the resulting uncompressed DAG, every path traversing a setup arc of class~$c$ packs at least an item $i \in \items_c$.

Figure~\ref{fig:af_construction} shows the DAG obtained after the construction phase for the example instance of Figure~\ref{fig:1}.

\begin{sidewaysfigure}
\centering
\resizebox{\textheight}{!}{%
{\color{black}
\tikzset{
  afentry/.style={draw, rectangle, rounded corners=1pt,
                  minimum size=5mm, fill=gray!30,
                  font=\tiny\bfseries, inner sep=1pt},
  afano1/.style={draw=green!55!black, circle, minimum size=5mm,
                 fill=white, font=\tiny, inner sep=1pt, thick},
  afayes1/.style={draw=green!55!black, circle, minimum size=5mm,
                  fill=green!25, font=\tiny, inner sep=1pt, thick},
  afano2/.style={draw=blue!65!black, circle, minimum size=5mm,
                 fill=white, font=\tiny, inner sep=1pt, thick},
  afayes2/.style={draw=blue!65!black, circle, minimum size=5mm,
                  fill=blue!20, font=\tiny, inner sep=1pt, thick},
  afdead/.style={dashed, fill=gray!8, draw=gray!45},
  afsource/.style={draw, rectangle, minimum size=5mm, fill=gray!50,
                 font=\tiny\bfseries, inner sep=1pt},
  afsink/.style={draw, rectangle, minimum size=5mm, fill=gray!50,
                 font=\tiny\bfseries, inner sep=1pt},
  afstep3/.style={draw, circle, minimum size=5mm,
                  fill=white, font=\tiny, inner sep=1pt},
  afsetup/.style={->, thick, orange!85!black},
  afitem1/.style={->, green!50!black, thin},
  afitem2/.style={->, blue!65!black, thin},
  afbyp/.style={->, gray!65, very thin},
  afexit/.style={->, gray!45, thin},
  afloss/.style={->, black, dashed, very thin},
  afcbp/.style={->, gray!65, very thin},  
}

\begin{tikzpicture}[
  x=1.28cm, y=0.9cm,
  >=stealth,
]

\node[font=\tiny, gray!70] at (-0.9,   0) {$d=0$};
\node[font=\tiny, gray!70] at (-0.9,   1) {$d=1$};
\node[font=\tiny, gray!70] at (-0.9,   2) {$d=2$};
\node[font=\tiny, gray!70] at (-0.9,   3) {$d=3$};
\node[font=\tiny, gray!70] at (-0.9,   4) {$d=4$};
\node[font=\tiny, gray!70] at (-0.9,   5) {$d=5$};
\node[font=\tiny, gray!70] at (-0.9,   6) {$d=6$};

\fill[gray!20]   (-0.55,-1.5) rectangle (0.55,7.5);   
\fill[green!10]  (0.55,-1.5)  rectangle (5.55,7.5);   
\fill[gray!12]   (5.55,-1.5)  rectangle (6.55,7.5);   
\fill[blue!10]   (6.55,-1.5)  rectangle (11.55,7.5);  
\fill[gray!12]   (11.55,-1.5) rectangle (12.55,7.5);  
\fill[gray!20]   (12.55,-1.5) rectangle (13.55,7.5);

\foreach \s/\lbl in {0/0,1/1,2/2,3/3,4/4,5/5,6/6,7/7,8/8,9/9,10/10,11/11,12/12,13/}{
  \node[font=\scriptsize, gray!70] at (\s, -1.8) {$\lbl$};
}

\foreach \s/\lbl in {1.5/1,2.5/2,3.5/3,4.5/4}{
  \node[font=\tiny, green!60!black] at (\s, 7.3) {$i = \lbl$};
}

\foreach \s/\lbl in {7.5/5,8.5/6,9.5/7,10.5/8}{
  \node[font=\tiny, blue!60!black] at (\s, 7.3) {$i = \lbl$};
}

\node[font=\scriptsize, green!60!black] at (3.05,   7.7) {$c=1$};
\node[font=\scriptsize, gray!60] at (6.05,   7.7) {entry/exit};
\node[font=\scriptsize, blue!60!black]  at (9.05,   7.7) {$c=2$};
\node[font=\scriptsize, gray!60] at (12.05,  7.7) {entry/exit};

\node[afsource] (S)     at (0,0)  {\footnotesize $\sigma$};
\node[afentry] (e6_0)  at (6,0)   {$0$};
\node[afentry] (e6_4)  at (6,4)   {$4$};
\node[afentry] (e12_2) at (12,2)  {$2$};
\node[afentry] (e12_3) at (12,3)  {$3$};
\node[afentry] (e12_4) at (12,4)  {$4$};
\node[afentry] (e12_5) at (12,5)  {$5$};
\node[afentry] (e12_6) at (12,6)  {$6$};

\node[afano1] (n1_1)  at (1,1) {$1$};
\node[afano1] (n2_1)  at (2,1) {$1$};
\node[afano1] (n3_1)  at (3,1) {$1$};
\node[afano1] (n4_1)  at (4,1) {$1$};

\node[afayes1] (n2_4) at (2,4) {$4$};
\node[afayes1] (n3_4) at (3,4) {$4$};
\node[afayes1] (n4_4) at (4,4) {$4$};
\node[afayes1] (n5_4) at (5,4) {$4$};

\node[afano2] (n7_1)  at (7,1)  {$1$};
\node[afano2] (n7_5)  at (7,5)  {$5$};
\node[afano2] (n8_1)  at (8,1)  {$1$};
\node[afano2] (n8_5)  at (8,5)  {$5$};
\node[afano2] (n9_1)  at (9,1)  {$1$};
\node[afano2] (n9_5)  at (9,5)  {$5$};
\node[afano2] (n10_1) at (10,1) {$1$};
\node[afano2] (n10_5) at (10,5) {$5$};

\node[afayes2] (n8_2)  at (8,2)     {$2$};
\node[afayes2] (n8_6)  at (8,6)     {$6$};
\node[afayes2] (n9_2)  at (9,2)     {$2$};
\node[afayes2] (n9_3)  at (9,3)     {$3$};
\node[afayes2] (n9_6)  at (9,6)     {$6$};
\node[afayes2] (n10_2) at (10,2)    {$2$};
\node[afayes2] (n10_3) at (10,3)    {$3$};
\node[afayes2] (n10_4) at (10,4)    {$4$};
\node[afayes2] (n10_6) at (10,6)    {$6$};
\node[afayes2] (n11_2) at (11,2)    {$2$};
\node[afayes2] (n11_3) at (11,3)    {$3$};
\node[afayes2] (n11_4) at (11,4)    {$4$};
\node[afayes2] (n11_5) at (11,5) {$5$};  
\node[afayes2] (n11_6) at (11,6)    {$6$};

\node[afsink]  (T) at (13,6) {\footnotesize $\tau$};

\draw[afsetup] (S) -- node[above, sloped, font=\tiny]{$(2, 1)$} (n1_1);

\draw[afcbp] (S) to
  node[below, font=\tiny]{$(0, 0)$} (e6_0);

\draw[afitem1] (n1_1)  -- node[above, sloped, font=\tiny]{$3$} (n2_4);
\draw[afitem1] (n2_1)  -- node[above, sloped, font=\tiny]{$3$} (n3_4);
\draw[afitem1] (n3_1)  -- node[above, sloped, font=\tiny]{$3$} (n4_4);
\draw[afitem1] (n4_1)  -- node[above, sloped, font=\tiny]{$3$} (n5_4);

\draw[afbyp] (n1_1) -- (n2_1);
\draw[afbyp] (n2_1) -- (n3_1);
\draw[afbyp] (n3_1) -- (n4_1);

\draw[afbyp] (n2_4) -- (n3_4);
\draw[afbyp] (n3_4) -- (n4_4);
\draw[afbyp] (n4_4) -- (n5_4);

\draw[afexit] (n5_4) -- (e6_4);

\draw[afsetup] (e6_0) -- node[above, sloped, font=\tiny]{$(3, 1)$} (n7_1);

\draw[afsetup] (e6_4) -- node[above, sloped, font=\tiny]{$(3,1)$} (n7_5);

\draw[afcbp] (e6_4) to[bend left=13] (e12_4);

\draw[afitem2] (n7_1)  -- node[above, sloped, font=\tiny]{$1$} (n8_2);
\draw[afitem2] (n7_5)  --  node[above, sloped, font=\tiny]{$1$} (n8_6);
\draw[afitem2] (n8_1)  --  node[above, sloped, font=\tiny]{$1$} (n9_2);
\draw[afitem2] (n8_2)  --  node[above, sloped, font=\tiny]{$1$} (n9_3);
\draw[afitem2] (n8_5)  --  node[above, sloped, font=\tiny]{$1$} (n9_6);
\draw[afitem2] (n9_1)  --  node[above, sloped, font=\tiny]{$1$} (n10_2);
\draw[afitem2] (n9_2)  --  node[above, sloped, font=\tiny]{$1$} (n10_3);
\draw[afitem2] (n9_3)  --  node[above, sloped, font=\tiny]{$1$} (n10_4);
\draw[afitem2] (n9_5)  --  node[above, sloped, font=\tiny]{$1$} (n10_6);
\draw[afitem2] (n10_1) --  node[above, sloped, font=\tiny]{$1$} (n11_2);
\draw[afitem2] (n10_2) --  node[above, sloped, font=\tiny]{$1$} (n11_3);
\draw[afitem2] (n10_3) --  node[above, sloped, font=\tiny]{$1$} (n11_4);
\draw[afitem2] (n10_4) --  node[above, sloped, font=\tiny]{$1$} (n11_5);
\draw[afitem2] (n10_5) --  node[above, sloped, font=\tiny]{$1$} (n11_6);

\draw[afbyp] (n7_1)  -- (n8_1);
\draw[afbyp] (n7_5)  -- (n8_5);
\draw[afbyp] (n8_1)  -- (n9_1);
\draw[afbyp] (n8_2)  -- (n9_2);
\draw[afbyp] (n8_5)  -- (n9_5);
\draw[afbyp] (n8_6)  -- (n9_6);
\draw[afbyp] (n9_1)  -- (n10_1);
\draw[afbyp] (n9_2)  -- (n10_2);
\draw[afbyp] (n9_3)  -- (n10_3);
\draw[afbyp] (n9_5)  -- (n10_5);
\draw[afbyp] (n9_6)  -- (n10_6);
\draw[afbyp] (n10_2) -- (n11_2);
\draw[afbyp] (n10_3) -- (n11_3);
\draw[afbyp] (n10_4) -- (n11_4);
\draw[afbyp] (n10_6) -- (n11_6);

\draw[afexit] (n11_2) -- (e12_2);
\draw[afexit] (n11_3) -- (e12_3);
\draw[afexit] (n11_4) -- (e12_4);
\draw[afexit] (n11_5) -- (e12_5);
\draw[afexit] (n11_6) -- (e12_6);

\draw[afloss] (e12_2) -- (T);
\draw[afloss] (e12_3) -- (T);
\draw[afloss] (e12_4) -- (T);
\draw[afloss] (e12_5) -- (T);
\draw[afloss] (e12_6) -- (T);

\end{tikzpicture}}%
}
\caption{
  DAG $\mathcal{G} = (\mathcal{V}, \mathcal{A})$ constructed for the \BPPS{} instance of Figure~\ref{fig:1}, with bin capacity $d = 6$, $n = 8$ items partitioned into $m = 2$ classes ($\mathcal{I}_1 = \{1,2,3,4\}$, $\mathcal{I}_2 = \{5,6,7,8\}$), item weights $w_1 = w_2 = w_3 = w_4 = 3$ and $w_5 = w_6 = w_7 = w_8 = 1$, setup weights $s_1 = s_2 = 1$, and setup costs $f_1 = 2$, $f_2 = 3$.
  The DAG is shown before the compression steps of Section~\ref{sec:afCompression}.
  Each vertex $(\ell, t, e) \in \mathcal{V}$ is placed at height $\ell$ on the vertical axis and at position $t$ on the horizontal axis; stages are indicated at the bottom.
  Vertex colors encode the active flag $e$: entry/exit vertices ($e = -1$) are drawn as gray squares; interior vertices with $e = 0$ (setup weight reserved, no item packed yet) are drawn as white circles; vertices with $e = 1$ (at least one item packed) are drawn as filled circles, green for class~$1$ and blue for class~$2$.
  Arc colors encode the arc type.
  Class setup arcs are drawn in orange and labeled with the pair $(\lambda_a, \omega_a)$: the setup arc of class~$1$ has label $(2, 1)$, reflecting setup cost $f_1 = 2$ and weight $s_1 = 1$. The setup arc of class~$2$ has label $(3, 1)$, reflecting cost $f_2 = 3$ and weight $s_2 = 1$.
  Item arcs ($a \in \itemarcs$) are drawn in the color of their class and labeled with $\omega_a = w_i$: arcs corresponding to items of class~$1$ are green and have label~3, while arcs corresponding to items of class~2 are blue and have label~1.
  Bypass arcs are drawn in gray.
  Loss arcs, connecting the final entry/exit layer to the sink~$\tau$, are drawn as dashed black arcs. Each has zero weight and incurs the bin cost $\bincost = 10$; this pair is not displayed as an arc label to keep the figure readable.
  The bypass arc from $\sigma$ to the entry/exit vertex at stage $t = 6$, labeled $(0, 0)$ in gray, encodes bins in which class~1 is entirely absent. Thus, arcs leaving $\sigma$ incur either zero cost or the setup cost of class~1, while the bin cost is charged only on the loss arcs entering~$\tau$.}
\label{fig:af_construction}
\end{sidewaysfigure}

\subsection{DAG compression}
\label{sec:afCompression}

The DAG, built as described in Section~\ref{sec:afConstruction}, is then compressed in two successive steps, corresponding to Steps~3 and~4 of~\citet{Brand16}, respectively.
For each compression step, we refer to the tuple of vertex attributes used to determine whether two vertices can be merged as the \emph{merge key}.
Since the \BPPS{} DAG contains explicit layers for setup, item, and class-bypass decisions, we keep the stage component in the merge key.
Thus, vertices belonging to different layers are never merged. This is a crucial difference with respect to the compression scheme of~\citet{Brand16}, where vertices may be merged even if they belong to different stages, provided that they have the same source and sink labels.
In our setting, merging states of different stages would disregard the role of class setup costs and could therefore weaken the resulting arc-flow formulation.

The idea behind the compression is to identify vertices that play the same role relative to the rest of a source-to-sink path.
In the first step, vertices are relabeled according to the maximum weight that can still be accumulated on a path from that vertex to the sink.
If two vertices are in the same layer $t$ and have the same value of this label, then the same amount of capacity must be left for completing any continuation from them; merging them only moves unused capacity to the beginning of the path, as in the compression procedure of~\citet{Brand16}.
The second step applies the symmetric idea from the source side.
After unused capacity has been shifted toward the beginning of the paths by the first compression step, vertices are relabeled according to the maximum weight that can be accumulated from the source to the vertex.
Vertices in the same layer with the same source label are indistinguishable with respect to the load that can be accumulated before reaching them, and can therefore be merged.
Keeping the stage in both merge keys preserves the order in which setup and item decisions are taken.

The first compression step works as follows. For each vertex $u \in \mathcal{V}$, define the \emph{sink label} $\phi(u, \tau)$ as the weight of the maximum weight path from $u$ to the sink~$\tau$ in $\mathcal{G}$.
Notice that $d - \phi(u, \tau)$ corresponds to the maximum load that a bin can have upon reaching~$u$ in order for every suffix reachable from~$u$ to remain within capacity.
Two vertices $u = (\ell, t, e)$ and $u' = (\ell', t', e')$ are then merged if and only if they share the same stage and the same sink label, i.e., $t = t'$ and $\phi(u, \tau) = \phi(u', \tau)$, resulting in a single new vertex $v = (d - \phi(u, \tau), \, t, \, \cdot )$. Feasibility is preserved under merging: any path entering the merged vertex $v$ has total weight at most $d - \phi(u, \tau)$, and any path leaving from that vertex has total arc weight at most $\phi(u, \tau)$, so the combined weight does not exceed~$d$.

The second compression step works as follows. On the DAG resulting from the sink-label compression step, define the \emph{source label} $\phi(\sigma, u)$ as the weight of the maximum weight path from the source~$\sigma$ to $u$ in $\mathcal{G}$.
Two vertices $u = (\ell, t, e)$ and $u' = (\ell', t', e')$ are then merged if and only if they share the same stage and the same source label, i.e., $t = t'$ and $\phi(\sigma, u) =  \phi(\sigma, u')$, resulting in a single new vertex $v = (\phi(\sigma, u), \, t, \, \cdot )$.
Feasibility is preserved by an argument analogous to that of the sink-label compression step.
The stage component in the merge key preserves the order of class and item layers after compression.

We deliberately do not include the active flag in the merge keys. Consequently, compression may introduce source-to-sink paths that traverse the setup arc of a class without traversing any of its item arcs. Such paths do not correspond to proper BPPS packing patterns, as they contain an unnecessary class activation. However, when setup weights and setup costs are nonnegative, every such path is dominated by one that bypasses the unused class, thereby consuming no more capacity and incurring no greater cost. We therefore accept these dominated paths in exchange for the greater reduction in DAG size obtained by disregarding the active flag during compression.

Figure~\ref{fig:af_compression} shows the DAG obtained after the compression phase for the example instance of Figure~\ref{fig:1}.

\begin{sidewaysfigure}
\centering
\resizebox{\textheight}{!}{%
{\color{black}
\tikzset{
  afentry/.style={draw, rectangle, rounded corners=1pt,
                  minimum size=5mm, fill=gray!30,
                  font=\tiny\bfseries, inner sep=1pt},
  afano1/.style={draw=green!55!black, circle, minimum size=5mm,
                 fill=white, font=\tiny, inner sep=1pt, thick},
  afayes1/.style={draw=green!55!black, circle, minimum size=5mm,
                  fill=green!25, font=\tiny, inner sep=1pt, thick},
  afano2/.style={draw=blue!65!black, circle, minimum size=5mm,
                 fill=white, font=\tiny, inner sep=1pt, thick},
  afayes2/.style={draw=blue!65!black, circle, minimum size=5mm,
                  fill=blue!20, font=\tiny, inner sep=1pt, thick},
  afdead/.style={dashed, fill=gray!8, draw=gray!45},
  afsource/.style={draw, rectangle, minimum size=5mm, fill=gray!50,
                 font=\tiny\bfseries, inner sep=1pt},
  afsink/.style={draw, rectangle, minimum size=5mm, fill=gray!50,
                 font=\tiny\bfseries, inner sep=1pt},
  afstep3/.style={draw, circle, minimum size=5mm,
                  fill=white, font=\tiny, inner sep=1pt},
  afsetup/.style={->, thick, orange!85!black},
  afitem1/.style={->, green!50!black, thin},
  afitem2/.style={->, blue!65!black, thin},
  afbyp/.style={->, gray!65, very thin},
  afexit/.style={->, gray!45, thin},
  afloss/.style={->, black, dashed, very thin},
  afcbp/.style={->, gray!65, very thin},  
}

\begin{tikzpicture}[
  x=1.28cm, y=0.9cm,
  >=stealth,
]

\node[font=\tiny, gray!70] at (-0.9,   0) {$d=0$};
\node[font=\tiny, gray!70] at (-0.9,   1) {$d=1$};
\node[font=\tiny, gray!70] at (-0.9,   2) {$d=2$};
\node[font=\tiny, gray!70] at (-0.9,   3) {$d=3$};
\node[font=\tiny, gray!70] at (-0.9,   4) {$d=4$};
\node[font=\tiny, gray!70] at (-0.9,   5) {$d=5$};
\node[font=\tiny, gray!70] at (-0.9,   6) {$d=6$};

\fill[gray!20]   (-0.55,-1.5) rectangle (0.55,7.5);   
\fill[green!10]  (0.55,-1.5)  rectangle (5.55,7.5);   
\fill[gray!12]   (5.55,-1.5)  rectangle (6.55,7.5);   
\fill[blue!10]   (6.55,-1.5)  rectangle (11.55,7.5);  
\fill[gray!12]   (11.55,-1.5) rectangle (12.55,7.5);  
\fill[gray!20]   (12.55,-1.5) rectangle (13.55,7.5);

\foreach \s/\lbl in {0/0,1/1,2/2,3/3,4/4,5/5,6/6,7/7,8/8,9/9,10/10,11/11,12/12,13/}{
  \node[font=\scriptsize, gray!70] at (\s, -1.8) {$\lbl$};
}

\foreach \s/\lbl in {1.5/1,2.5/2,3.5/3,4.5/4}{
  \node[font=\tiny, green!60!black] at (\s, 7.3) {$i = \lbl$};
}

\foreach \s/\lbl in {7.5/5,8.5/6,9.5/7,10.5/8}{
  \node[font=\tiny, blue!60!black] at (\s, 7.3) {$i = \lbl$};
}

\node[font=\scriptsize, green!60!black] at (3.05,   7.7) {$c=1$};
\node[font=\scriptsize, gray!60] at (6.05,   7.7) {entry/exit};
\node[font=\scriptsize, blue!60!black]  at (9.05,   7.7) {$c=2$};
\node[font=\scriptsize, gray!60] at (12.05,  7.7) {entry/exit};

\node[afsource] (S)     at (0,0)   {$\footnotesize \sigma$};
\node[afentry] (e6_1)  at (6,1)   {$1$};
\node[afentry] (e6_4)  at (6,4)   {$4$};
\node[afentry] (e12_6) at (12,6)  {$6$};

\node[afano1, fill=green!10] (n1_1)  at (1,1) {$1$};
\node[afano1, fill=green!10] (n2_1)  at (2,1) {$1$};
\node[afano1, fill=green!10] (n3_1)  at (3,1) {$1$};
\node[afano1, fill=green!10] (n4_1)  at (4,1) {$1$};

\node[afayes1, fill=green!10] (n2_4) at (2,4) {$4$};
\node[afayes1, fill=green!10] (n3_4) at (3,4) {$4$};
\node[afayes1, fill=green!10] (n4_4) at (4,4) {$4$};
\node[afayes1, fill=green!10] (n5_4) at (5,4) {$4$};

\node[afano2, fill=blue!10] (n7_2)  at (7,2)  {$2$};
\node[afano2, fill=blue!10] (n7_5)  at (7,5)  {$5$};
\node[afano2, fill=blue!10] (n8_5)  at (8,5)  {$5$};
\node[afano2, fill=blue!10] (n9_4)  at (9,4)  {$4$};
\node[afano2, fill=blue!10] (n9_5)  at (9,5)  {$5$};
\node[afano2, fill=blue!10] (n10_5) at (10,5) {$5$};

\node[afayes2, fill=blue!10] (n8_3)  at (8,3)     {$3$};
\node[afayes2, fill=blue!10] (n8_6)  at (8,6)     {$6$};
\node[afayes2, fill=blue!10] (n9_6)  at (9,6)     {$6$};
\node[afayes2, fill=blue!10] (n10_6) at (10,6)    {$6$};
\node[afayes2, fill=blue!10] (n11_6) at (11,6)    {$6$};

\node[afsink]  (T) at (13,6) {$\footnotesize \tau$};

\draw[afsetup] (S) -- node[above, sloped, font=\tiny]{$(2, 1)$} (n1_1);

\draw[afcbp] (S) to[bend right = 13]
  node[below, sloped, font=\tiny]{$(0, 0)$} (e6_1);

\draw[afitem1] (n1_1)  -- node[above, sloped, font=\tiny]{$3$} (n2_4);
\draw[afitem1] (n2_1)  -- node[above, sloped, font=\tiny]{$3$} (n3_4);
\draw[afitem1] (n3_1)  -- node[above, sloped, font=\tiny]{$3$} (n4_4);
\draw[afitem1] (n4_1)  -- node[above, sloped, font=\tiny]{$3$} (n5_4);

\draw[afbyp] (n1_1) -- (n2_1);
\draw[afbyp] (n2_1) -- (n3_1);
\draw[afbyp] (n3_1) -- (n4_1);

\draw[afbyp] (n2_4) -- (n3_4);
\draw[afbyp] (n3_4) -- (n4_4);
\draw[afbyp] (n4_4) -- (n5_4);

\draw[afexit] (n5_4) -- (e6_4);

\draw[afsetup] (e6_1) -- node[above, sloped, font=\tiny]{$(3, 1)$} (n7_2);

\draw[afsetup] (e6_4) -- node[above, sloped, font=\tiny]{$(3,1)$} (n7_5);

\draw[afitem2] (n7_5)  --  node[above, sloped, font=\tiny]{$1$} (n8_6);
\draw[afitem2] (n8_5)  --  node[above, sloped, font=\tiny]{$1$} (n9_6);
\draw[afitem2] (n9_5)  --  node[above, sloped, font=\tiny]{$1$} (n10_6);

\draw[afitem2] (n9_4) to[bend right=12] node[below, sloped, font=\tiny]{$1$} (n10_5);
\draw[afbyp] (n9_4) to[bend left=12] (n10_5);

\draw[afitem2] (n10_5) to[bend right=12] node[below, sloped, font=\tiny]{$1$} (n11_6);
\draw[afbyp] (n10_5) to[bend left=12] (n11_6);

\draw[afitem2] (n8_3) to[bend right=12] node[below, sloped, font=\tiny]{$1$} (n9_4);
\draw[afbyp] (n8_3) to[bend left=12] (n9_4);

\draw[afitem2] (n7_2) to[bend right=12] node[below, sloped, font=\tiny]{$1$} (n8_3);
\draw[afbyp] (n7_2) to[bend left=12] (n8_3);

\draw[afcbp] (e6_4) -- (6,6.7) -- (12,6.7) -- (e12_6);

\draw[afbyp] (n7_5)  -- (n8_5);
\draw[afbyp] (n8_5)  -- (n9_5);
\draw[afbyp] (n8_6)  -- (n9_6);
\draw[afbyp] (n9_5)  -- (n10_5);
\draw[afbyp] (n9_6)  -- (n10_6);
\draw[afbyp] (n10_6) -- (n11_6);

\draw[afexit] (n11_6) -- (e12_6);

\draw[afloss] (e12_6) -- (T);

\end{tikzpicture}}%
}
\caption{
  Compressed DAG $\mathcal{G} = (\mathcal{V}, \mathcal{A})$ for the \BPPS{} instance of Figure~\ref{fig:1}, obtained after applying the sink-label and source-label compression steps detailed in Section~\ref{sec:afCompression} to the DAG of Figure~\ref{fig:af_construction}.
  Vertex placement follows the same convention as in Figure~\ref{fig:af_construction}: the vertical axis reports the load $\ell$ and the horizontal axis reports the stage $t$.
  After compression, the active flag $e$ is not displayed via different vertex colors, as it carries no information in the compressed DAG.
  Vertex and arc colors follow the same convention as in Figure~\ref{fig:af_construction}.
  The effect of compression is visible in the class~$2$ layer ($t \in \{7, 8, \ldots, 11\}$): the $22$ vertices present in the original DAG are reduced to $11$ vertices. In the class~$1$ layer ($t \in \{1, 2, \ldots, 5\}$), the structure is unchanged, as there are no vertices at the same stage which can be merged.}
  \vspace{3.5cm}
\label{fig:af_compression}
\end{sidewaysfigure}

\subsection{The arc-flow ILP formulation}
\label{subsec:arcflow_ilp}

For each arc $a \in \mathcal{A}$ of the compressed DAG, let $\varphi_a \in \mathbb{Z}_{\ge 0}$ be an integer variable which represents the flow traversing arc~$a$. The arc-flow formulation for the \BPPS, which we refer to as~$\afone$, reads:
\begin{subequations}\label{form:arcflow}
\begin{alignat}{3}
  (\afone) \qquad
  \min_{\boldsymbol{\varphi}} \quad \sum_{a \in \mathcal{A}} \lambda_a\, \varphi_a \label{eq:af_obj} \\[1ex]
  \sum_{a \in \delta^-(v)} \varphi_a - \sum_{a \in \delta^+(v)} \varphi_a &\;=\; 0,
    & \qquad v \in \mathcal{V} \setminus \{\sigma,\, \tau\},
    \label{eq:af_flow} \\[1ex]
  \sum_{a \in \itemarcs} \varphi_a &\;=\; 1,
    & \qquad i \in \items,
    \label{eq:af_assignment} \\[1ex]
  \varphi_a &\in \mathbb{Z}_{\geq 0}, & \qquad a \in \mathcal{A}, \label{eq:af_int}
\end{alignat}
\end{subequations}
where $\delta^-(v)$ and $\delta^+(v)$ denote the sets of arcs entering and leaving vertex~$v$, respectively.
The objective function~\eqref{eq:af_obj} coincides with the total cost, which comprises the bin-opening costs and the class setup costs.
The \emph{flow conservation constraints}~\eqref{eq:af_flow} ensure that each unit of flow from $\sigma$ to $\tau$ represents one bin, and the total outflow from~$\sigma$ equals the number of bins used.
The \emph{assignment constraints}~\eqref{eq:af_assignment} require each item $i \in \items$ to be packed exactly once; these constraints are defined over the item arcs $\itemarcs$---which, by construction, are the only arcs encoding the packing of item~$i$.
Since $\mathcal{G}$ is acyclic, any feasible solution $\boldsymbol{\varphi}$ to $\afone$ can be decomposed into $\sum_{a\in\delta^+(\sigma)}\varphi_a$ source-to-sink paths. Each path represents one used bin and determines a feasible packing pattern $S\subseteq\items$ consisting of the items whose item arcs are traversed by the path. By the assignment constraints~\eqref{eq:af_assignment}, these packing patterns are pairwise disjoint and cover~$\items$; hence, they define a feasible partition $\mathcal{S}$ in the sense introduced in Section~\ref{sec:intro}.
Overall, $\afone$ involves $|\mathcal{A}|$ integer variables and $|\mathcal{V}| - 2 + |\items|$ linear constraints.
The size of the DAG is pseudo-polynomial in the input data.
Indeed, before compression, the DAG has one layer for each item and a linear number of additional layers for setup and class-transition decisions.
Within each layer, the cumulative load can take at most $d+1$ values, and the active flag has only a constant number of values.
Moreover, each state generates only a constant number of outgoing arcs.
Thus, the uncompressed DAG has a size of order $O((n+m)d)$, and the compression steps can only reduce this size.

\subsection{Strength of the LP relaxation of $\afone$}
\label{sec:afStrength}

After describing the arc-flow formulation, we now compare the strength of the LP relaxation of the arc-flow formulation with that of the natural formulation introduced in Section~\ref{sec:ILP}.
The next proposition formalizes that this relaxation is at least as strong as the compact one. The proof shows that every feasible solution of $\afoneLP$ has an objective function value at least equal to the optimal objective function value of $\foneLP$ (see \eqref{eq:LPoptval}); taking the minimum over all feasible solutions of $\afoneLP$ then yields the result. Related dominance results have also been established for the \BPP{} in the literature by exploiting the connection between arc-flow formulations and the Dantzig--Wolfe reformulation of the underlying compact model; see, e.g., ~\citet{deCarvalho1999}.
Let $\zeta(\afoneLP)$ denote the optimal objective function value of the LP relaxation of~$\afone$.

\begin{proposition}\label{prop:af_dominates_compact}
The LP relaxation of the arc-flow formulation~$\afone$ is at least as strong as the LP relaxation of the natural compact formulation~$\fone$, i.e.,
\[
  \zeta(\afoneLP) \ge \zeta(\foneLP).
\]
\end{proposition}

{\color{black}

\begin{proof}
Consider an arbitrary feasible solution $\boldsymbol\varphi$ of $\afoneLP$.
Since $\mathcal G$ is acyclic, any feasible flow on $\mathcal G$ can be decomposed into a collection $\mathcal Q$ of source-to-sink paths with nonnegative path-flow values $\xi_P$, $P\in\mathcal Q$. Each path $P$ is identified with the set of arcs it contains, so that $P\subseteq\mathcal A$. Hence,
\[
 \varphi_a=\sum_{P\in\mathcal Q:\,a\in P}\xi_P,
 \qquad
 \sum_{P\in\mathcal Q}\xi_P=\sum_{a\in\delta^-(\tau)}\varphi_a.
\]
For every $P\in\mathcal Q$, let $\items(P)$ and $\classes(P)$ be, respectively, the items and classes represented by its item and setup arcs. For each path $P \in \mathcal Q$, we have that $\sum_{i\in\items(P)}w_i+\sum_{c\in\classes(P)}s_c\le d$ by construction.
Multiplying this inequality by $\xi_P$, summing over $\mathcal Q$, and using the path decomposition arc by arc gives
\begin{equation}\label{proof_step:dominanceAF_red} \sum_{i\in\items}w_i\sum_{a\in\itemarcs}\varphi_a +\sum_{c\in\classes}s_c\sum_{a\in\classarcs}\varphi_a \le d\sum_{a\in\delta^-(\tau)}\varphi_a. \end{equation}

The assignment constraints give $\sum_{a\in\itemarcs}\varphi_a=1$ for every
$i\in\items$. Moreover, for each $c\in\classes$, the assignment of any item $i\in\items_c$ imply
$\sum_{a\in\classarcs}\varphi_a\ge1$. Thus, by~\eqref{proof_step:dominanceAF_red}
and $s_c\ge0$ for all $c\in\classes$, we have
\[
 \sum_{i\in\items}w_i+\sum_{c\in\classes}s_c
 \le d\sum_{a\in\delta^-(\tau)}\varphi_a.
\]
Only loss arcs entering $\tau$ and setup arcs have nonzero costs; therefore, since
$f_c\ge0$ for all $c\in\classes$, we have
\[
\sum_{a\in\mathcal A}\lambda_a\varphi_a = \bincost\sum_{a\in\delta^-(\tau)}\varphi_a + \sum_{c\in\classes}f_c\sum_{a\in\classarcs}\varphi_a \ge \frac{\bincost}{d}\left(\sum_{i\in\items}w_i+\sum_{c\in\classes}s_c\right) + \sum_{c\in\classes}f_c = \zeta(\foneLP),
\]
where the last equality follows from Proposition~\ref{prop:lprelax}. Taking the
minimum over the feasible solutions of $\afoneLP$ proves the result.
\end{proof}
}

The dominance can be strict, as illustrated by the two instances in the example of Figure~\ref{fig:1}. Specifically, in the instance of part~(a), $\zeta(\foneLP)=35$, while $\zeta(\afoneLP)=60$; similarly, in the instance of part~(b), $\zeta(\foneLP)=8$, while $\zeta(\afoneLP)=16$. This is also reflected in the computational results: at the price of a larger number of variables and constraints, the arc-flow formulation provides a consistently stronger LP relaxation than the natural compact formulation across the tested instances; see Section~\ref{sec:comp_LP}.

\subsection{Valid inequalities}
\label{sec:afValidIneq}

The valid inequalities introduced for the natural formulation can be transferred to $\afone$.
Their validity follows from the same arguments used in Sections~\ref{sec:MCI} and~\ref{sec:MBI}, since the corresponding arc-flow expressions count, respectively, the number of activations of each class and the number of used bins.
Specifically, the MCIs~\eqref{MCI} become
\begin{equation} \label{eq:arcflow_mcis}
    \sum_{a \in \classarcs} \varphi_a \;\geq\; \binnumberLBclass_c, \qquad c \in \classes,
\end{equation}
while the MBI translates to a lower bound on the total outflow from the source~$\sigma$:
\begin{equation} \label{eq:arcflow_mbi}
  \sum_{a \in \delta^+(\sigma)} \varphi_a \;\geq\;
  \left\lceil \frac{\sum_{i \in \items} w_i +
  \sum_{c \in \classes} \binnumberLBclass_c\, s_c}{d} \right\rceil.
\end{equation}

We denote by $\aftwo$ the formulation obtained by adding the MCIs~\eqref{eq:arcflow_mcis} to $\afone$, and by $\afthree$ the formulation obtained by adding both the MCIs~\eqref{eq:arcflow_mcis} and the MBI~\eqref{eq:arcflow_mbi} to $\afone$.
Their LP relaxations are denoted by $\aftwoLP$ and $\afthreeLP$, respectively, and their optimal objective function values by $\zeta(\aftwoLP)$ and $\zeta(\afthreeLP)$.
}

\section{Computational experiments} 
\label{sec:computationals}
In this section, we present the results of \Rev{the computational experiments carried out to evaluate} the performance of the ILP formulations for the \BPPS{} problem introduced in Sections~\ref{sec:ILP}~and~\ref{sec:arcflow}, along with the proposed enhancements: the MCIs (see~\eqref{MCI}~and~\eqref{eq:arcflow_mcis}), the MBI (see~\eqref{MBI}~and~\eqref{eq:arcflow_mbi}), and the upper bound on the number of bins that can be used in any optimal \BPPS{} solution, described in Section~\ref{sec:bin_upper_bound}.
We are interested in determining optimal \BPPS{} solutions by solving the ILP models~\eqref{form:kantorovich}~and~\eqref{form:arcflow} by means of a general-purpose ILP solver. When optimality cannot be certified within the imposed time limit, we assess solution quality using the optimality gap.
Our experiments pursue three main goals: {(i) assessing} the computational effectiveness of the proposed ILP formulations and their variants by identifying the largest instance sizes (and their structural features) that can be solved to optimality within a given time limit {and} determining the features that most {affect} problem difficulty{; (ii) evaluating} the relative performance of the different variants of the natural formulation $\fone$ \Rev{and the arc-flow formulation $\afone$ against each other, to identify the families of instances on which each model performs best (see Section~\ref{sec:performance})}; and (iii) assessing the strength of the LP relaxations of the two formulation families and its relation to their computational performance (see Section~\ref{sec:comp_LP}).

To the best of our knowledge, the \BPPS{} is a novel problem that has not been previously studied in the literature. Consequently, no existing models or algorithms are available for comparison with our proposed ILP formulations. Furthermore, no publicly available benchmark library exists. For this reason, we designed a diverse and systematic set of test instances that capture the main structural features of \BPPS{} instances, which we describe next.

\subsection{Library of benchmark BPPS instances}
\label{sec:instances}

In this section, we introduce the library of benchmark instances specifically designed for the \BPPS{}\Rev{, comprising both randomly generated instances and real-world instances from a vehicle-routing application}. This testbed aims to provide a diverse and representative collection of instances, varying in size and structural characteristics, enabling a comprehensive evaluation of the proposed ILP formulations and an assessment of how different instance features influence computational performance.

\Rev{We propose a set of randomly generated instances, obtained by extending} the approach proposed in \citet{SCHWERIN1997377} for the classical \BPP{}. We therefore adopted a methodology consistent with the literature on bin packing problems while incorporating item classes and their associated setup weights and costs. Each \BPPS{} instance is defined by a tuple~$(n, \boldsymbol{w}, d, m, \mathcal{P}, \boldsymbol{s}, \boldsymbol{f}, \bincost)$. While most of these input data can be freely chosen, item weights and class setup weights must be generated to ensure feasibility.
In line with \cite{SCHWERIN1997377}, we introduce parameters $\upsilon_1, \upsilon_2, \eta_1, \eta_2 \in [0,1]$, with $\upsilon_1 < \upsilon_2$ and $\eta_1 < \eta_2$, which determine the sampling ranges for item and setup weights, respectively. Specifically, for each $i \in \items$, the item weight is sampled uniformly from the integer values in the set $[\upsilon_1 \, d,\, \upsilon_2 \, d] \cap \mathbb{Z}$, and for each $c \in \classes$, its setup weight is sampled uniformly in the set $[\eta_1 \, d,\, \eta_2 \, d] \cap \mathbb{Z}$.

To ensure instance feasibility, we enforce the condition $\upsilon_2 + \eta_2 \leq 1$. Each item is assigned to one of the $m$ available classes, with the assignment chosen uniformly at random. Every instance
is generated either with or without setup costs. In the former case, setup costs are sampled uniformly in the set $[1,\, \bincost / 2] \cap \mathbb{Z}$ for each $c \in \classes$, with the bin cost set to $\bincost = 10$. In the latter, setup costs are set to zero, and the bin cost is set to $\bincost = 1$. This latter configuration minimizes the number of bins used, since no setup costs are incurred.
We generated BPPS instances for all combinations of the following instance-generator parameters: 
the number of items $n \in \{25,\, 50,\, 100,\, 200\}$, the number of classes $m \in \{5,\, 10\}$, 
the bin capacity $d \in \{200,\, 1000,\, 10000\}$, the cost structure (with or without setup costs and bin costs), 
\Rev{four} item size categories---\Rev{\textit{small} ($\upsilon_1 = 0.01$, $\upsilon_2 = 0.10$), \textit{medium} ($\upsilon_1 = 0.10$, $\upsilon_2 = 0.20$), \textit{large} ($\upsilon_1 = 0.20$, $\upsilon_2 = 0.30$) and \textit{mixed} ($\upsilon_1 = 0.01$, $\upsilon_2 = 0.30$)}---and \Rev{three} setup weight categories---\textit{small} 
($\eta_1 = 0.01$, $\eta_2 = 0.10$), \textit{large} ($\eta_1 = 0.10$, $\eta_2 = 0.20$) \Rev{and \textit{mixed} ($\eta_1 = 0.01$, $\eta_2 = 0.20$)}.
\Rev{One instance was generated} for each combination of these parameters, resulting in a total of \Rev{576} random instances.

{\color{black}
We also consider a set of 36 real-world instances arising from a vehicle-routing application in the distribution of fresh and frozen grocery products, based on data provided by a major international grocery retailer.
The operational setting involves transferring goods between distribution terminals using two types of vehicles:

\begin{itemize}
    \item[I)] a six-axle non-refrigerated multi-trailer vehicle with dimensions
          $(7.5\,\text{m}+7.5\,\text{m})\times 2.44\,\text{m}\times 3\,\text{m}$, capable of carrying up to $36$ Euro pallets.
    \item[II)] a six-axle refrigerated tractor--semitrailer whose internal loading compartment measures $13.6\,\text{m}\times 2.44\,\text{m}\times
2.6\,\text{m}$ (length $\times$ width $\times$ height) and accommodates up to $32$ Euro pallets in a single floor layer. When both product types are carried simultaneously, a movable bulkhead separates the two temperature
zones. Since the bulkhead can be repositioned between any two consecutive pallet rows, the vehicle can flexibly accommodate any mix of fresh and frozen products, subject to the overall pallet capacity.
\end{itemize}

Pallets cannot be stacked and are arranged in rows of two.  Type~I vehicles carry fresh
products only, while Type~II vehicles may carry both fresh and frozen products simultaneously. Using a Type~II vehicle incurs a higher cost than using a Type~I vehicle.
In this scenario, the daily set of customer orders is expressed in Euro-pallet units and defines the item set.  Each order is classified as either \emph{fresh} or \emph{frozen}, yielding two item classes ($m=2$): items in class $ c=1$ (fresh) can be loaded into either vehicle type.  Because fresh orders impose no refrigeration requirement, they do not consume extra capacity and incur no additional cost when placed in a bin.  Hence, the setup weight and setup cost for class $c=1$ are both zero ($s_1 = 0$, $f_1 = 0$). On the other hand, items from class $c=2$ (frozen) can be loaded only into a Type~II vehicle (a refrigerated bin).  Hence, when at least one frozen order is assigned to a vehicle, the effective pallet capacity decreases by $4$ (from $36$ to $32$, matching the Type~II specification), and an additional cost $\alpha \, \bincost$ is charged (we consider instances with $\alpha \in \{0.1, 0.3\}$. Thus $s_2 = 4$ and $f_2 = \alpha \bincost$. The bin capacity is set to $d = 36$ (the Type~I capacity) and the bin cost to $\bincost = 100$. Figure~\ref{fig:rw_illustration} provides an illustration of the two vehicle types.

\bigskip

\begin{figure}[h]
\centering
\resizebox{\textwidth}{!}{%
{\color{black}
\definecolor{freshcol}{RGB}{240,220,160}
\definecolor{frozencol}{RGB}{70,130,200}
\definecolor{bulkheadcol}{RGB}{200,30,30}
\definecolor{trailerbody}{RGB}{218,218,218}
\definecolor{traileredge}{RGB}{140,140,140}
\definecolor{cabwhite}{RGB}{248,248,248}
\definecolor{cabgrey}{RGB}{185,185,185}
\definecolor{cabdark}{RGB}{100,100,100}
\definecolor{wheelouter}{RGB}{45,45,45}
\definecolor{wheelrim}{RGB}{205,205,205}
\definecolor{wheelinner}{RGB}{150,150,150}
\definecolor{glasscol}{RGB}{195,220,238}
\definecolor{underchassis}{RGB}{95,95,95}

\newcommand{\wheelside}[3]{%
  \pgfmathsetmacro{\WCx}{#1}%
  \pgfmathsetmacro{\WCy}{#2}%
  \pgfmathsetmacro{\WR}{#3}%
  \pgfmathsetmacro{\shx}{\WCx - \WR*0.85}%
  \pgfmathsetmacro{\shy}{\WCy - \WR*0.92}%
  \fill[black!35,opacity=0.30]
    (\shx,\shy) ellipse ({\WR*0.75} and {\WR*0.11});
  \pgfmathsetmacro{\rox}{\WCx - 0.07}%
  \pgfmathsetmacro{\roy}{\WCy - 0.04}%
  \fill[wheelouter] (\rox,\roy) circle({\WR});
  \fill[wheelrim]   (\rox,\roy) circle({\WR*0.67});
  \fill[wheelinner] (\rox,\roy) circle({\WR*0.40});
  \fill[wheelouter] (\WCx,\WCy) circle({\WR});
  \fill[wheelrim]   (\WCx,\WCy) circle({\WR*0.67});
  \fill[wheelinner] (\WCx,\WCy) circle({\WR*0.40});
  \foreach \ang in {0,60,120,180,240,300}{%
    \pgfmathsetmacro{\bx}{\WCx + \WR*0.51*cos(\ang)}%
    \pgfmathsetmacro{\by}{\WCy + \WR*0.51*sin(\ang)}%
    \fill[black!75] (\bx,\by) circle({\WR*0.07});
  }%
  \fill[black!85] (\WCx,\WCy) circle({\WR*0.12});
}

\newcommand{\wheelsingle}[3]{%
  \pgfmathsetmacro{\WCx}{#1}%
  \pgfmathsetmacro{\WCy}{#2}%
  \pgfmathsetmacro{\WR}{#3}%
  \pgfmathsetmacro{\shx}{\WCx - \WR*0.85}%
  \pgfmathsetmacro{\shy}{\WCy - \WR*0.92}%
  \fill[black!35,opacity=0.30]
    (\shx,\shy) ellipse ({\WR*0.75} and {\WR*0.11});
  \fill[wheelouter] (\WCx,\WCy) circle({\WR});
  \fill[wheelrim]   (\WCx,\WCy) circle({\WR*0.67});
  \fill[wheelinner] (\WCx,\WCy) circle({\WR*0.40});
  \foreach \ang in {0,60,120,180,240,300}{%
    \pgfmathsetmacro{\bx}{\WCx + \WR*0.51*cos(\ang)}%
    \pgfmathsetmacro{\by}{\WCy + \WR*0.51*sin(\ang)}%
    \fill[black!75] (\bx,\by) circle({\WR*0.07});
  }%
  \fill[black!85] (\WCx,\WCy) circle({\WR*0.12});
}

\newcommand{\fender}[3]{%
  \pgfmathsetmacro{\WCx}{#1}%
  \pgfmathsetmacro{\WCy}{#2}%
  \pgfmathsetmacro{\WR}{#3}%
  \pgfmathsetmacro{\FR}{\WR*1.26}%
  \draw[traileredge, line width=1.3pt, line cap=round]
    (\WCx,\WCy) ++({196}:{\FR}) arc[start angle=196, end angle=-16, radius={\FR}];
}

\newcommand{\cabTypeII}[2]{%
  \pgfmathsetmacro{\RX}{#1}%
  \pgfmathsetmacro{\BY}{#2}%
  \pgfmathsetmacro{\chL}{\RX-2.60}%
  \fill[underchassis] (\chL,\BY) rectangle (\RX,{\BY+0.26});
  \fill[cabwhite]
    ({\RX-2.62},{\BY+0.26})--
    ({\RX-2.62},{\BY+1.08})--
    ({\RX-2.38},{\BY+1.22})--
    ({\RX-0.10},{\BY+1.22})--
    ({\RX-0.10},{\BY+0.26})--cycle;
  \draw[cabdark,line width=0.55pt]
    ({\RX-2.62},{\BY+0.26})--
    ({\RX-2.62},{\BY+1.08})--
    ({\RX-2.38},{\BY+1.22})--
    ({\RX-0.10},{\BY+1.22})--
    ({\RX-0.10},{\BY+0.26})--cycle;
  \fill[cabgrey]
    ({\RX-2.74},{\BY+0.26}) rectangle ({\RX-2.58},{\BY+0.80});
  \draw[cabdark,line width=0.5pt]
    ({\RX-2.74},{\BY+0.26}) rectangle ({\RX-2.58},{\BY+0.80});
  \fill[yellow!55!white, draw=cabdark, line width=0.35pt]
    ({\RX-2.70},{\BY+0.82}) rectangle ({\RX-2.50},{\BY+0.98});
  \fill[cabwhite]
    ({\RX-1.82},{\BY+1.22})--
    ({\RX-1.82},{\BY+3.16})--
    ({\RX-0.62},{\BY+3.38})--
    ({\RX-0.10},{\BY+3.28})--
    ({\RX-0.10},{\BY+1.22})--cycle;
  \draw[cabdark,line width=0.55pt]
    ({\RX-1.82},{\BY+1.22})--
    ({\RX-1.82},{\BY+3.16})--
    ({\RX-0.62},{\BY+3.38})--
    ({\RX-0.10},{\BY+3.28})--
    ({\RX-0.10},{\BY+1.22})--cycle;
  \fill[cabwhite]
    ({\RX-2.38},{\BY+1.22})--
    ({\RX-1.82},{\BY+1.22})--
    ({\RX-1.82},{\BY+2.25})--
    ({\RX-2.14},{\BY+2.25})--cycle;
  \draw[cabdark,line width=0.45pt]
    ({\RX-2.38},{\BY+1.22})--
    ({\RX-1.82},{\BY+1.22})--
    ({\RX-1.82},{\BY+2.25})--
    ({\RX-2.14},{\BY+2.25})--cycle;
  \fill[glasscol,opacity=0.80]
    ({\RX-1.77},{\BY+1.65}) rectangle ({\RX-1.1},{\BY+2.98});
  \fill[glasscol,opacity=0.80]
  ({\RX-1},{\BY+2.15}) rectangle ({\RX-0.18},{\BY+2.98});
  \draw[cabdark,line width=0.35pt]
    ({\RX-1.77},{\BY+1.65}) rectangle ({\RX-0.18},{\BY+2.98});
  \draw[cabdark,line width=0.35pt]
    ({\RX-1},{\BY+2.15}) rectangle ({\RX-0.18},{\BY+2.15});
  \draw[cabdark,line width=0.5pt]
    ({\RX-1.02},{\BY+1.65})--({\RX-1.02},{\BY+2.98});
  \draw[cabdark,line width=0.4pt]
    ({\RX-1.07},{\BY+1.22})--({\RX-1.07},{\BY+3.30});
  \fill[cabdark]
    ({\RX-0.55},{\BY+1.85}) rectangle ({\RX-0.34},{\BY+1.92});
  \fill[cabgrey]
    ({\RX-1.82},{\BY+3.16})--
    ({\RX-0.90},{\BY+3.38})--
    ({\RX-0.90},{\BY+3.62})--
    ({\RX-1.82},{\BY+3.48})--cycle;
  \draw[cabdark,line width=0.4pt]
    ({\RX-1.82},{\BY+3.16})--
    ({\RX-0.90},{\BY+3.38})--
    ({\RX-0.90},{\BY+3.62})--
    ({\RX-1.82},{\BY+3.48})--cycle;
  \fill[cabgrey,draw=cabdark,line width=0.3pt]
    ({\RX-0.23},{\BY+1.22}) rectangle ({\RX-0.12},{\BY+3.65});
}
\newcommand{\cabTypeI}[2]{%
  \pgfmathsetmacro{\RX}{#1}%
  \pgfmathsetmacro{\BY}{#2}%
  \fill[underchassis] ({\RX-2.45},\BY) rectangle (\RX,{\BY+0.24});
  \fill[cabwhite]
    ({\RX-2.47},{\BY+0.24})--
    ({\RX-2.47},{\BY+1.00})--
    ({\RX-2.24},{\BY+1.14})--
    ({\RX-0.10},{\BY+1.14})--
    ({\RX-0.10},{\BY+0.24})--cycle;
  \draw[cabdark,line width=0.55pt]
    ({\RX-2.47},{\BY+0.24})--
    ({\RX-2.47},{\BY+1.00})--
    ({\RX-2.24},{\BY+1.14})--
    ({\RX-0.10},{\BY+1.14})--
    ({\RX-0.10},{\BY+0.24})--cycle;
  \fill[cabgrey]
    ({\RX-2.58},{\BY+0.24}) rectangle ({\RX-2.43},{\BY+0.72});
  \draw[cabdark,line width=0.5pt]
    ({\RX-2.58},{\BY+0.24}) rectangle ({\RX-2.43},{\BY+0.72});
  \foreach \gx in {0.06,0.12,0.18,0.24}{
    \pgfmathsetmacro{\gxx}{\RX-2.57+\gx}%
    \draw[cabdark,line width=0.25pt]
      (\gxx,{\BY+0.27})--(\gxx,{\BY+0.68});
  }%
  \fill[yellow!55!white,draw=cabdark,line width=0.35pt]
    ({\RX-2.55},{\BY+0.75}) rectangle ({\RX-2.36},{\BY+0.90});
  \foreach \vx in {0.35,0.65,0.95,1.25}{
    \pgfmathsetmacro{\vxx}{\RX-2.22+\vx}%
    \draw[cabdark,line width=0.25pt,opacity=0.45]
      (\vxx,{\BY+0.64})--(\vxx,{\BY+1.10});
  }%
  \fill[cabwhite]
    ({\RX-1.68},{\BY+1.14})--
    ({\RX-1.68},{\BY+2.96})--
    ({\RX-0.58},{\BY+3.16})--
    ({\RX-0.10},{\BY+3.06})--
    ({\RX-0.10},{\BY+1.14})--cycle;
  \draw[cabdark,line width=0.55pt]
    ({\RX-1.68},{\BY+1.14})--
    ({\RX-1.68},{\BY+2.96})--
    ({\RX-0.58},{\BY+3.16})--
    ({\RX-0.10},{\BY+3.06})--
    ({\RX-0.10},{\BY+1.14})--cycle;
  \fill[cabwhite]
    ({\RX-2.24},{\BY+1.14})--
    ({\RX-1.68},{\BY+1.14})--
    ({\RX-1.68},{\BY+2.10})--
    ({\RX-2.00},{\BY+2.10})--cycle;
  \draw[cabdark,line width=0.45pt]
    ({\RX-2.24},{\BY+1.14})--
    ({\RX-1.68},{\BY+1.14})--
    ({\RX-1.68},{\BY+2.10})--
    ({\RX-2.00},{\BY+2.10})--cycle;
  \fill[glasscol,opacity=0.80]
    ({\RX-1.98},{\BY+1.19})--
    ({\RX-1.70},{\BY+1.19})--
    ({\RX-1.70},{\BY+2.05})--
    ({\RX-1.96},{\BY+2.05})--cycle;
  \draw[cabdark,line width=0.35pt]
    ({\RX-1.98},{\BY+1.19})--
    ({\RX-1.70},{\BY+1.19})--
    ({\RX-1.70},{\BY+2.05})--
    ({\RX-1.96},{\BY+2.05})--cycle;
  \fill[glasscol,opacity=0.80]
    ({\RX-1.63},{\BY+1.55}) rectangle ({\RX-0.18},{\BY+2.82});
  \draw[cabdark,line width=0.35pt]
    ({\RX-1.63},{\BY+1.55}) rectangle ({\RX-0.18},{\BY+2.82});
  \draw[cabdark,line width=0.5pt]
    ({\RX-0.95},{\BY+1.55})--({\RX-0.95},{\BY+2.82});
  \fill[cabdark]
    ({\RX-0.50},{\BY+1.75}) rectangle ({\RX-0.30},{\BY+1.82});
  \draw[cabdark,line width=0.4pt]
    ({\RX-1.00},{\BY+1.14})--({\RX-1.00},{\BY+3.08});
  \fill[cabgrey]
    ({\RX-1.68},{\BY+2.96})--
    ({\RX-0.82},{\BY+3.16})--
    ({\RX-0.82},{\BY+3.38})--
    ({\RX-1.68},{\BY+3.22})--cycle;
  \draw[cabdark,line width=0.4pt]
    ({\RX-1.68},{\BY+2.96})--
    ({\RX-0.82},{\BY+3.16})--
    ({\RX-0.82},{\BY+3.38})--
    ({\RX-1.68},{\BY+3.22})--cycle;
  \fill[cabgrey,draw=cabdark,line width=0.3pt]
    ({\RX-1.42},{\BY+0.24}) rectangle ({\RX-0.30},{\BY+0.44});
  \fill[cabgrey,draw=cabdark,line width=0.3pt]
    ({\RX-1.42},{\BY+0.44}) rectangle ({\RX-0.30},{\BY+0.62});
  \fill[cabgrey,draw=cabdark,line width=0.3pt]
    ({\RX-0.22},{\BY+1.14}) rectangle ({\RX-0.11},{\BY+3.40});
  \fill[cabgrey,draw=cabdark,line width=0.4pt,rounded corners=1.5pt]
    ({\RX-1.35},{\BY+0.50}) rectangle ({\RX-0.33},{\BY+1.02});
}

\newcommand{\sectionblock}[3]{%
  \pgfmathtruncatemacro{\Nlast}{#2-1}%
  \foreach \i in {0,...,\Nlast}{%
    \pgfmathsetmacro{\xp}{#1 + \i*0.80}%
    \fill[#3,draw=white,line width=0.5pt] (\xp,0.13) rectangle ++(0.72,0.68);
    \fill[#3,draw=white,line width=0.5pt] (\xp,0.87) rectangle ++(0.72,0.68);
  }%
}

\newcommand{\tractortop}[1]{%
  \pgfmathsetmacro{\fx}{#1}%
  \fill[cabdark] ({\fx-0.18},0.66) rectangle (\fx,1.06);            
  \fill[cabgrey,draw=cabdark,line width=0.5pt,rounded corners=2pt]
        ({\fx-2.05},0.16) rectangle ({\fx-0.18},1.56);             
  \fill[glasscol,opacity=0.85] ({\fx-2.00},0.24) rectangle ({\fx-1.80},1.48);
  \draw[cabdark,line width=0.3pt] ({\fx-2.00},0.24) rectangle ({\fx-1.80},1.48);
  \draw[cabdark,line width=0.3pt] ({\fx-1.55},0.16)--({\fx-1.55},1.56);
  \draw[cabdark,line width=0.3pt] ({\fx-0.70},0.16)--({\fx-0.70},1.56);
  \fill[black!55] ({\fx-0.35},0.86) circle(0.10);                  
}

\begin{tikzpicture}[x=1cm,y=1cm,font=\small]

\begin{scope}[]
  \fill[trailerbody] (2.65,0.00) rectangle (10.03,1.72);
  \draw[cabdark,line width=0.8pt] (2.65,0.00) rectangle (10.03,1.72);
  \fill[traileredge] (2.65,0.00) rectangle (10.03,0.10);
  \fill[traileredge] (2.65,1.62) rectangle (10.03,1.72);
  \fill[trailerbody] (10.45,0.00) rectangle (17.83,1.72);
  \draw[cabdark,line width=0.8pt] (10.45,0.00) rectangle (17.83,1.72);
  \fill[traileredge] (10.45,0.00) rectangle (17.83,0.10);
  \fill[traileredge] (10.45,1.62) rectangle (17.83,1.72);
  \fill[cabdark] (10.03,0.58) rectangle (10.45,1.14);          
  \sectionblock{2.76}{9}{freshcol}
  \sectionblock{10.56}{9}{freshcol}
  \draw[<->,cabdark,thin] (2.76,-0.45)--(9.96,-0.45)
     node[midway,below,font=\scriptsize]{9 sections -- 18 pallets (fresh)};
  \draw[<->,cabdark,thin] (10.56,-0.45)--(17.76,-0.45)
     node[midway,below,font=\scriptsize]{9 sections -- 18 pallets (fresh)};
  \node[right,font=\footnotesize,align=left] at (18.05,0.86)
     {\textbf{Type I} \textit{(top view)}};
\end{scope}

\begin{scope}[yshift=-5cm]
  \fill[trailerbody] (2.80,0.00) rectangle (15.75,1.72);
  \draw[cabdark,line width=0.8pt] (2.80,0.00) rectangle (15.75,1.72);
  \fill[traileredge] (2.80,0.00) rectangle (15.75,0.10);
  \fill[traileredge] (2.80,1.62) rectangle (15.75,1.72);
  \sectionblock{2.9}{10}{frozencol}
  \sectionblock{10.94}{6}{freshcol}
  \fill[bulkheadcol] (10.84,0.06) rectangle (10.90,1.66);
  \draw[->,cabdark,thin] (10.87,2.3)--(10.87,1.75);
  \node[above,font=\scriptsize] at (10.87,2.35) {bulkhead};
  \draw[<->,cabdark,thin] (2.91,-0.45)--(10.83,-0.45)
     node[midway,below,font=\scriptsize]{10 sections -- 20 pallets (frozen)};
  \draw[<->,cabdark,thin] (10.91,-0.45)--(15.63,-0.45)
     node[midway,below,font=\scriptsize]{6 sections -- 12 pallets (fresh)};
  \node[right,font=\footnotesize,align=left] at (16.45,0.86)
     {\textbf{Type II} \textit{(top view)}};
\end{scope}

\end{tikzpicture}
}%
}
\smallskip
\caption{
    Schematic representation of the two vehicle types used in the real-world \BPPS{} instances.
}
\label{fig:rw_illustration}
\end{figure}

The real-world instances are organized into three groups of $12$ instances: \texttt{rwA}, \texttt{rwB}, \texttt{rwC}. The \texttt{rwA} and \texttt{rwB} instances share the same item-class partition (i.e., the same number of fresh and frozen orders for every instance). They differ in the minimum pallet demand per customer order: specifically, every fresh and frozen order in \texttt{rwA} contains at least $12$ pallets, whereas in \texttt{rwB} the minimum is $14$ pallets for fresh orders and $10$ pallets for frozen orders. In \texttt{rwC}, some pallet requirements are fractional. To express all BPPS input data as integers, pallet quantities are multiplied by $10$. Thus, $d=360$ represents a bin capacity of $36$ pallets, $s_2=40$ represents a setup weight of $4$ pallets, and the minimum item weight of $112$ represents an order of $11.2$ pallets.
}

The complete benchmark set and its documentation, along with the random instance generator toolkit, are publicly available at the following GitHub repository: \url{https://github.com/FabioCiccarelli/BPPS}. \Rev{The repository also contains the implementation of the mathematical formulations introduced in this paper, as well as the algorithms for constructing and compressing the arc-flow DAG described in Section~\ref{sec:arcflow}.} We hope that the resources made available through this repository will foster further research on the \BPPS{} by providing a common reference framework for developing and comparing novel exact and heuristic algorithms, mathematical formulations, computational studies, and polyhedral analyses.

\subsection{Computational performance of the ILP formulations}
\label{sec:performance}

In this section, we compare the computational performance of the following \Rev{seven} variants of the proposed ILP formulations~\eqref{form:kantorovich}~and~\eqref{form:arcflow} for the BPPS. \Rev{The base variant of the natural formulation}, denoted by $\fone$, is formulation~\eqref{form:kantorovich} in which the upper bound $\binnumberUB$ on the number of bins in any optimal BPPS solution is set to $n$, namely the number of items. The first variant \Rev{of the natural formulation}, denoted by $\ftwo$, extends $\fone$ by including the MCIs~\eqref{MCI}. The second variant, denoted by $\ftwobis$, extends $\fone$ by including the MCIs~\eqref{MCI} and the MBI~\eqref{MBI}. The third variant \Rev{of the natural formulation}, denoted by $\fthree$, extends $\ftwobis$ by including the upper bound $\hat{\binnumberUB}$ on the number of bins in any optimal BPPS solution (see Proposition~\ref{upper_bound_bins}); the values $\overline{\beta}_c$, for each class $c \in \classes$, used to compute $\hat{\binnumberUB}$ were obtained \Rev{using} classical heuristic algorithms for the BPP\Rev{, specifically} the Next Fit (NF), First Fit (FF), and Best Fit (BF) algorithms (see \cite{Johnson1974WorstCasePB}), each executed under 50 random permutations of the item order (including the non-increasing order of weights). \Rev{The base variant of the arc-flow formulation is denoted by $\afone$ and corresponds to Model~\eqref{form:arcflow}. The first variant of the arc-flow formulation, denoted by $\aftwo$, extends $\afone$ by including the MCIs~\eqref{eq:arcflow_mcis}. The second variant of the arc-flow formulation, denoted by $\afthree$, extends $\afone$ by including the MCIs~\eqref{eq:arcflow_mcis} and the MBI~\eqref{eq:arcflow_mbi}.
In all seven variants, the solver is provided with a warm-start feasible solution constructed as follows. For each class $c \in \classes$, the items of class $c$ are packed into dedicated bins using the heuristic solution computed to derive $\overline{\beta}_c$. The resulting solution, which features single-class bins only, is passed to the solver as an initial solution.}

\Rev{Concerning the arc-flow formulation, it is interesting to assess the effectiveness of the compression phase described in Section~\ref{sec:afCompression} in reducing the size of the arc-flow DAGs. Across the 576 random instances, it removes a median of $71.0\%$ of the vertices and $67.5\%$ of the arcs. Although the extent of the reduction varies considerably across instances, it remains substantial in all cases, ranging from $43.6\%$ to $97.8\%$ for vertices and from $42.1\%$ to $96.9\%$ for arcs.}

All tests were conducted on an Ubuntu machine equipped with an Intel(R) Xeon(R) Gold 5218 CPU @ 2.30GHz and 256~GB of RAM. The ILP formulations are 
solved using \Rev{\textit{Gurobi 13.0.2}}, running in single-thread mode with default settings. A time limit of 1800 seconds is imposed for each test.

Table~\ref{tab:exact_1} reports, for each of the \Rev{seven} ILP formulation variants presented above, the number of instances solved to proven optimality within a time limit of 1800 seconds (columns ``opt'') and the average optimality gap (columns ``gap'') over the instances not solved to optimality. The optimality gap is defined as the percentage difference between the best upper bound and the best lower bound computed by the solver during the solution process for a given formulation, taken with respect to the best upper bound.
\Rev{For the three arc-flow variants, the table additionally reports, for each instance group, the number of instances for which no valid dual bound could be computed within the time limit (column ``n.c.''); these instances are excluded from the average gap computation for the corresponding variant.} The random instances are grouped according to the six instance-generator parameters described in the previous section on the benchmark instance library, \Rev{while the real-world instances are grouped according to the subset they belong to (\texttt{rwA}, \texttt{rwB} or \texttt{rwC}).} For each group, the table
also reports the number of instances in that group (column ``inst'').

\Rev{The results in Table~\ref{tab:exact_1} reveal distinct patterns depending on the instance characteristics, which we discuss separately for the natural and arc-flow formulation families before comparing them directly.}
\Rev{Within the natural formulation family,} performance differences become more pronounced as the instance size or complexity increases. For small instances ($n = 25$), all four variants perform similarly. However, as $n$ grows, $\fone$ rapidly loses
effectiveness, with a drop in the number of optimally solved instances and a substantial increase in the gap. The addition of the MCI constraints in $\ftwo$ yields a clear improvement; the further inclusion of the MBI in $\ftwobis$ allows a larger number of instances to be solved to optimality. The combination of these valid inequalities
with the tighter upper bound on the number of bins in $\fthree$ provides the best results among the natural formulation variants for most of the instance groups, both in terms of optimal solutions and reduced gap. A similar trend holds when varying the number of classes $m$, the bin
capacity $d$, the setup cost structure, and the setup and item weight distributions, with $\fthree$ matching or outperforming $\ftwobis$ across most of the groups.
\Rev{Within the arc-flow family, $\afthree$ is consistently the strongest variant. A notable limitation, however, is the significant number of instances for which no valid dual bound could be computed, particularly for larger values of $n$ (up to 66 instances out of 144 for $n = 200$) and larger bin capacities (up to 71 out of 192 for $d = 10,000$). \\
A further comparison can be done between the two formulations introduced in this article, namely the natural model and the arc-flow model. The item weight distribution emerges as the dominant factor: for instances with large item weights, the arc-flow variants are substantially more effective, with $\afthree$ solving 132 out of 144 instances compared to 34 for $\fthree$, and with an average optimality gap of $1.3\%$ against $6.7\%$. Conversely, for small and medium item weights, the natural formulation variants are definitely superior, and the arc-flow formulation suffers from a higher rate of instances marked as "n.c.". This difference explains the behavior on the real-world instances, which feature large item weights: the arc-flow variant $\afthree$ solves all 36 instances to optimality, while the best natural formulation variant solves only 8. } 

\Rev{Overall, across the 576 random instances, 503 ($87.3\%$) are solved to proven optimality by at least one of the seven configurations, and all 36 real-world instances are solved as well. The remaining 73 random instances ($12.7\%$) are not solved by any configuration within the time limit. The minimum optimality gap on these instances ranges from $0.27\%$ to $6.67\%$ (average $2.85\%$, median $2.78\%$), so the solver remains close to optimality even on these harder instances. Difficulty is influenced almost entirely by the values of $n$ and $m$: no instance with $n \le 50$ remains unsolved, while the unsolved rate reaches $44.4\%$ for $n=200,\,m=10$ and $36.1\%$ for $n=100,\,m=10$ and bin-setup costs, reflecting the higher complexity induced by many classes with many items each. Two secondary factors further contribute to instance difficulty: adding the bin-setup costs raises the unsolved rate from $9.7\%$ to $15.6\%$, and instances with medium or mixed item-weight ranges are harder ($20.8\%$ and $20.1\%$ unsolved, respectively) than small or large ones ($2.1\%$ and $7.6\%$). No single formulation consistently provides better bounds on the open instances: $\ftwo$ yields the best optimality gap in 23 of the 73 cases and $\afthree$ in 18, confirming the complementary behavior of the natural and arc-flow models discussed above.}

\begin{table}[h]
\scriptsize
\centering
\tabcolsep 2pt
\renewcommand
\arraystretch{1.35}

\caption{Performance comparison of the seven ILP formulation variants for the BPPS, reporting the number of instances solved to optimality (opt), the average optimality gap over unsolved instances (gap) \Rev{and, for the arc-flow variants, the number of instances for which no valid dual bound could be computed within the time limit (n.c.).} The instances are grouped by instance generator parameters. Time limit: 1800 seconds.}

\label{tab:exact_1}

{\color{black}
\begin{tabular}{llrrrrrrrrrrrrrrrrrrrrrrrrrrr}
\toprule
                                   &             &      &  &  & \multicolumn{11}{r}{Natural}                                                                                                        &  &  & \multicolumn{11}{r}{Arc-flow}                                                                      \\ \cline{6-16} \cline{19-21} \cline{23-25} \cline{27-29}
                                   &             &      &  &  & \multicolumn{2}{r}{$\fone$} &  & \multicolumn{2}{r}{$\ftwo$} &  & \multicolumn{2}{r}{$\ftwobis$} &  & \multicolumn{2}{r}{$\fthree$} &  &  & \multicolumn{3}{r}{$\afone$} &  & \multicolumn{3}{r}{$\aftwo$} &  & \multicolumn{3}{r}{$\afthree$} \\ \cline{6-7} \cline{9-10} \cline{12-13} \cline{15-16} \cline{19-21} \cline{23-25} \cline{27-29}
                              &             & inst &  &  & opt          & gap          &  & opt          & gap          &  & opt                & gap       &  & opt               & gap       &  &  & opt           & gap   & n.c. &  & opt     & gap     & n.c.     &  & opt            & gap   & n.c.  \\ \cline{1-3} \cline{6-7} \cline{9-10} \cline{12-13} \cline{15-16} \cline{19-21} \cline{23-25} \cline{27-29}
                                   &             &      &  &  &              &              &  &              &              &  &                    &           &  &                   &           &  &  &               &       &      &  &         &         &          &  &                &       &       \\[-3ex]
\tt item number & $n$ = 25 & 144 &  &  & 137 & 10.7 &  & 136 & 10.4 &  & 137 & 11.1 &  & 137 & 11.1 &  &  & \textbf{144} & - & 0 &  & 143 & 3.9 & 0 &  & \textbf{144} & - & 0 \\
 & $n$ = 50 & 144 &  &  & 83 & 8.3 &  & 99 & 7.2 &  & 107 & 6.8 &  & 107 & 6.7 &  &  & 108 & 4.4 & 9 &  & 119 & 4.5 & 9 &  & \textbf{128} & 7.1 & 9 \\
 & $n$ = 100 & 144 &  &  & 43 & 11.2 &  & 65 & 4.8 &  & 77 & 4.9 &  & \textbf{79} & 5.0 &  &  & 52 & 6.3 & 24 &  & 57 & 3.9 & 36 &  & 69 & 3.7 & 36 \\
 & $n$ = 200 & 144 &  &  & 18 & 16.2 &  & 54 & 4.1 &  & \textbf{64} & 4.4 &  & 63 & 4.5 &  &  & 27 & 5.5 & 45 &  & 25 & 4.6 & 66 &  & 38 & 4.8 & 64 \\[0.75ex]
\tt class number & $m$ = 5 & 288 &  &  & 150 & 11.8 &  & 200 & 5.2 &  & \textbf{214} & 5.3 &  & \textbf{214} & 5.5 &  &  & 161 & 5.4 & 40 &  & 170 & 4.0 & 57 &  & 200 & 4.2 & 56 \\
 & $m$ = 10 & 288 &  &  & 131 & 13.5 &  & 154 & 5.2 &  & 171 & 5.3 &  & 172 & 5.2 &  &  & 170 & 6.0 & 38 &  & 174 & 4.5 & 54 &  & \textbf{179} & 4.6 & 53 \\[0.75ex]
\tt bin capacity & $d$ = 200 & 192 &  &  & 93 & 13.1 &  & 125 & 5.0 &  & 135 & 4.9 &  & 134 & 5.1 &  &  & 128 & 6.1 & 0 &  & 131 & 4.9 & 0 &  & \textbf{148} & 5.0 & 0 \\
 & $d$ = 1,000 & 192 &  &  & 95 & 12.9 &  & 120 & 5.3 &  & 127 & 5.4 &  & \textbf{129} & 5.4 &  &  & 106 & 5.9 & 18 &  & 111 & 4.0 & 40 &  & 122 & 4.3 & 38 \\
 & $d$ = 10,000 & 192 &  &  & 93 & 12.1 &  & 109 & 5.4 &  & \textbf{123} & 5.5 &  & \textbf{123} & 5.5 &  &  & 97 & 4.4 & 60 &  & 102 & 2.8 & 71 &  & 109 & 2.7 & 71 \\[0.75ex]
\tt bin-setup costs & no & 288 &  &  & 150 & 8.3 &  & 184 & 6.2 &  & 190 & 6.1 &  & 189 & 6.1 &  &  & \textbf{208} & 6.1 & 42 &  & 207 & 5.8 & 56 &  & 200 & 6.3 & 55 \\
 & yes & 288 &  &  & 131 & 16.6 &  & 170 & 4.3 &  & 195 & 4.4 &  & \textbf{197} & 4.4 &  &  & 123 & 5.5 & 36 &  & 137 & 3.9 & 55 &  & 179 & 3.4 & 54 \\[0.75ex]
\tt item weights & small & 144 &  &  & 124 & 11.2 &  & 139 & 6.5 &  & 140 & 6.0 &  & \textbf{141} & 5.2 &  &  & 70 & 11.7 & 50 &  & 68 & 8.8 & 55 &  & 72 & 10.3 & 55 \\
 & medium & 144 &  &  & 63 & 11.4 &  & 90 & 3.5 &  & 102 & 3.4 &  & \textbf{105} & 3.4 &  &  & 78 & 5.6 & 8 &  & 82 & 4.1 & 22 &  & 93 & 3.9 & 21 \\
 & large & 144 &  &  & 30 & 14.3 &  & 30 & 6.8 &  & 35 & 6.5 &  & 34 & 6.7 &  &  & 110 & 1.4 & 0 &  & 118 & 1.5 & 0 &  & \textbf{132} & 1.3 & 0 \\
 & mixed & 144 &  &  & 64 & 12.1 &  & 95 & 3.4 &  & \textbf{108} & 3.6 &  & 106 & 3.3 &  &  & 73 & 5.7 & 20 &  & 76 & 3.9 & 34 &  & 82 & 2.9 & 33 \\[0.75ex]
\tt setup weights & small & 192 &  &  & 98 & 11.4 &  & 123 & 5.3 &  & \textbf{129} & 5.4 &  & \textbf{129} & 5.6 &  &  & 104 & 5.6 & 32 &  & 109 & 4.9 & 44 &  & 118 & 4.3 & 43 \\
 & large & 192 &  &  & 87 & 14.0 &  & 114 & 5.2 &  & 124 & 5.1 &  & 127 & 5.1 &  &  & 121 & 5.7 & 17 &  & 123 & 3.8 & 29 &  & \textbf{134} & 4.5 & 31 \\
 & mixed & 192 &  &  & 96 & 12.6 &  & 117 & 5.2 &  & \textbf{132} & 5.4 &  & 130 & 5.3 &  &  & 106 & 5.7 & 29 &  & 112 & 4.1 & 38 &  & 127 & 4.6 & 35 \\[0.75ex] \cline{1-3} \cline{6-7} \cline{9-10} \cline{12-13} \cline{15-16} \cline{19-21} \cline{23-25} \cline{27-29} 
Total/Average &  & 576 &  &  & 281 & 12.7 &  & 354 & 5.2 &  & 385 & 5.3 &  & \textbf{386} & 5.3 &  &  & 331 & 5.7 & 78 &  & 344 & 4.3 & 111 &  & 379 & 4.5 & 109 \\ 
\hline
\\[-3ex]
\tt real-world & \texttt{rwA} & 12 &  &  & 2 & 10.2 &  & 2 & 5.7 &  & 1 & 7.5 &  & 4 & 6.5 &  &  & 6 & 0.3 & 0 &  & 11 & 0.4 & 0 &  & \textbf{12} & - & 0 \\
 & \texttt{rwB} & 12 &  &  & 2 & 7.1 &  & 1 & 9.4 &  & 1 & 5.1 &  & 4 & 1.1 &  &  & \textbf{12} & - & 0 &  & \textbf{12} & - & 0 &  & \textbf{12} & - & 0 \\
 & \texttt{rwC} & 12 &  &  & 0 & 7.0 &  & 1 & 7.5 &  & 1 & 7.1 &  & 0 & 7.7 &  &  & 11 & 0.2 & 0 &  & \textbf{12} & - & 0 &  & \textbf{12} & - & 0 \\[0.75ex]
 \cline{1-3} \cline{6-7} \cline{9-10} \cline{12-13} \cline{15-16} \cline{19-21} \cline{23-25} \cline{27-29} 
Total/Average &  & 36 &  &  & 4 & 8.0 &  & 4 & 7.6 &  & 3 & 6.6 &  & 8 & 5.5 &  &  & 29 & 0.3 & 0 &  & 35 & 0.4 & 0 &  & \textbf{36} & - & 0 \\
\bottomrule
\end{tabular}
}
\end{table}

Figure~\ref{fig:survival_plots} plots the number of random instances solved as a function of time for the \Rev{seven} ILP formulation variants. The $x$-axis represents the computing time (in seconds, on a logarithmic scale), and the $y$-axis reports the cumulative number of instances solved to optimality within that time.
\Rev{Among the natural formulation variants, the ranking is consistent throughout almost the entire time horizon: $\fthree$ dominates, followed by $\ftwobis$, $\ftwo$, and $\fone$. The gap between $\fthree$ and $\fone$ is substantial and emerges within the first few seconds, showing that the proposed strengthening techniques are also effective on the easier instances. Within the arc-flow family, $\afthree$ similarly dominates $\aftwo$ and $\afone$ throughout the time horizon, although the separation among the three variants is less pronounced than within the natural family. Comparing the two formulation families on the random instances, $\fthree$ and $\ftwobis$ achieve the best performance in the early phase, below approximately one second, solving more instances than any arc-flow variant. Specifically, within the first $0.1$ seconds, $\fthree$ solves around $150$ instances, compared with approximately $50$ for $\afthree$. As time increases, the arc-flow variants narrow the gap, and $\afthree$ ultimately solves almost as many instances as $\fthree$ by the end of the time horizon. These profiles therefore show that, on the random benchmark, the strongest natural formulation variants are more effective on instances solved within short computing times, whereas the performance of $\fthree$ and $\afthree$ becomes comparable over longer time horizons. This conclusion does not extend to all instance groups: as shown in Table~\ref{tab:exact_1}, the relative performance of the two formulation families depends mainly on the item weight distribution, with the arc-flow variants outperforming the natural ones on instances with large item weights and on the real-world testbed.}

\begin{figure}[h]
\centering
    \includegraphics[width=0.9\linewidth]{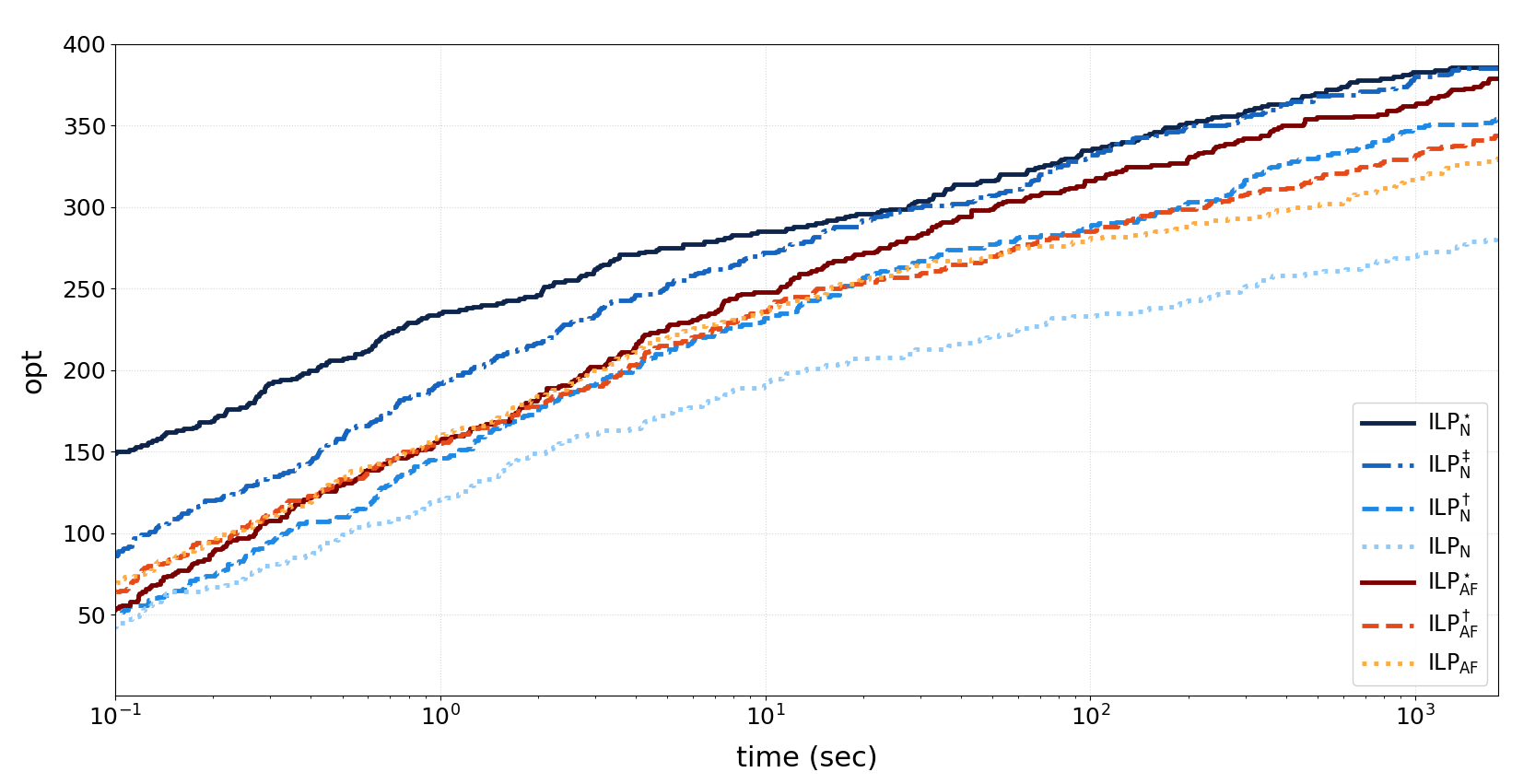}   
\caption{Number of random instances solved as a function of time for the different variants of the ILP formulations for the BPPS. Time limit: 1800 seconds. The horizontal axis is on a logarithmic scale.  
}
    \label{fig:survival_plots}
\end{figure}

\subsection{Strength of the LP relaxations and features of the optimal solutions}
\label{sec:comp_LP}

\Rev{We conclude the computational analysis with an assessment of the LP relaxation strength of the seven variants, as tighter LP bounds generally correlate with faster branch-and-bound convergence and higher rates of optimally solved instances. Moreover, we briefly discuss the main structural features of the optimal \BPPS{} solutions. A detailed breakdown is reported in Section~E.3 of the electronic companion. 
The analysis is conducted on 427 randomly selected instances for which an optimal solution is known, and an optimal LP solution can be computed for all arc-flow formulation variants within a 30-minute time limit. The average integrality gap, defined as the percentage difference between the LP relaxation value and the known optimal integer objective, taken with respect to the optimum, confirms the hierarchy observed in the computational experiments. Among the natural formulation variants, $\foneLP$ exhibits the largest average gap ($17.8\%$), which is substantially reduced by the addition of the MCIs in $\ftwoLP$ ($8.3\%$) and by the inclusion of the MBI in $\fthreeLP$ ($3.0\%$). The arc-flow variants have tighter LP relaxations consistently, confirming the theoretical dominance stated in Proposition~\ref{prop:af_dominates_compact}: $\afoneLP$ already achieves an average gap of $6.6\%$, comparable to $\ftwoLP$, while $\afthreeLP$ reaches $1.6\%$, the lowest among all variants. The contrast between the two families is most striking on the $36$ real-world instances, where the natural formulation variants exhibit gaps in the range $9$--$23\%$ even for $\fthreeLP$, whereas all three arc-flow variants achieve an average integrality gap below $1\%$; this is consistent with the dominance of the arc-flow formulation on the real-world instances, observed in Table~\ref{tab:exact_1}.
Nevertheless, a stronger LP relaxation does not automatically translate into superior computational performance. As shown in Table~\ref{tab:exact_1}, the arc-flow variants are outperformed by the natural formulation on several instance groups.
This apparent discrepancy is explained by the size of the arc-flow DAG, which can grow very large depending on the instance structure, making the LP at the root node itself computationally expensive to solve and, in the most extreme cases, preventing the computation of any valid dual bound within the time limit.
The LP strength advantage of the arc-flow formulation is therefore offset, in such instances, by the excessive computational overhead caused by the DAG size. }

\Rev{Beyond the comparison of the LP relaxations, the statistics reported in Section~E.3 of the electronic companion also provide some insight into the structure of optimal \BPPS{} solutions. On the random instances, optimal bins are generally highly utilized, with an average fill rate of $92.4\%$, and contain on average $5.8$ items belonging to $1.6$ active classes. The real-world solutions exhibit an even more homogeneous class composition, with exactly one active class and $1.7$ items per bin on average, while maintaining an average fill rate of $87.8\%$.}

\section{Conclusions}
\label{sec:conclusions}

In this paper, we introduced and studied the \emph{Bin Packing Problem with Setups} (\BPPS{}), a novel generalization of the classical Bin Packing Problem in which items are partitioned into classes and, whenever an item from a given class is packed into a bin, a setup weight and cost are incurred.  
This setting models a wide range of practical applications where setup operations are required, particularly in production planning and logistics.  
We proposed a natural ILP formulation for the \BPPS{}. We analyzed the structural properties of its LP relaxation, which admits a closed-form optimal solution but may yield arbitrarily weak lower bounds. To strengthen the relaxation, we introduced a new family of valid inequalities, the \emph{Minimum Classes Inequalities} (MCIs), and proved that their inclusion guarantees a worst-case performance ratio of $1/2$. We also introduced the \emph{Minimum Bins Inequality} (MBI), which further improved the quality of the LP relaxation.  
\Rev{Beyond strengthening the natural formulation, we developed an arc-flow formulation for the \BPPS{}, based on a tailored graph construction and compression procedure that encodes class activation, setup weights, and setup costs directly in the graph. We proved that the LP relaxation of this formulation dominates that of the natural formulation, and we extended the MCIs and the MBI to the arc-flow model as well.}
To support computational evaluation, we developed a publicly available benchmark library of \Rev{$576$ randomly generated} \BPPS{} instances, designed to capture a broad spectrum of instance characteristics, including variations in item weights, class distributions, setup costs, and setup weights\Rev{, complemented by $36$ real-world instances arising from a vehicle-routing application in the distribution of fresh and frozen grocery products}. Using this library, we conducted an extensive computational campaign to assess the performance of \Rev{both formulation families and} the enhancements proposed for \Rev{each}. Our experiments show that \Rev{the proposed enhancements substantially improve computational performance within both formulation families, increasing the number of instances solved to proven optimality and reducing the optimality gaps on unsolved instances.} \Rev{Comparing the two formulations, the item weight distribution emerges as the dominant factor: the natural formulation is generally more effective on instances with small or medium item weights, while the arc-flow formulation dominates on instances with large item weights and on the real-world testbed, where it solves all instances to optimality and consistently yields tighter LP relaxations. This complementary behavior indicates that neither formulation is consistently superior, and that the instance structure should guide the choice between them.}

Future research directions include investigating the polyhedral structure of \Rev{the proposed formulations}, developing a branch-and-price algorithm to solve a set-covering formulation of the problem via column generation, and designing tailored exact and heuristic algorithms that exploit the specific role of class-induced setup constraints. \Rev{The formulations and the benchmark library introduced in this paper provide a natural starting point for this line of work.} Another promising avenue is to develop polynomial-time approximation algorithms with provable performance guarantees.

\section*{Acknowledgments}
The authors thank the anonymous reviewers and the editor for their valuable comments and thorough review of the paper.

\footnotesize
\bibliographystyle{abbrvnat}

\biboptions{authoryear}
\setlength{\bibsep}{3pt}    

\begin{spacing}{0.95}

\end{spacing}

\normalsize

\pagebreak

\setcounter{section}{0}
\setcounter{subsection}{0}
\setcounter{equation}{0}
\setcounter{figure}{0}
\setcounter{table}{0}
\setcounter{page}{1}

\renewcommand{\thesection}{E.\arabic{section}}
\renewcommand{\thesubsection}{E.\arabic{section}.\arabic{subsection}}

\renewcommand{\theequation}{E.\arabic{equation}}
\renewcommand{\thefigure}{E.\arabic{figure}}
\renewcommand{\thetable}{E.\arabic{table}}

\section*{\large Electronic companion}

\section{Notation summary}

In Table~\ref{tab:notation}, we summarize the notation used throughout the paper.
The first three blocks of the table list, respectively, the data of the classical \BPP{}, the additional data specific to the \BPPS{}, and the auxiliary parameters and valid bounds used in our article.
The fourth block describes the natural ILP formulation $\fone$ and its enhancements ($\ftwo$, $\ftwobis$, and $\fthree$).
The fifth and sixth blocks introduce, respectively, the notation for the arc-flow DAG $\mathcal{G}$ and the arc-flow formulation $\afone$ and its enhancements ($\aftwo$ and $\afthree$).
The last block reports the optimal objective function values of the LP relaxations of all formulations, which provide lower bounds on~$\opt$.

\scriptsize
\renewcommand{\arraystretch}{1}
\tabcolsep 4pt
\begin{longtable}[c]{lp{10.8cm}}

\caption{\footnotesize{Summary of notation}}
\label{tab:notation}

\endfirsthead
\endhead

\endfoot
\endlastfoot

\toprule
\textbf{\textbf{\BPP\ givens}}  & \multicolumn{1}{l}{\textbf{Description}} \\
\midrule
$d \in \mathbb{Z}_{> 0}$ & Bin capacity \\
$n \in \mathbb{Z}_{> 0}$ & Number of items \\
$\mathcal{I} = \{1, 2, \dots, n\}$ & Set of items \\
$w_i \in \mathbb{Z}_{> 0}$ & Weight of item $i \in \mathcal{I}$ \\
$\beta \in \mathbb{Z}_{> 0}$ & Optimal objective function value for the \BPP{}\\
\midrule
{\textbf{\BPPS\ givens}}  & \multicolumn{1}{l}{\textbf{Description}} \\
\midrule
$m \in \mathbb{Z}_{\ge 1}$ & Number of classes \\
$\mathcal{C} = \{1, 2, \dots, m\}$ & Set of item classes \\
$\mathcal{P} = \{\items_1, \items_2,\dots, \items_m\}$ & Partition of the items into classes \\
$\mathcal{I}_c \subseteq \items$ & Subset of items belonging to class $c \in \mathcal{C}$ \\
$s_c \in \mathbb{Z}_{\ge 0}$ & Setup weight for class $c \in \mathcal{C}$ \\
$f_c \in \mathbb{Z}_{\ge 0}$ & Setup cost for class $c \in \mathcal{C}$ \\
$\bincost \in \mathbb{Z}_{> 0}$ & Bin cost \\
$\mathcal{S} \subseteq 2^\items$ & Feasible \BPPS\ partition into packing patterns \\
$S \in \mathcal{S}$ & A single packing pattern (bin) belonging to a feasible \BPPS\ partition $\mathcal{S}$ \\
$\mathcal{C}(S) = \{c \in \classes : S \cap \items_c \ne \emptyset\}$ & Set of classes active in a subset of items $S \subseteq \items$ \\
$\UB(\mathcal{S}) \in \mathbb{Z}_{> 0}$ & Objective function value of a feasible \BPPS{} partition $\mathcal{S}$\\
$\opt \in \mathbb{Z}_{\ge 2}$ & Optimal objective function value for the \BPPS{}\\
\midrule
{\textbf{\BPPS\ additional givens}}  & \multicolumn{1}{l}{\textbf{Description}} \\
\midrule
$\solclassopt \subseteq 2^{\items_c}$ & Optimal solution for the \BPP\ with items of class $c \in \classes$ and bin capacity $d - s_c$\\
$\beta_c$ & Optimal objective function value for the \BPP\ with items of class $c \in \classes$ and bin capacity  $d - s_c$\\
$\overline\beta_c$ & Upper bound on the number of bins to pack all items in $\items_c$ and bin capacity  $d - s_c$\\
$ \binnumberLBclass_c = \lceil \sum_{i \in \items_c} w_i  \, / (d - s_c)\rceil$ & Lower bound on the number of bins required for class $c \in \classes$ \\
$\binnumberUB$ & Upper bound on the number of bins used in any optimal \BPPS\ solution\\
$\mathcal{B} = \{1, 2, \dots, \binnumberUB\}$ & Set of available bins \\
$\hat{\binnumberUB} = \sum_{c\in\classes}\overline\beta_c$ & Upper bound on the number of bins used in any optimal \BPPS\ solution \\
$\binnumberLB=\lceil(\sum_{i \in \items} w_i \;+\; \sum_{c \in \classes} \binnumberLBclass_c \, s_c)/d \rceil$ & Lower bound on the number of bins used in any optimal \BPPS\ solution  \\
\midrule
\textbf{Natural formulation for the BPPS}  & \multicolumn{1}{l}{\textbf{Description}} \\
\midrule
$\fone$ & Natural formulation for the \BPPS \\
$\ftwo$ & $\fone$ with MCIs \\
$\ftwobis$ & $\fone$ with MCIs and MBI\\
$\fthree$ & $\fone$ with MCIs and MBI, and UB to the number of bins in any optimal solution \\
\midrule
\textbf{Arc-flow DAG notation}  & \multicolumn{1}{l}{\textbf{Description}} \\*
\midrule
$\mathcal{G} = (\mathcal{V}, \mathcal{A})$ & Arc-flow DAG, with vertex set $\mathcal{V}$ and arc set $\mathcal{A}$ \\
$\sigma$, $\tau$ & Source and sink vertices of $\mathcal{G}$ \\
$(\ell, t, e)$ & Vertex of $\mathcal{G}$, identified by load $\ell \in \mathbb{Z}_{\ge 0}$, stage $t \in \mathbb{Z}_{\ge 0}$, and active flag $e \in \{-1, 0, 1\}$ \\
$\omega_a \in \mathbb{Z}_{\ge 0}$ & Weight of arc $a \in \mathcal{A}$ \\
$\lambda_a \in \mathbb{Z}_{\ge 0}$ & Cost of arc $a \in \mathcal{A}$ \\
$\itemarcs \subseteq \mathcal{A}$ & Set of item arcs associated with item $i \in \items$ \\
$\classarcs \subseteq \mathcal{A}$ & Set of setup arcs associated with class $c \in \classes$ \\
$\phi(u, \tau)$ & Sink label of $u \in \mathcal{V}$: weight of the maximum weight path from $u$ to $\tau$ in $\mathcal{G}$ \\
$\phi(\sigma, u)$ & Source label of $u \in \mathcal{V}$: weight of the maximum weight path from $\sigma$ to $u$ in $\mathcal{G}$ \\
\midrule
\textbf{Arc-flow formulation for the BPPS}  & \multicolumn{1}{l}{\textbf{Description}} \\
\midrule
$\afone$ & Arc-flow formulation for the \BPPS \\
$\aftwo$ & $\afone$ with MCIs \\
$\afthree$ & $\afone$ with MCIs and MBI \\
\midrule
\textbf{Lower bounds on $\opt $}  & \multicolumn{1}{l}{\textbf{Description}} \\
\midrule
$\zeta(\foneLP)$, $\zeta(\ftwoLP)$, $\zeta(\ftwoLPbis)$, $\zeta(\fthreeLP)$ & Optimal objective function value of the LP relaxation  of $\fone$, $\ftwo$, $\ftwobis$ and $\fthree$, respectively   \\
$\zeta(\afoneLP)$, $\zeta(\aftwoLP)$, $\zeta(\afthreeLP)$ & Optimal objective function value of the LP relaxation  of $\afone$, $\aftwo$ and $\afthree$, respectively   \\
$\ratio(LB) \in [0,1]$ & Worst-case performance ratio of a lower bound $LB$ for the \BPPS{} \\
\bottomrule
\end{longtable}
\normalsize

\section{Deferred proofs}

\subsection{A structural property of the BPP}

An instance of the \BPP{} is defined by a bin capacity \( d \) and a set of 
items \( \items = \{1,2,\ldots,n\} \), where each item \( i \in \items \) has a positive 
integer weight \( w_i \in \mathbb{Z}_{>0} \). The minimum number of bins required to pack 
all items of a BPP instance is denoted by \( \beta \). 
In the literature (see, e.g., \cite[Section~8.3.2]{MT90}), the following inequality is known 
to hold for every \BPP{} instance:
\begin{equation}
\label{prop:lemmaBPP}
\left\lceil \frac{\sum_{i \in \items} w_i}{d} \right\rceil \;\geq\; \frac{\beta}{2},
\end{equation}
which states that the ceiling of the total weight of the items divided by the bin capacity is always at least 
half the minimum number of bins.
The next lemma provides a slight generalization of this result, used in the proof of 
Proposition~\ref{prop:ratio_MCI}, showing that, whenever more than one bin is required in any optimal \BPP{} solution, the total weight of the items must exceed half of the 
total capacity of the bins in that solution.  

\begin{lemma}\label{prop:lemma}
For every \BPP{} instance, if \( \beta > 1 \) then
\begin{equation}
\sum_{i \in \items} w_i \;>\; \frac{\beta \,
d}{2}.
\label{lemma}
\end{equation}
\end{lemma}

\begin{proof}
Consider any optimal solution to the \BPP{}.  
Let $\mathcal{S}^\star$ be the set of the $\beta$ bins used in this optimal solution.
Notice that, in an optimal \BPP{} solution, there cannot exist two bins with total weight lower than or equal to $\frac{d}{2}$ since, otherwise, they could be merged into a single bin.
If all bins have load greater than $\frac{d}{2}$, the claim follows immediately since
\[
\sum_{i\in\items} w_i \;=\; \sum_{S \in \mathcal{S}^\star} \sum_{i\in S} w_i \;>\;  \frac{\beta \,d}{2}.
\]

Otherwise, there is at most one bin $\check{S} \in \mathcal{S}^\star$ with total load $\sum_{i\in\check{S}} w_i \le \frac{d}{2}$.  
Since $\beta>1$, there exists another bin $\hat{S} \in \mathcal{S}^\star$; by optimality, we must have
\[
\sum_{i \in \check{S}} w_i + \sum_{i\in \hat{S}} w_i \;>\; d,
\]
otherwise the two bins $\check{S}$ and $\hat{S}$ could be merged.

Hence,
\[
\sum_{i\in\items} w_i
= \sum_{\substack{i\in\items: \\ i \notin \check{S} \cup \hat{S}}} w_i
+ \sum_{i\in \check{S}} w_i
+ \sum_{i\in \hat{S}} w_i
\;>\; \frac{d}{2}(\beta-2) + d
= \frac{\beta \, d}{2}.
\]

\end{proof}

\subsection{Proof of Proposition \ref{prop:lprelax}}

\begin{proof}
Multiplying the capacity constraints~\eqref{con:kanto-cap} by the positive bin cost~$\bincost$ and summing over all $b\in\bins$ gives
\begin{equation*}
    \bincost \sum_{i \in \items} w_i \sum_{b \in \bins} x_{ib} 
    + \bincost \sum_{c \in \classes} s_c \sum_{b \in \bins} y_{cb} 
    \ \le\ d \sum_{b \in \bins} \bincost \, z_b .
\end{equation*}
By the assignment constraints~\eqref{con:kanto-ass}, $\sum_{b \in \bins} x_{ib}=1$ for every $i\in\items$.  
Since every class $c \in \classes$ contains at least one item $i \in \items_c$, the linking constraints~\eqref{con:kanto-alpha} imply
\begin{equation}
\label{KKK}
\sum_{b\in\bins} y_{cb} \;\ge\; \sum_{b\in\bins} x_{ib} \;=\; 1.    
\end{equation}

Therefore
\begin{equation}
\label{LB_reason}    
    \bincost \sum_{i \in \items} w_i + \bincost \sum_{c \in \classes} s_c
    \ \le\ d \sum_{b \in \bins} \bincost \, z_b,
    \quad\Longrightarrow\quad
    \sum_{b \in \bins} \bincost \, z_b 
    \ \ge\ \frac{\bincost}{d} \!\left( \sum_{i \in \items} w_i + \sum_{c \in \classes} s_c \right).
\end{equation}
Moreover,
\[
    \sum_{b \in \bins} \sum_{c \in \classes} f_c \, y_{cb} 
    \ = \ \sum_{c \in \classes} f_c \sum_{b \in \bins} y_{cb} 
    \ \ge\ \sum_{c \in \classes} f_c.
\]
Therefore, every feasible solution of $\foneLP$ satisfies
\begin{equation}\label{proof:geq_of}
    \sum_{b \in \bins} \!\Bigl( \bincost \, z_{b} + \sum_{c \in \classes} f_c \, y_{cb} \Bigr) 
    \ \ge\   
      \frac{\bincost}{d} \!\left( \sum_{i \in \items} w_i + \sum_{c \in \classes} s_c \right) + \sum_{c \in \classes} f_c
\end{equation}
and, accordingly, the right-hand side of \eqref{proof:geq_of} is a lower bound on $\zeta(\foneLP)$.
Now consider the solution defined in~\eqref{sol:LP1}.  
It satisfies constraints~\eqref{RELAX} by construction.  
Indeed, for every $i\in\items$,
\[
\sum_{b \in \bins} x_{ib}
= \underbrace{\binnumberUB \, \frac{1}{\binnumberUB}}_{\text{by }\eqref{sol:LP1}}
= 1,
\]
so the assignment constraints~\eqref{con:kanto-ass} hold. 
Moreover, for each $b\in\bins$,
\[
\sum_{i \in \items} w_i \, x_{ib} + \sum_{c \in \classes} s_c \, y_{cb}
= \underbrace{\sum_{i \in \items} w_i \, \frac{1}{\binnumberUB}}_{\text{by }\eqref{sol:LP1}}
  + \underbrace{\sum_{c \in \classes} s_c \, \frac{1}{\binnumberUB}}_{\text{by }\eqref{sol:LP1}}
= \frac{1}{\binnumberUB}\!\left(\sum_{i \in \items} w_i + \sum_{c \in \classes} s_c\right)
= \underbrace{d \, z_b}_{\text{by }\eqref{sol:LP1}} ,
\]
so the capacity constraints~\eqref{con:kanto-cap} also hold at equality.  
Finally, the linking constraints~\eqref{con:kanto-alpha} are satisfied since by \eqref{sol:LP1} the $x$-variables and $y$-variables take the same value $1/\binnumberUB$.

Hence, the solution defined in~\eqref{sol:LP1} is a feasible solution for $\foneLP$. Its objective value is
\[
\sum_{b \in \bins} \!\Bigl(\bincost \, z_b + \sum_{c \in \classes} f_c \, y_{cb}\Bigr)
= \underbrace{\bincost \sum_{b\in\bins} \frac{\sum_{i \in \items} w_i + \sum_{c \in \classes} s_c}{\binnumberUB \, d}}_{\text{by }\eqref{sol:LP1}}
\;+\;
\underbrace{\sum_{c \in \classes} f_c \sum_{b \in \bins} \frac{1}{\binnumberUB}}_{\text{by }\eqref{sol:LP1}}
= \frac{\bincost}{d}\!\left(\sum_{i \in \items} w_i + \sum_{c \in \classes} s_c\right) + \sum_{c \in \classes} f_c .
\]
Thus, the objective value of the solution in~\eqref{sol:LP1} coincides with the right-hand side of~\eqref{proof:geq_of}.  
Consequently, the solution attains the lower bound and is therefore optimal, with optimal objective function value as in~\eqref{eq:LPoptval}.
\end{proof}

\subsection{Proof of Proposition \ref{prop:mcoi}}

\begin{proof}
Consider a feasible solution to $\fone$.  
From the assignment constraints~\eqref{con:kanto-ass} and the linking constraints~\eqref{con:kanto-alpha}, we know that, for every $c \in \classes$, if an item of class~$c$ is packed in bin~$b$, then $y_{cb} = 1$.  
Thus, items of any class $c \in \classes$ can only be packed in bins $b$ for which $y_{cb} = 1$.  
For every $c \in \classes$, let us denote this set of bins as 
\[
\bins_c = \{\, b \in \bins \mid y_{cb} = 1 \,\}.
\]
For every $b \in \bins_c$, since $z_b \le 1$, the capacity constraints~\eqref{con:kanto-cap} imply that
\[
\underbrace{\sum_{i \in \items} w_i \, x_{ib}}_{\ge \sum_{i \in \items_c} w_i \, x_{ib}} + s_c + \underbrace{\sum_{g \in \classes, \, g \ne c} s_g \, y_{gb}}_{\ge 0} \;\le\; d
    \quad\Longrightarrow\quad
\sum_{i \in \items_c} w_i \, x_{ib}  \;\le\; d - s_c.
\]
Summing over all $b \in \bins_c$ yields
\[
\sum_{b \in \bins_c} \; \sum_{i \in \items_c} w_i \, x_{ib} \;\le\; \sum_{b \in \bins_c} (d - s_c).
\]
The left-hand side equals $\sum_{i \in \items_c} w_i$, since items $i \in \items_c$ are packed in bins in $\bins_c$, and thus $\sum_{b \in \bins_c} x_{ib} = \sum_{b \in \bins} x_{ib} = 1$.  
The right-hand side equals $|\bins_c|(d - s_c) = \big(\sum_{b \in \bins} y_{cb}\big)(d - s_c)$, since $y_{cb} = 1$ only if $b \in \bins_c$.
Therefore, for every $c \in \classes$, we have
\[
\sum_{i \in \items_c} w_i \;\le\; \left( \sum_{b \in \bins} y_{cb} \right)(d - s_c) 
    \quad\Longrightarrow\quad
\sum_{b \in \bins} y_{cb} \;\ge\; \frac{\sum_{i \in \items_c} w_i}{d - s_c}.
\]
Since $\sum_{b \in \bins} y_{cb}$ is an integer, the right-hand side can be rounded up, which coincides with the definition of~$\binnumberLBclass_c$.
\end{proof}

\subsection{Proof of Proposition \ref{prop:lprelax_2}}

\begin{proof}
The proof follows the same logical scheme as that of Proposition~\ref{prop:lprelax}.  
The only difference is that, instead of using~\eqref{KKK}, we now directly rely on the MCI~\eqref{MCI}.  
It follows that every feasible solution of $\ftwoLP$ satisfies
\begin{equation}\label{eq:LB_ftwoLP}
\sum_{b\in\bins}\!\Bigl(\bincost \, z_b + \sum_{c\in\classes} f_c\,y_{cb}\Bigr)
\;\ge\;
\frac{\bincost}{d}\!\left(\sum_{i\in\items} w_i + \sum_{c\in\classes} \binnumberLBclass_c \, s_c \right)
\;+\;
\sum_{c\in\classes} \binnumberLBclass_c \, f_c.
\end{equation}
and, accordingly, the right-hand side of \eqref{eq:LB_ftwoLP} is a lower bound on $\zeta(\ftwoLP)$.
Now consider the solution defined in~\eqref{sol:LP2}.  
As in Proposition~\ref{prop:lprelax}, it can be verified that this solution satisfies all the constraints of $\ftwoLP$, and is therefore a feasible solution.
Its objective value is
\begin{align*}
    \sum_{b \in \bins} \!\Bigl(\bincost \, z_b + \sum_{c \in \classes} f_c \, y_{cb}\Bigr)
&= \underbrace{\bincost \sum_{b\in\bins} \frac{\sum_{i \in \items} w_i + \sum_{c \in \classes} \binnumberLBclass_c \, s_c}{\binnumberUB \, d}}_{\text{by }\eqref{sol:LP2}}
\;+\;
\underbrace{\sum_{c \in \classes} f_c \sum_{b \in \bins} \frac{\binnumberLBclass_c}{\binnumberUB}}_{\text{by }\eqref{sol:LP2}} \\[1ex]
&= \frac{\bincost}{d}\!\left(\sum_{i \in \items} w_i + \sum_{c \in \classes} \binnumberLBclass_c \, s_c\right) + \sum_{c \in \classes} \binnumberLBclass_c \, f_c .
\end{align*}
Thus, the objective value of this solution coincides with the right-hand side of~\eqref{eq:LB_ftwoLP}.  
Consequently, the solution attains the lower bound and is therefore optimal, with objective function value~\eqref{eq:LPoptval_2}.
\end{proof}

\subsection{Proof of Proposition \ref{prop:ratio_MCI}}

\begin{proof}
The proof relies on constructing a feasible \BPPS{} solution through a three-step procedure, which we refer to as the \emph{Constructive Heuristic Algorithm} (CHA) for the \BPPS{}. 
The CHA constructs the three feasible solutions $\mathcal S_1$, $\mathcal S_2$, and $\mathcal S_3$ whenever the stopping criteria for the corresponding steps are met. The proof considers the solution associated with the first applicable termination case and proves that $\UB(\mathcal S) < 2\,\zeta(\ftwoLP)$ for that case.
Depending on the case, it may be easier to establish the inequality using a worse (higher-cost) solution; this is why we consider different possible terminations of the CHA.

Having a feasible \BPPS{} solution $\mathcal S$ with $\UB(\mathcal S) < 2\,\zeta(\ftwoLP)$ suffices to prove the proposition. Indeed, the thesis is equivalent to
\begin{equation}\label{eq:equivalent_th}
\zeta(\ftwoLP) > \frac12 \,\opt,
\end{equation}
and, since every feasible \BPPS{} solution $\mathcal S$ satisfies $\UB(\mathcal S) \ge \opt$, to prove \eqref{eq:equivalent_th} it suffices to show that $\zeta(\ftwoLP) > \frac12\,\UB(\mathcal S)$ for the feasible solution $\mathcal S$ returned by the CHA.

These are the three steps of the CHA:

\begin{enumerate}

\item \label{itm:uno} For each class $c\in\classes$, let $I_c$ denote the associated \BPP{} instance with item set $\items_c$, weights $\{w_i\}_{i\in\items_c}$, and bin capacity $d_c:=d-s_c$. In the first step of the CHA, for each class $c\in\classes$ we solve the associated \BPP{} instance $I_c$ to optimality. We denote by $\solclassopt$ an optimal solution to $I_c$, and by $\beta_c := |\solclassopt|$ the corresponding optimal number of bins, i.e., the minimum number of bins of capacity $d_c=d-s_c$ required to pack all the items of class~$c$. Let $\tilde\classes\subseteq\classes$ denote the subset of classes (if any) for which $\beta_c=1$, i.e., whose items can be packed into a single bin; for every $c\in\tilde\classes$, $\solclassopt=\{\items_c\}$ is the corresponding single packing pattern. Moreover, since $\beta_c=1$, $\sum_{i\in\mathcal I_c}w_i\le d-s_c$; this yields $\gamma_c=1$ for all $c\in\tilde\classes$.

The outcome of this step is a feasible \BPPS{} solution $\mathcal S_1 = \bigcup_{c\in\classes} \solclassopt$,
which contains, for every class $c\in\classes$, exactly $\beta_c$ bins whose items all belong to class~$c$.
If $\tilde\classes=\emptyset$, i.e., if every class requires at least two bins, the CHA terminates at this point; otherwise, it proceeds to Step~2.
The cost of the feasible solution $\mathcal{S}_1$ is
\begin{equation}\label{eq:cost_step_1}
\UB(\mathcal{S}_1) = \sum_{c\in\classes}\beta_c\,f_c + \bincost \sum_{c\in\classes}\beta_c.
\end{equation}

\item \label{itm:due} In the second step of the CHA, we define an auxiliary \BPP{} instance $I_{\tilde\classes}$ in which each class $c\in\tilde\classes$ is represented by a single \emph{class-aggregated item} of weight $\sum_{i\in\items_c}w_i+s_c$, and the bin capacity is $d$.
We assume $I_{\tilde\classes}$ is solved to optimality, and denote by $\mathcal S^\star(I_{\tilde\classes})$ an optimal solution, and by $\delta := |\mathcal S^\star(I_{\tilde\classes})| $ the corresponding optimal number of bins. 
If $\delta\ge2$, i.e., if the class-aggregated items cannot all be packed into a single bin of capacity $d$, the outcome of this step is obtained by merging some of the bins obtained in Step~\ref{itm:uno}.
In this case, when $\delta\ge2$, the CHA terminates at this point; otherwise, it proceeds to Step~\ref{itm:tre}.

The heuristic solution $\mathcal{S}_2$ obtained at this step uses $\sum_{c\in\classes\setminus\tilde\classes}\beta_c+\delta$ bins, and items of any class $c$ are still packed in exactly $\beta_c$ bins. Hence, its cost is
\begin{equation}\label{eq:cost_step_2}
\UB(\mathcal{S}_2) = \sum_{c\in\classes}\beta_c\,f_c + \bincost\Big(\sum_{c\in\classes\setminus\tilde\classes}\beta_c \,+\, \delta\Big).
\end{equation}

\item \label{itm:tre} In the third step of the CHA, we create a single meta item, of weight $\sum_{c\in\tilde\classes}\Big(\sum_{i\in\items_c}w_i+s_c\Big)$, corresponding to the total weight of the unique packing pattern assigned to a bin in Step~\ref{itm:due} (this step is only reached when $\delta=1$). Next, we select any class $c^\dagger\in\classes\setminus\tilde\classes$; such a class exists, since if $\classes\setminus\tilde\classes=\emptyset$ after Step~\ref{itm:uno} and $\delta=1$, then all items in $\items$ fit into a single bin, which is the trivial case excluded in the Introduction to this paper.
We then try to place the newly created meta item into the residual capacity of one of the packing patterns in $\mathcal S^\star(I_{c^\dagger})$, say, the one with the minimum load.
At the end of this step, there are two possible outcomes, both of which terminate the CHA. If the meta item fits into a packing pattern $S \in\mathcal S^\star(I_{c^\dagger})$, we obtain a feasible \BPPS{} solution $\mathcal{S}_3$ which uses $\sum_{c\in\classes\setminus\tilde\classes}\beta_c$ bins, with every class $c$ still active in exactly $\beta_c$ bins. Hence, its cost is
\begin{equation}\label{eq:cost_step_3_fits}
\UB(\mathcal{S}_3) = \sum_{c\in\classes}\beta_c\,f_c + \bincost\sum_{c\in\classes\setminus\tilde\classes}\beta_c.
\end{equation}

On the other hand, if the meta item does not fit into any packing pattern of $\mathcal S^\star(I_{c^\dagger})$, then the feasible solution $\mathcal{S}_3$ coincides with that from Step~\ref{itm:due}, and its cost (since here $\delta=1$) is
\begin{equation}\label{eq:cost_step_3_no_fit}
\UB(\mathcal{S}_3) = \sum_{c\in\classes}\beta_c\,f_c + \bincost\Big(\sum_{c\in\classes\setminus\tilde\classes}\beta_c\,+\,1\Big).
\end{equation}

\end{enumerate}

We now show that, at every possible termination of the CHA, $\zeta(\ftwoLP)$ is more than half of \eqref{eq:cost_step_1}, \eqref{eq:cost_step_2}, \eqref{eq:cost_step_3_fits}, or \eqref{eq:cost_step_3_no_fit}, depending on the respective termination.

Let us consider $\zeta(\ftwoLP)$, computed as in \eqref{eq:LPoptval_2}:
\[
\zeta(\ftwoLP) \, = \, \sum_{c \in \classes} \binnumberLBclass_c \, f_c
    + \frac{\bincost}{d} \left(\sum_{i \in \items} w_i + \sum_{c \in \classes} \binnumberLBclass_c \, s_c\right),
\]
which has a component of the form $\sum_{c\in\classes}\binnumberLBclass_c f_c$. We observe that all of \eqref{eq:cost_step_1}, \eqref{eq:cost_step_2}, \eqref{eq:cost_step_3_fits}, \eqref{eq:cost_step_3_no_fit} have a component of the form $\sum_{c\in\classes}\beta_c f_c$, and we have
\begin{equation}
    \sum_{c \in \classes} \binnumberLBclass_c \, f_c \ge \frac{\sum_{c \in \classes} \beta_c \,f_c}{2} \nonumber
\end{equation}
because of \eqref{prop:lemmaBPP} applied to the \BPP{} instance $I_c$, for every $c\in\classes$.

Hence, to prove that $\zeta(\ftwoLP)>\frac12\,\UB(\mathcal{S})$, where $\mathcal{S}$ is a feasible \BPPS{} solution returned by the CHA, it remains to show, depending on the termination of the CHA, the following four inequalities:

\begin{itemize}
\item $\displaystyle\frac{\bincost}{d} \left(\sum_{i \in \items} w_i + \sum_{c \in \classes} \binnumberLBclass_c \, s_c\right)>\frac 12\,\bincost\,\sum_{c\in\classes}\beta_c $ if the CHA terminates at the end of Step~\ref{itm:uno}, which means that $\beta_c\ge 2$ for every $c\in\classes$ and $\tilde\classes$ is empty. This inequality can be shown by observing that
\begin{equation*}
    \frac{1}{d} \, \left(\sum_{i \in \items} w_i + \sum_{c \in \classes} \binnumberLBclass_c \; s_c\right) =\frac 1d\, \sum_{c\in\classes}\left(\sum_{i\in\items_{c}}w_i+\binnumberLBclass_cs_c \right)>\frac 12\,\sum_{c\in\classes}\beta_c.
\end{equation*}
Indeed, for every $c\in\classes\setminus\tilde{\classes}$, which includes all classes $c\in\classes$ in this specific termination, we have
\begin{align}\label{eq:before_big_in}
\sum_{i\in\items_c}w_i+\binnumberLBclass_c \, s_c &\ge\sum_{i \in \items_c}\left( w_i +  \frac{w_i}{d-s_c} \; s_c\right)
=\sum_{i \in \items_c}\left(\frac{(d-s_c) \,w_i+s_c \,w_i}{d-s_c} \right)\nonumber\\[2 ex]
&=d \,  \frac{\sum_{i \in \items_c}w_i}{d-s_c} > \frac{\beta_c \, d}2
\end{align}
where the last inequality comes from Lemma \ref{prop:lemma}, applied to the \BPP{} instance $I_c$, which has an optimal solution using at least two bins since $c\in\classes\setminus\tilde{\classes}$.

\item $\displaystyle\frac{\bincost}{d} \left(\sum_{i \in \items} w_i + \sum_{c \in \classes} \binnumberLBclass_c \, s_c\right)>\frac 12\,\bincost\left(\sum_{c\in\classes\setminus\tilde\classes}\beta_c\,+\,\delta\right) $ if the CHA terminates at the end of Step~\ref{itm:due}, which means that $\delta\ge 2$. This inequality can be shown by observing that
\begin{align*}
    \frac{1}{d} \, \left(\sum_{i \in \items} w_i + \sum_{c \in \classes} \binnumberLBclass_c \; s_c\right) &=    
    \frac 1d\left( \sum_{c\in\classes\setminus\tilde{\classes}}\left(\sum_{i\in\items_{c}}w_i+\binnumberLBclass_cs_c \right)
     + \sum_{c \in \tilde\classes} \left( \sum_{i \in \items_c} w_i + s_c \right)
    \right)\\[2 ex]
    &> \frac 12
    \sum_{c\in\classes\setminus\tilde{\classes}}\beta_c \, + \, 
   \frac 12\,\delta.
\end{align*}
where, for every $c\in\classes\setminus\tilde\classes$, we used \eqref{eq:before_big_in}, and
\begin{equation}\label{eq:delta_ineq}
\sum_{c \in \tilde\classes} \left( \sum_{i \in \items_c} w_i + s_c \right) > \frac{\delta\,d}{2}
\end{equation}
is obtained by applying Lemma \ref{prop:lemma} to the \BPP{} instance $I_{\tilde\classes}$ solved in Step~\ref{itm:due}, which has an optimal solution using at least two bins when the CHA terminates at the end of Step~\ref{itm:due}, since $\delta>1$.

\item$\displaystyle\frac{\bincost}{d} \left(\sum_{i \in \items} w_i + \sum_{c \in \classes} \binnumberLBclass_c \, s_c\right)>\frac 12\,\bincost\sum_{c\in\classes\setminus\tilde\classes}\beta_c $ if the CHA terminates at the end of Step~\ref{itm:tre} and the meta item fits into one of the packing patterns in $\mathcal S^\star(I_{c^\dagger})$, for some class $c^\dagger\in\classes\setminus\tilde\classes$. This inequality can be shown by observing that
\begin{align*}
    \frac{1}{d} \, \left(\sum_{i \in \items} w_i + \sum_{c \in \classes} \binnumberLBclass_c \; s_c\right) &=    
    \frac 1d\left( \sum_{c\in\classes\setminus\tilde{\classes}}\left(\sum_{i\in\items_{c}}w_i+\binnumberLBclass_cs_c \right) +
     \sum_{c \in \tilde\classes} \left( \sum_{i \in \items_c} w_i + s_c \right)
    \right)\\[2 ex]
    &> \frac 12
    \sum_{c\in\classes\setminus\tilde{\classes}}\beta_c 
\end{align*}
where, for every $c\in\classes\setminus\tilde\classes$, we used \eqref{eq:before_big_in}.

\item $\displaystyle\frac{\bincost}{d} \left(\sum_{i \in \items} w_i + \sum_{c \in \classes} \binnumberLBclass_c \, s_c\right)>\frac 12\,\bincost\left(\sum_{c\in\classes\setminus\tilde\classes}\beta_c\,+\,1\right) $ if the CHA terminates at the end of Step~\ref{itm:tre} and the meta item does not fit into any of the packing patterns in $\mathcal S^\star(I_{c^\dagger})$. This inequality can be shown by observing that
\begin{align}
    \frac{1}{d} \, \left(\sum_{i \in \items} w_i + \sum_{c \in \classes} \binnumberLBclass_c \; s_c\right) &=    
    \frac 1d\left( 
    \sum_{\substack{c\in\classes\setminus\tilde{\classes}\\ c\ne c^\dagger}}\left(\sum_{i\in\items_{c}}w_i+\binnumberLBclass_cs_c \right)
    + \sum_{i\in\items_{c^\dagger}}w_i+\binnumberLBclass_{c^\dagger}s_{c^\dagger} + \sum_{c \in \tilde\classes} \left( \sum_{i \in \items_c} w_i + s_c \right)
    \right)\nonumber\\[2 ex]
    &> \frac 12
    \sum_{\substack{c\in\classes\setminus\tilde{\classes}\\ c\ne c^\dagger}}\beta_c \, + \, 
   \frac 12\,(\beta_{c^\dagger}+1)=\frac 12\left(\sum_{c\in\classes\setminus\tilde\classes}\beta_c\,+\,1\right).\label{eq:last_ineq}
\end{align}
where, for every $c\in\classes\setminus\tilde\classes$ different from $c^\dagger$, we used \eqref{eq:before_big_in}. 

Since $\mathcal S^\star(I_{c^\dagger})$ is an optimal solution to the \BPP{} instance $I_{c^\dagger}$, no two of its $\beta_{c^\dagger}$ bins can be merged without exceeding the capacity $d_{c^\dagger}$, as otherwise $\mathcal S^\star(I_{c^\dagger})$ would not be optimal. Moreover, by assumption, the meta item --~of weight $\sum_{c\in\tilde\classes}\left(\sum_{i\in\items_c}w_i+s_c\right)$~-- does not fit into the residual capacity of any bin in $\mathcal S^\star(I_{c^\dagger})$ either.
Consider the auxiliary \BPP{} instance with capacity $d-s_{c^\dagger}$, whose items are the $\beta_{c^\dagger}$ meta items obtained by aggregating the contents of each packing pattern in $\mathcal S^\star(I_{c^\dagger})$, plus one further item, namely the meta item of weight $\sum_{c\in\tilde\classes}\left(\sum_{i\in\items_c}w_i+s_c\right)$.
This auxiliary instance is either infeasible, if the additional meta item alone exceeds $d-s_{c^\dagger}$, or its optimal solution uses exactly $\beta_{c^\dagger}+1$ bins. Applying Lemma \ref{prop:lemma} — trivially, in the infeasible case — then yields

\begin{equation*}\label{eq:tre_not_join}
\sum_{i\in\items_{c^\dagger}}w_i+ \sum_{c \in \tilde\classes} \left( \sum_{i \in \items_c} w_i + s_c \right) > \left( d- s_{c^\dagger}\right)\,\frac{\beta_{c^\dagger}+1}2.
\end{equation*}

Hence, to obtain \eqref{eq:last_ineq} it suffices to show that
\[
\frac 1d\left( 
    \sum_{\substack{c\in\classes\setminus\tilde{\classes}\\ c\ne c^\dagger}}\left(\sum_{i\in\items_{c}}w_i+\binnumberLBclass_cs_c \right)
    +\binnumberLBclass_{c^\dagger}s_{c^\dagger} + \left( d- s_{c^\dagger}\right)\,\frac{\beta_{c^\dagger}+1}2
    \right)\nonumber
    \ge \frac 12
    \sum_{\substack{c\in\classes\setminus\tilde{\classes}\\ c\ne c^\dagger}}\beta_c \, + \, 
   \frac 12\,(\beta_{c^\dagger}+1),
\]
and therefore it is sufficient to show that
\[
\binnumberLBclass_{c^\dagger}\ge\frac{\beta_{c^\dagger}+1}2.
\]
This is true because $\binnumberLBclass_{c^\dagger}=\left\lceil \frac{\sum_{i\in\items_{c^\dagger}}w_i}{d-s_{c^\dagger}}\right\rceil$, and
$\frac{\sum_{i\in\items_{c^\dagger}}w_i}{d-s_{c^\dagger}}>\frac{\beta_{c^\dagger}}2$ (by applying Lemma \ref{prop:lemma} to the instance $I_{c^\dagger}$, since $c^\dagger\in\classes\setminus\tilde{\classes}$) and because $\beta_{c^\dagger}$ is also an integer.
\end{itemize}

\end{proof}

\subsection{Proof of Proposition \ref{prop:mbi}}

\begin{proof}
Consider any feasible solution $(\boldsymbol{x},\boldsymbol{y},\boldsymbol{z})$
of $\fone$. Summing the capacity constraints~\eqref{con:kanto-cap}
over all $b\in\bins$ and using the assignment constraints gives
\[
\sum_{i\in\items} w_i
 + \sum_{c\in\classes} s_c\sum_{b\in\bins}y_{cb}
 \le d\sum_{b\in\bins}z_b.
\]
By Proposition~\ref{prop:mcoi}, every integer-feasible solution of
$\fone$ satisfies the MCIs, and hence
$\sum_{b\in\bins}y_{cb}\ge \binnumberLBclass_c$ for every
$c\in\classes$. Since $s_c\ge0$, it follows that
\[
\sum_{i\in\items} w_i
 + \sum_{c\in\classes}\binnumberLBclass_c s_c
 \le d\sum_{b\in\bins}z_b.
\]
Therefore,
\[
\sum_{b\in\bins}z_b \ge
\frac{\sum_{i\in\items} w_i+
      \sum_{c\in\classes}\binnumberLBclass_c s_c}{d}.
\]
Finally, since the left-hand side is a sum of binary variables, the right-hand side can be rounded up, yielding
\eqref{MBI}.
\end{proof}

\subsection{Proof of Proposition \ref{prop:lprelax_3}}

\begin{proof}
The proof follows the same logical scheme as that of Proposition~\ref{prop:lprelax}.  
The only difference is that we now rely \emph{directly} on the MBI~\eqref{MBI} together with the MCIs~\eqref{MCI}.  
Hence, every feasible solution of $\ftwoLPbis$ satisfies
\begin{equation}\label{eq:LB_ftwoLP_mu}
\sum_{b\in\bins}\!\Bigl(\bincost \, z_b + \sum_{c\in\classes} f_c\,y_{cb}\Bigr)
\;\ge\;
\bincost \!\sum_{b\in\bins} z_b
\;+\;
\sum_{c\in\classes} f_c \!\sum_{b\in\bins} y_{cb}
\;\ge\;
\bincost \, \binnumberLB
\;+\;
\sum_{c\in\classes} \binnumberLBclass_c \, f_c,
\end{equation}
Accordingly, the right-hand side of \eqref{eq:LB_ftwoLP_mu} is a lower bound on $\zeta(\ftwoLPbis)$.

Now consider the point defined in~\eqref{sol:LP2bis}.  For each $b\in\bins$,
\[
\sum_{i \in \items} w_i \, x_{ib} + \sum_{c \in \classes} s_c \, y_{cb}
= \underbrace{\sum_{i \in \items} w_i \,\frac{1}{\binnumberUB}}_{\text{by }\eqref{sol:LP2bis}}
  + \underbrace{\sum_{c \in \classes} s_c \,\frac{\binnumberLBclass_c}{\binnumberUB}}_{\text{by }\eqref{sol:LP2bis}}
= \frac{1}{\binnumberUB}\!\left(\sum_{i \in \items} w_i + \sum_{c \in \classes} \binnumberLBclass_c s_c\right)
\;\le\; \underbrace{\frac{d\,\binnumberLB}{\binnumberUB}}_{\text{by }\eqref{MBI}}
= \underbrace{d \, z_b}_{\text{by }\eqref{sol:LP2bis}} ,
\]
so the capacity constraints~\eqref{con:kanto-cap} are satisfied (in general, not at equality). The satisfaction of the other constraints is the same as
in Proposition~\ref{prop:lprelax}. Therefore, it can be verified that this point satisfies all the constraints of $\ftwoLPbis$, and is therefore a feasible solution.
Its objective value is
\begin{align*}
    \sum_{b \in \bins} \!\Bigl(\bincost \, z_b + \sum_{c \in \classes} f_c \, y_{cb}\Bigr)
&= \underbrace{\bincost \sum_{b\in\bins} \frac{\binnumberLB}{\binnumberUB}}_{\text{by }\eqref{sol:LP2bis}}
\;+\;
\underbrace{\sum_{c \in \classes} f_c \sum_{b \in \bins} \frac{\binnumberLBclass_c}{\binnumberUB}}_{\text{by }\eqref{sol:LP2bis}} \\[1ex]
&= \bincost\,\binnumberLB + \sum_{c \in \classes} \binnumberLBclass_c \, f_c .
\end{align*}
Thus, the objective value of this solution coincides with the right-hand side of~\eqref{eq:LB_ftwoLP_mu}.  
Consequently, the solution attains the lower bound and is therefore optimal, with objective function value~\eqref{eq:LPoptval_2bis}.
\end{proof}

\subsection{Proof of Proposition \ref{prop:worstcasebis}}

\begin{proof}
Consider the family of \BPPS{} instances introduced in the proof of Proposition \ref{prop:worstcase} where 
\[
\opt = n \, \bincost
~~~{\rm and}~~~
\binnumberLBclass_1 \;=\; \left\lceil \frac{n \,\vartheta}{2 \, \vartheta-1} \right\rceil.
\]
Let $k = n$. By Proposition~\ref{prop:lprelax_3}, the optimal objective function value of $\ftwoLPbis$ is
\[
\zeta(\ftwoLPbis)
= 
 \bincost \!\left\lceil \frac{n \, \vartheta + \left\lceil \frac{n \, \vartheta}{2 \, \vartheta-1} \right\rceil}{2 \, \vartheta} \right\rceil 
=  \bincost \!\left\lceil \frac{n}{2} + \frac{\left\lceil \frac{n \, \vartheta}{2 \, \vartheta-1} \right\rceil}{2 \, \vartheta} \right\rceil 
\]
Consider an even number of items $n$. Taking the limit of $\zeta(\ftwoLPbis)$ as $\vartheta \to \infty$ gives
\[
\lim_{\vartheta\to\infty}\; \zeta(\ftwoLPbis)
= \lim_{\vartheta\to\infty}\; \bincost \!\left\lceil \frac{n}{2} + \frac{\left\lceil \frac{n \, \vartheta}{2 \, \vartheta-1} \right\rceil}{2 \, \vartheta} \right\rceil
= \bincost \left(\frac n2 + 1 \right),
\]

since, for all sufficiently large values of $\vartheta$, 
$$ \left\lceil \frac{n}{2} + \frac{ \left\lceil \frac{n\vartheta}{2\vartheta-1} \right\rceil }{ 2\vartheta } \right\rceil = \frac{n}{2}+1. $$ 

Thus, we have that
$$ \lim_{\vartheta\to\infty} \frac{\zeta(\ftwoLPbis)}{\opt} = \frac{ \bincost\left(\frac{n}{2}+1\right) }{ \bincost n } = \frac{1}{2}+\frac{1}{n}, $$ 

which tends to $\frac 12$ for any sequence of instances with an even number of items tending to infinity. The result follows by combining this limit with Proposition~\ref{prop:ratio_MCI}.

\end{proof}

\subsection{Proof of Proposition \ref{upper_bound_bins}}

\begin{proof}
In any optimal \BPPS{} solution $\mathcal{S}^\star \subseteq 2^\items$, we have
\begin{equation}
\label{first}
\underbrace{r\,|\mathcal S^\star| + \!\!\sum_{S\in\mathcal S^\star}\sum_{c\in\mathcal C(S)}\!\! f_c}_{= \opt}
\;\leq\; r\sum_{c \in \classes}\beta_c + \sum_{c \in \classes} f_c \, \beta_c,
\end{equation}
since the optimal objective function value of the \BPPS{}, $\opt$, cannot exceed the value of the heuristic solution obtained in Step~\ref{itm:uno} of the Constructive Heuristic Algorithm described in the proof of Proposition~\ref{prop:ratio_MCI}.

Moreover, since the items of any class $c \in \classes$ cannot be packed into fewer than $\beta_c$ bins, we have

\begin{equation} \label{second}
\sum_{S\in\mathcal S^\star}\sum_{c\in\mathcal C(S)} f_c \;\ge\; \sum_{c \in \classes} f_c \, \beta_c \, .    
\end{equation}

By chaining inequalities~\eqref{first} and~\eqref{second}, we obtain
\[
r\,|\mathcal S^\star| + \!\!\sum_{S\in\mathcal S^\star}\sum_{c\in\mathcal C(S)}\!\! f_c \le r \sum_{c \in \classes} \beta_c + \sum_{S\in\mathcal S^\star}\sum_{c\in\mathcal C(S)} f_c \, ,
\]
which implies
\[
|\mathcal{S}^\star| \;\leq\; \sum_{c \in \classes} \beta_c \, .
\]

Therefore, the number of bins used in any optimal BPPS solution is at most equal to 
the sum, over all classes, of the minimum number of bins $\beta_c$ required to pack the 
items of class $c \in \classes$. In particular, replacing $\beta_c$ with any upper bound 
$\overline{\beta}_c \ge \beta_c$ still yields a valid (though possibly weaker) overall bound:
\[
|\mathcal{S}^\star| \;\leq\; \sum_{c \in \classes} \overline{\beta}_c.
\]
\end{proof}

\bigskip

\section{Additional computational results}

This section presents supplementary computational results that complement the analysis of Section~\ref{sec:performance}. Specifically, Section~\ref{sec:boxplots} reports the distribution of computing times and optimality gaps across the variants of $\fone$ and $\afone$ via box plots. Section~\ref{sec:relaxation_instance_features} examines instead the quality of the LP relaxation bounds and the structural features of the optimal solutions across several instance classes.

\subsection{Distribution of computing times and optimality gaps}
\label{sec:boxplots}

This section reports the distributions of computation times and optimality gaps for the variants of~$\fone$ and~$\afone$ across all tested instance sizes (in terms of the number of items $n$). 
Figures~\ref{fig:boxplots_time_natural}~and~\ref{fig:boxplots_time_arcflow} show the distribution of computing times (in seconds, on a logarithmic scale) for the variants of $\fone$ and $\afone$, respectively. Each box plot 
reports the median (solid red line), the interquartile range, and the extreme values, together with the mean (dashed blue line) for the subset of instances sharing the same value of $n$. The box-plot data are restricted to the subset of instances solved to proven optimality by the strongest variant of each formulation, namely $\fthree$ and $\afthree$, respectively.

For instances with $n = 25$, all variants in both families exhibit comparable median times, below $0.4$ seconds, with $\fthree$ and $\afthree$ already showing a lower mean due to fewer extreme outliers. The gap widens as $n$ grows. For $n = 50$, the median time of $\fthree$ is approximately $0.14$ seconds, roughly two orders of magnitude below that of $\fone$; within the arc-flow family, $\afthree$ and $\aftwo$ achieve comparable medians below $6$ seconds, while $\afone$ displays a substantially wider interquartile range and a higher mean. For $n = 100$, $\fthree$ reaches a median around $3$ seconds, while $\afthree$ achieves a median slightly above $50$ seconds. At $n = 200$, $\fone$ fails to solve the majority of instances within the time limit, while $\fthree$ maintains a median below $20$ seconds. On the other hand, all the arc-flow variants reach median times of hundreds of seconds for $n=200$. Overall, the arc-flow variants tend to require more time than their natural-formulation counterparts across all values of $n$. 

\begin{figure}[h]
    \centering
    
    \scriptsize
    \renewcommand{\arraystretch}{1.25}
    \tabcolsep 8.5pt
    \begin{tabular}{lrrrrrrrrrrrrr}
\toprule
        &      &  & \multicolumn{2}{r}{$\fone$}  &  & \multicolumn{2}{r}{$\ftwo$}  &  & \multicolumn{2}{r}{$\ftwobis$} &  & \multicolumn{2}{r}{$\fthree$} \\ \cline{4-5} \cline{7-8} \cline{10-11} \cline{13-14} 
        &      &  & \multicolumn{2}{r}{time (s)} &  & \multicolumn{2}{r}{time (s)} &  & \multicolumn{2}{r}{time (s)}   &  & \multicolumn{2}{r}{time (s)}  \\ \cline{4-5} \cline{7-8} \cline{10-11} \cline{13-14} 
   & inst &  & median        & mean         &  & median         & mean        &  & median         & mean          &  & median         & mean         \\ \cline{1-2} \cline{4-5} \cline{7-8} \cline{10-11} \cline{13-14} 
        &      &  &               &              &  &                &             &  &                &               &  &                &              \\[-2ex]
$n=25$  & 137  &  & 0.4           & 41.7         &  & 0.3            & 62.0        &  & 0.1            & 25.3          &  & 0.0            & 20.8         \\[1ex]
$n=50$  & 107  &  & 6.8           & 502.3        &  & 2.3            & 175.9       &  & 0.5            & 92.3          &  & 0.1            & 41.5         \\[1ex]
$n=100$ & 79   &  & 607.2         & 920.5        &  & 28.2           & 439.3       &  & 9.6            & 169.5         &  & 3.3            & 103.7        \\[1ex]
$n=200$ & 63   &  & 1,800.0        & 1,407.0       &  & 297.2          & 575.6       &  & 48.6           & 207.4         &  & 19.0           & 136.3        \\[1ex] \bottomrule
\end{tabular}

    \bigskip

    \includegraphics[width=0.95\linewidth]{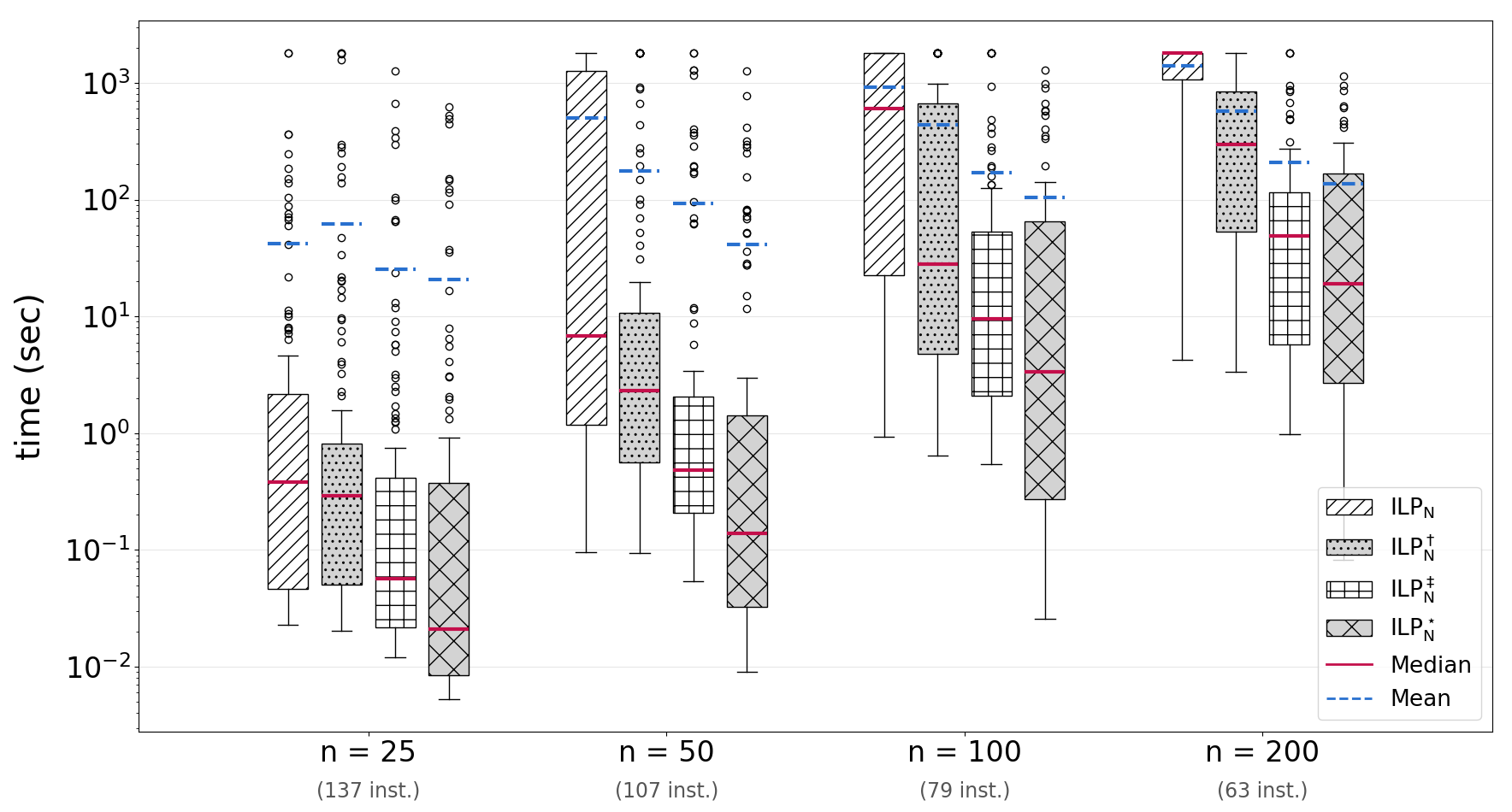}  

     \caption{Summary statistics (top) and box plots (bottom) of the solution times of the variants of $\fone$, restricted to instances solved to optimality by $\fthree$. Vertical axis on a logarithmic scale. Time limit: 1800 seconds.}
   
    \label{fig:boxplots_time_natural}
\end{figure}

\begin{figure}[h]
    \centering

\scriptsize
\renewcommand{\arraystretch}{1.25}
\tabcolsep 13pt

\begin{tabular}{lrrrrrrrrrr}
\toprule
        &      &  & \multicolumn{2}{r}{$\afone$} &  & \multicolumn{2}{r}{$\aftwo$} &  & \multicolumn{2}{r}{$\afthree$} \\ \cline{4-5} \cline{7-8} \cline{10-11} 
        &      &  & \multicolumn{2}{r}{time (s)} &  & \multicolumn{2}{r}{time (s)} &  & \multicolumn{2}{r}{time (s)}   \\ \cline{4-5} \cline{7-8} \cline{10-11} 
   & inst &  & median        & mean         &  & median        & mean         &  & median         & mean          \\ \cline{1-2} \cline{4-5} \cline{7-8} \cline{10-11} 
        &      &  &               &              &  &               &              &  &                &               \\[-2ex]
$n=25$  & 144  &  & 0.2           & 14.5         &  & 0.2           & 22.3         &  & 0.2            & 11.0          \\[1ex]
$n=50$  & 128  &  & 6.3           & 399.8        &  & 6.0           & 242.7        &  & 5.4            & 58.7          \\[1ex]
$n=100$ & 69   &  & 157.5         & 754.9        &  & 94.7          & 649.7        &  & 53.8           & 242.6         \\[1ex]
$n=200$ & 38   &  & 839.5         & 915.5        &  & 624.9         & 859.9        &  & 197.6          & 406.8         \\ [1ex]\bottomrule
\end{tabular}

\bigskip
    \includegraphics[width=0.95\linewidth]{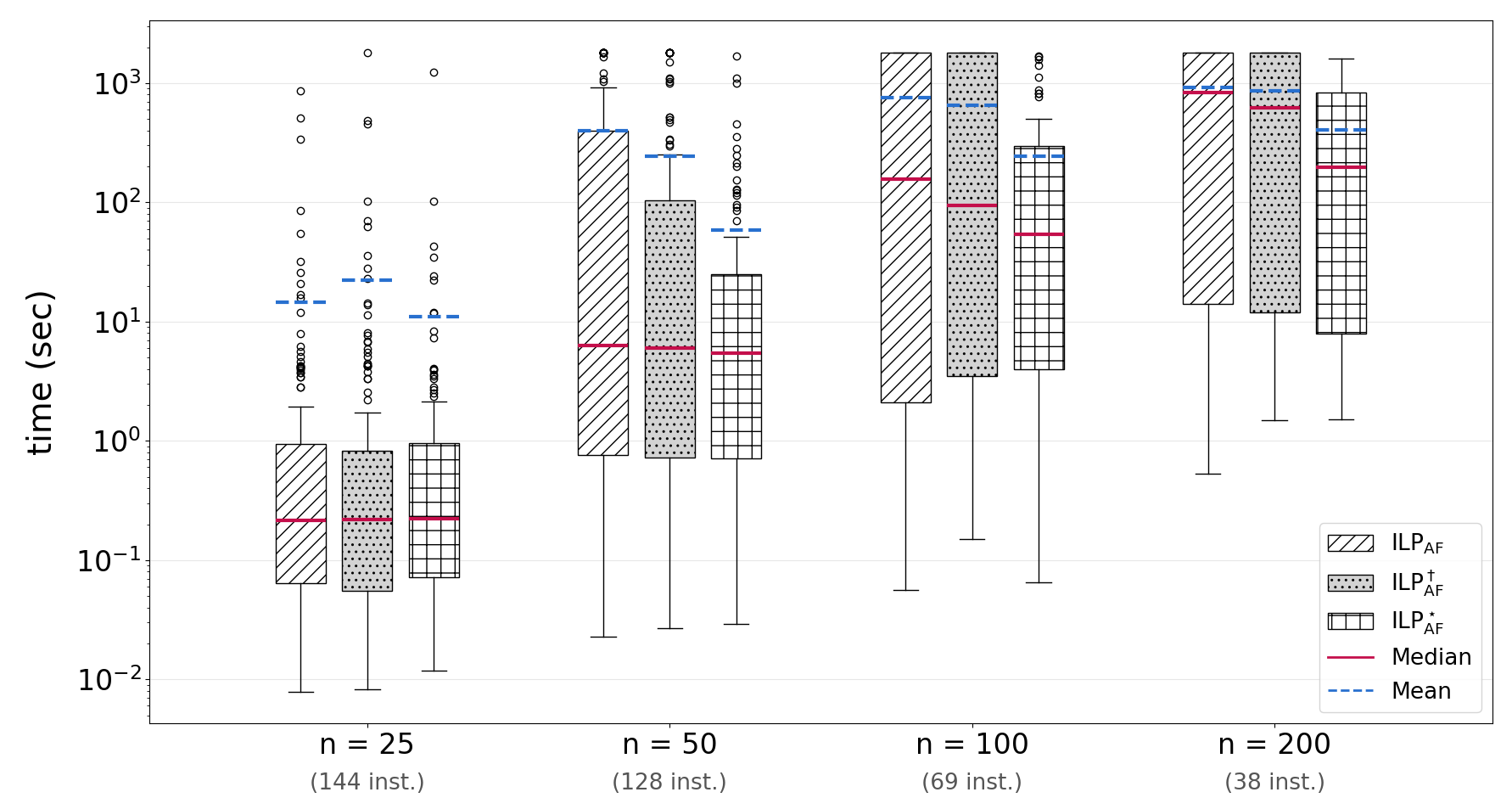}   

\caption{Summary statistics (top) and box plots (bottom) of the solution times of the variants of $\afone$, restricted to instances solved to optimality by $\afthree$. Vertical axis on a logarithmic scale. Time limit: 1800 seconds.}
    \label{fig:boxplots_time_arcflow}
\end{figure}

Figures~\ref{fig:boxplots_gap_natural}~and~\ref{fig:boxplots_gap_arcflow} show the distribution of the optimality gap (in percentage) for the variants of $\fone$ and $\afone$, respectively, restricted to the instances 
not solved to proven optimality by $\fthree$ and $\afthree$ within the time limit. For the arc-flow family, only instances for which a valid lower bound was obtained for all variants within the time limit are included. 

For the natural formulation, the results at $n = 25$ involve only seven instances and show little differentiation across variants, with median gaps all close to $11\%$. The picture changes substantially at larger scales. 
For $n = 50$, the median gap of $\fone$ is $8.3\%$, while $\fthree$ reduces it to $5.9\%$; the gap between $\fone$ and the variants that include the additional inequalities widens further as $n$ grows. For $n = 100$, $\fone$ 
achieves a median gap of $12.1\%$, whereas $\ftwo$, $\ftwobis$, and $\fthree$ all attain a median near $5.3\%$, a reduction of more than half. 
The most pronounced differences appear at $n = 200$: $\fone$ has a median gap of $12.7\%$ and a mean of $16.4\%$, while $\ftwo$, $\ftwobis$, and $\fthree$ all maintain a median below $3.9\%$ and a mean below $4.5\%$. 
These results indicate that the MCIs have a strong effect on bound quality, particularly on the hardest instances, and that the additional enhancements introduced in $\ftwobis$ and $\fthree$ beyond those in $\ftwo$ provide only marginal further improvement in the optimality gap.

For the arc-flow family, the number of unsolved instances available for this analysis is more limited, particularly at $n = 50$ (seven instances), so the corresponding statistics should be interpreted with caution. At $n = 100$, $\afone$ achieves a median gap of $5.6\%$, while $\aftwo$ and $\afthree$ reduce this to approximately $4\%$. At $n = 200$, the three variants become nearly indistinguishable in terms of median gap, with values of $3.0\%$, $3.3\%$, and $2.7\%$ for $\afone$, $\aftwo$, and $\afthree$, respectively, suggesting that, on instances of this size, the valid inequalities provide limited additional bound improvement.

\begin{figure}[]
    \centering
    
    \scriptsize
    \renewcommand{\arraystretch}{1.25}
    \tabcolsep 8.5pt

    \begin{tabular}{lrrrrrrrrrrrrr}
\toprule
        &      &  & \multicolumn{2}{r}{$\fone$}  &  & \multicolumn{2}{r}{$\ftwo$}  &  & \multicolumn{2}{r}{$\ftwobis$} &  & \multicolumn{2}{r}{$\fthree$} \\ \cline{4-5} \cline{7-8} \cline{10-11} \cline{13-14} 
        &      &  & \multicolumn{2}{r}{gap (\%)} &  & \multicolumn{2}{r}{gap (\%)} &  & \multicolumn{2}{r}{gap (\%)}   &  & \multicolumn{2}{r}{gap (\%)}  \\ \cline{4-5} \cline{7-8} \cline{10-11} \cline{13-14} 
   & inst &  & median             & mean    &  & median         & mean        &  & median             & mean      &  & median            & mean      \\ \cline{1-2} \cline{4-5} \cline{7-8} \cline{10-11} \cline{13-14} 
        &      &  &                    &         &  &                &             &  &                    &           &  &                   &           \\[-2ex]
$n=25$  & 7    &  & {11.1}      & 8.9     &  & 11.1           & 9.8         &  & 11.1               & 11.1      &  & 11.1              & 11.1      \\[1ex]
$n=50$  & 37   &  & 8.3                & 8.8     &  & 6.3            & 7.3         &  & 6.2                & 6.6       &  & {5.9}      & 6.7       \\[1ex]
$n=100$ & 65   &  & 12.1               & 12.5    &  & 5.3            & 5.0         &  & {5.3}       & 4.7       &  & 5.3               & 5.0       \\[1ex]
$n=200$ & 81   &  & 12.7               & 16.4    &  & 3.8            & 4.4         &  & {3.6}       & 4.3       &  & 3.8               & 4.5       \\[1ex] \bottomrule
\end{tabular}

        \bigskip

    \includegraphics[width=0.95\linewidth]{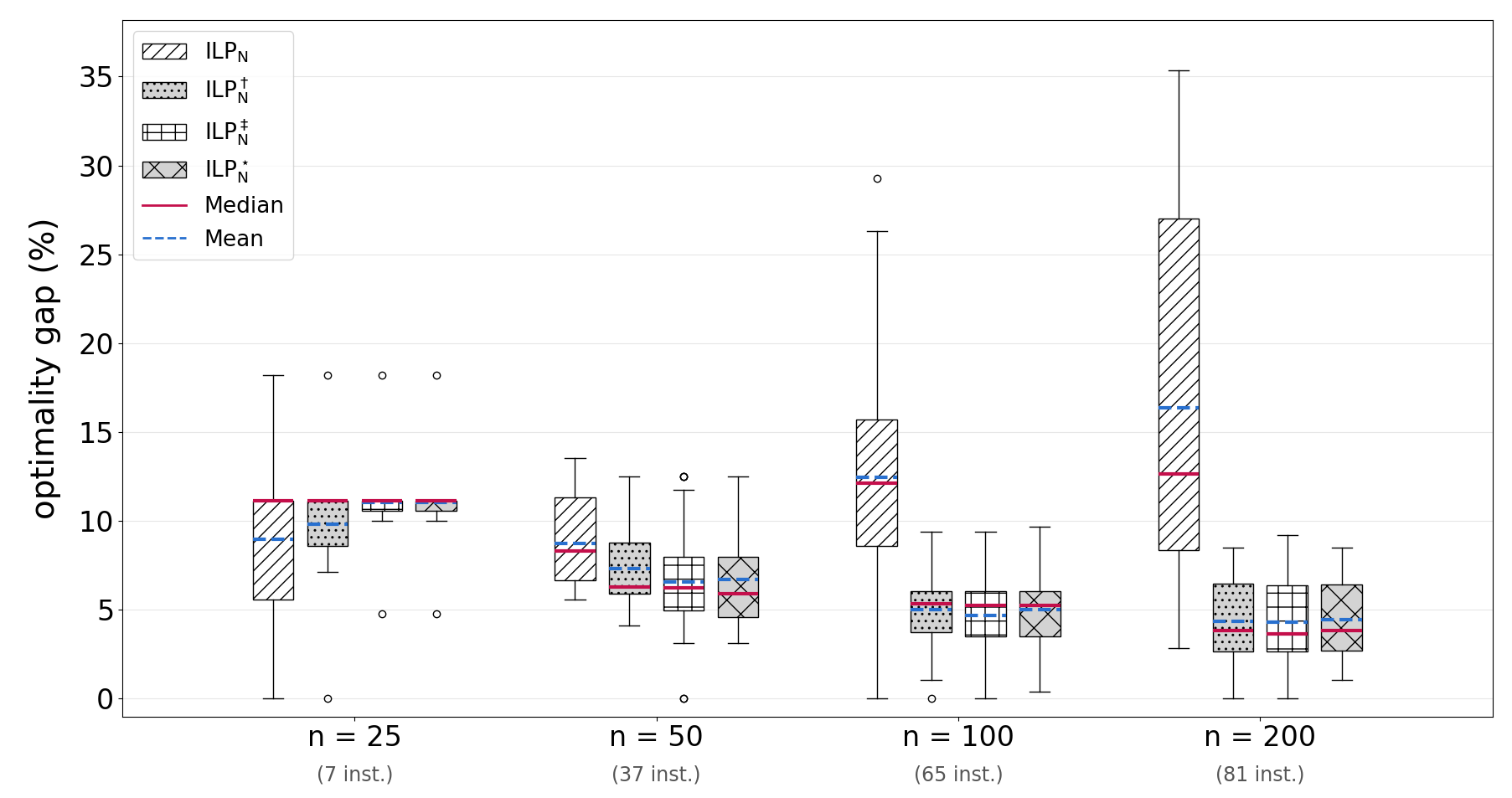}  
    \caption{Summary statistics (top) and box plots (bottom) of the optimality gaps of the variants of $\fone$, restricted to instances not solved to optimality by $\fthree$. Time limit: 1800 seconds.}
    \label{fig:boxplots_gap_natural}

\end{figure}

\begin{figure}[]
    \centering

     \scriptsize
    \renewcommand{\arraystretch}{1.25}
    \tabcolsep 13pt

    \begin{tabular}{lrrrrrrrrrr}
\toprule
        &      &  & \multicolumn{2}{r}{$\afone$} &  & \multicolumn{2}{r}{$\aftwo$} &  & \multicolumn{2}{r}{$\afthree$} \\ \cline{4-5} \cline{7-8} \cline{10-11} 
        &      &  & \multicolumn{2}{r}{gap (\%)} &  & \multicolumn{2}{r}{gap (\%)} &  & \multicolumn{2}{r}{gap (\%)}   \\ \cline{4-5} \cline{7-8} \cline{10-11} 
   & inst &  & median         & mean        &  & median         & mean        &  & median          & mean         \\ \cline{1-2} \cline{4-5} \cline{7-8} \cline{10-11} 
        &      &  &                &             &  &                &             &  &                 &              \\[-2ex]
$n=50$  & 7    &  & 7.0            & 7.6         &  & 3.9            & 4.6         &  & 5.6             & 7.1          \\[1ex]
$n=100$ & 39   &  & 5.6            & 5.0         &  & 3.9            & 3.7         &  & 4.0             & 3.7          \\[1ex]
$n=200$ & 39   &  & 3.0            & 5.1         &  & 3.3            & 5.3         &  & 2.7             & 4.8          \\[1ex] \bottomrule
\end{tabular}

\bigskip
    \includegraphics[width=0.95\linewidth]{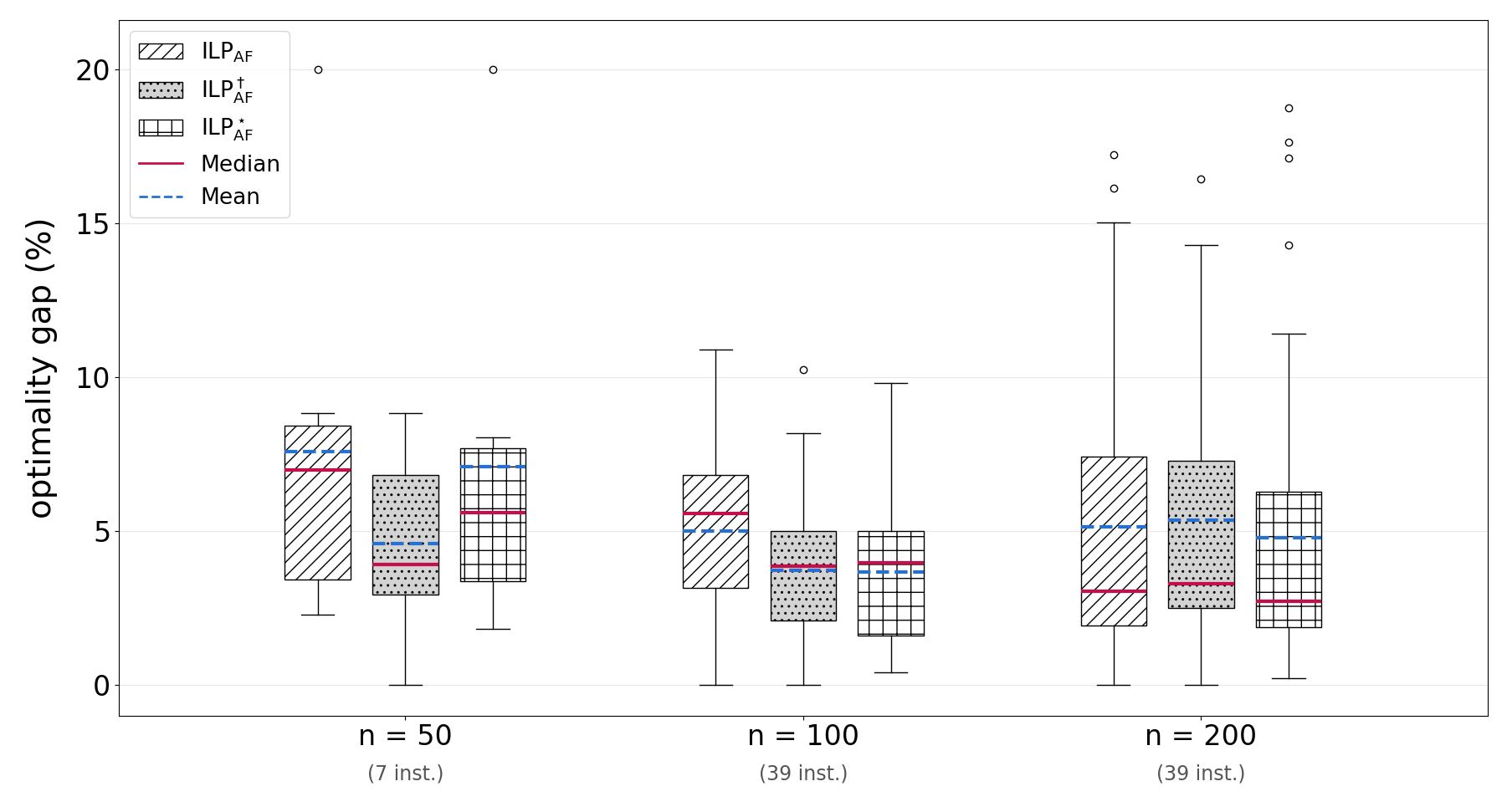}  
    \caption{Summary statistics (top) and box plots (bottom) of the optimality gaps of the variants of $\afone$, restricted to instances not solved to optimality by $\afthree$ and for which a valid dual bound was computed within the time limit for all the variants. Time limit: 1800 seconds.}
    \label{fig:boxplots_gap_arcflow}    
   
\end{figure}

\subsection{LP relaxation strength and optimal solution features: detailed results}
\label{sec:relaxation_instance_features}

In this section, we compare the LP relaxation lower bounds obtained with the variants of the {ILP} formulations we proposed in Sections~\ref{sec:ILP}~and~\ref{sec:arcflow}, and analyze the main features of the optimal solutions obtained on our test set. The results are reported in Table~\ref{tab:inst_info}. 
The table is organized into two main parts. It reports aggregate statistics for {the subset of instances with known optimum for which the LP relaxations of all tested formulations were solved to optimality within the time limit of 1800 seconds.} This subset is chosen because it provides a reliable and consistent basis for comparing bound quality across formulations. {The analysis covers a total of 427 random instances and 36 real-world instances, reported separately in the two panels of the table.}
{Random instances are grouped by different instance-generator parameter values, while real-world instances are grouped by the subset to which they belong.}

The first part of Table~\ref{tab:inst_info} (columns \emph{Integrality gaps}) reports the integrality gap, defined as the percentage difference between the optimal objective function value of the \BPPS{} and the optimal objective function value of the LP relaxation for a given formulation, taken with respect to the optimal objective function value of the \BPPS{}. We consider {the three natural formulations~$\fone$, $\ftwo$, and $\fthree$, whose optimal LP relaxation objective values are $\zeta(\foneLP)$, $\zeta(\ftwoLP)$, and $\zeta(\fthreeLP)$, respectively, as well as the three arc-flow formulations~$\afone$, $\aftwo$, and $\afthree$, whose optimal LP relaxation objective values are $\zeta(\afoneLP)$, $\zeta(\aftwoLP)$, and $\zeta(\afthreeLP)$, respectively.} As established in Equation~\eqref{eq:LPoptval_2bis}, the upper bound $\hat{\binnumberUB}$ does not affect the optimal objective function value of the LP relaxation for $\fthree$. Consequently, the optimal objective function values for the relaxations of $\ftwobis$ and $\fthree$ are identical, i.e., $\zeta(\fthreeLP) = \zeta(\ftwoLPbis)$. Therefore, the remainder of this section will focus exclusively on the value of $\zeta(\fthreeLP)$.

\begin{table}[h]
\scriptsize
\centering
\tabcolsep 2.5pt
\renewcommand\arraystretch{1.5}

\caption{Quality of the lower bound provided by the LP relaxation of the variants of the natural and arc-flow ILP formulations for the BPPS, and features of optimal BPPS solutions. The reported values refer to the subset of instances with a known optimum for which the LP relaxations of all tested formulations were solved to optimality within the 1800-second time limit.
}

\label{tab:inst_info}

\begin{tabular}{llrrrrrrrrrrrrrrrrr}
\hline
&       &        &  & \multicolumn{7}{r}{Integrality gap}  &  &  &  &  & \multicolumn{4}{r}{Optimal BPPS solution features} \\
\cline{5-11} \cline{16-19}
              &  & \#opt  &  & $\foneLP$   & $\ftwoLP$   & $\fthreeLP$   &  & $\afoneLP$   & $\aftwoLP$   & $\afthreeLP$   &  & $\binnumberLB$ & $\hat{\binnumberUB}$&  & \#bins     & \#items    & \#classes    & \%fill    \\ \cline{1-3} \cline{5-7} \cline{9-11} \cline{13-14} \cline{16-19}
                   &             &        &  &           &           &             &  &            &            &             &  &                            &                           &  &            &            &              &           \\[-3ex]
{\tt item number} & $n$ = 25 & 144 &  & 15.6 & 12.4 & 3.5 &  & 8.3 & 7.6 & 2.0 &  & 5.3 & 8.6 &  & 5.6 & 5.5 & 2.0 & 88.1 \\
 & $n$ = 50 & 130 &  & 17.8 & 8.1 & 3.3 &  & 6.7 & 4.8 & 1.8 &  & 10.3 & 13.4 &  & 10.7 & 5.7 & 1.5 & 92.7 \\
 & $n$ = 100 & 91 &  & 19.6 & 5.1 & 2.2 &  & 5.6 & 3.1 & 1.1 &  & 20.3 & 23.4 &  & 20.9 & 6.1 & 1.3 & 95.8 \\
 & $n$ = 200 & 62 &  & 20.8 & 3.7 & 2.1 &  & 4.0 & 2.1 & 0.8 &  & 40.8 & 45.1 &  & 42.0 & 6.6 & 1.2 & 96.8 \\
[0.75ex]
{\tt class number} & $m$ = 5 & 225 &  & 20.5 & 8.4 & 3.0 &  & 7.6 & 5.5 & 1.8 &  & 15.7 & 17.9 &  & 16.2 & 5.8 & 1.3 & 91.9 \\
 & $m$ = 10 & 202 &  & 14.9 & 8.2 & 2.9 &  & 5.6 & 4.4 & 1.2 &  & 14.6 & 19.2 &  & 15.1 & 5.8 & 1.9 & 92.9 \\
[0.75ex]
{\tt bin capacity} & $d$ = 200 & 174 &  & 17.9 & 7.6 & 2.2 &  & 6.7 & 4.9 & 1.1 &  & 16.0 & 19.3 &  & 16.4 & 6.7 & 1.5 & 93.1 \\
 & $d$ = 1,000 & 141 &  & 18.1 & 8.9 & 3.8 &  & 6.7 & 5.4 & 2.1 &  & 14.5 & 17.9 &  & 15.2 & 5.6 & 1.6 & 91.7 \\
 & $d$ = 10,000 & 112 &  & 17.5 & 8.5 & 3.2 &  & 6.3 & 4.8 & 1.7 &  & 14.7 & 18.2 &  & 15.3 & 4.8 & 1.7 & 92.1 \\
[0.75ex]
{\tt bin-setup costs} & no & 224 &  & 14.5 & 9.4 & 2.9 &  & 6.9 & 6.3 & 1.6 &  & 15.9 & 19.3 &  & 16.4 & 5.8 & 1.7 & 93.6 \\
 & yes & 203 &  & 21.6 & 7.1 & 3.1 &  & 6.3 & 3.5 & 1.5 &  & 14.4 & 17.6 &  & 14.9 & 5.9 & 1.5 & 91.1 \\
[0.75ex]
{\tt item weights} & small & 90 &  & 15.3 & 11.9 & 1.8 &  & 11.1 & 9.4 & 1.4 &  & 5.3 & 9.2 &  & 5.4 & 11.6 & 2.3 & 88.9 \\
 & medium & 107 &  & 16.4 & 6.5 & 1.3 &  & 6.6 & 4.5 & 0.9 &  & 13.0 & 15.8 &  & 13.1 & 5.2 & 1.4 & 93.8 \\
 & large & 133 &  & 21.9 & 8.5 & 6.1 &  & 3.5 & 2.5 & 2.4 &  & 25.5 & 29.5 &  & 27.1 & 3.1 & 1.3 & 92.6 \\
 & mixed & 97 &  & 16.3 & 6.6 & 1.6 &  & 6.7 & 4.8 & 1.3 &  & 12.4 & 15.2 &  & 12.5 & 4.9 & 1.5 & 93.8 \\
[0.75ex]
{\tt setup weights} & small & 137 &  & 14.7 & 7.8 & 2.0 &  & 6.4 & 5.1 & 1.0 &  & 13.4 & 16.9 &  & 13.8 & 6.4 & 1.8 & 92.6 \\
 & large & 151 &  & 20.4 & 8.8 & 4.0 &  & 6.7 & 4.9 & 2.1 &  & 16.6 & 19.8 &  & 17.3 & 5.4 & 1.4 & 91.9 \\
 & mixed & 139 &  & 18.2 & 8.2 & 2.8 &  & 6.7 & 5.1 & 1.5 &  & 15.3 & 18.8 &  & 15.8 & 5.7 & 1.6 & 92.7 \\
[0.75ex]
\cline{1-3} \cline{5-7} \cline{9-11} \cline{13-14} \cline{16-19}
Total/Average &  & 427 &  & 17.8 & 8.3 & 3.0 &  & 6.6 & 5.0 & 1.6 &  & 15.2 & 18.5 &  & 15.7 & 5.8 & 1.6 & 92.4 \\
\hline
{\tt real-world} & {\tt rwA} & 12 &  & 19.2 & 9.9 & 9.4 &  & 0.9 & 0.9 & 0.9 &  & 57.5 & 63.2 &  & 63.0 & 2.1 & 1.0 & 90.9 \\
 & {\tt rwB} & 12 &  & 22.7 & 14.7 & 14.1 &  & 0.4 & 0.4 & 0.4 &  & 59.8 & 70.7 &  & 69.8 & 1.9 & 1.0 & 85.2 \\
 & {\tt rwC} & 12 &  & 22.9 & 13.3 & 12.9 &  & 0.4 & 0.4 & 0.4 &  & 88.2 & 101.7 &  & 101.0 & 1.3 & 1.0 & 87.2 \\
[0.75ex]
\cline{1-3} \cline{5-7} \cline{9-11} \cline{13-14} \cline{16-19}
Total/Average &  & 36 &  & 21.6 & 12.7 & 12.1 &  & 0.5 & 0.5 & 0.5 &  & 68.5 & 78.5 &  & 77.9 & 1.7 & 1.0 & 87.8 \\
\bottomrule
\end{tabular}
\end{table}

The first part of the table (\emph{Integrality gap}) allows us to directly quantify the strengthening effect of the MCIs and that of the MBI for the two proposed formulations and across different instance parameters.
{For the natural formulations, the} results demonstrate that incorporating the MCIs significantly reduces the integrality gap across all problem settings compared to~$\foneLP$. {Formulation~$\fthreeLP$} achieves a further reduction {when the MBI is added.} {On average over all random instances, the gap decreases from $17.8\%$ for $\foneLP$ to $8.3\%$ for $\ftwoLP$ and $3.0\%$ for $\fthreeLP$.} {The arc-flow formulations achieve tighter bounds at each level of enhancement: the average gap is $6.6\%$ for~$\afoneLP$, $5.0\%$ for~$\aftwoLP$, and $1.6\%$ for~$\afthreeLP$.}
{Furthermore, as~$n$ increases, the gap of the base natural formulation grows ($15.6\%$ for $\foneLP$ at $n=25$ to $20.8\%$ at $n=200$), while those of the enhanced natural variants and of the arc-flow variants decrease substantially. For~$n=200$, $\afthreeLP$ achieves an average gap below~$1\%$.} Conversely, when the number of classes increases from $m=5$ to $m=10$, the gap of~$\foneLP$ decreases while those of the enhanced natural variants do not change significantly; $\fthreeLP$ and~$\afthreeLP$ maintain gaps of $2.9\%$ and $1.2\%$, respectively, at $m=10$.
Changes in bin capacity~$d$ have only a marginal effect on the integrality gaps across all formulations, suggesting that this parameter is less critical for bound quality.
The presence of setup costs increases the gap for~$\foneLP$ from $14.5\%$ to $21.6\%$, whereas $\fthreeLP$ stays around $3\%$. {Among the arc-flow variants, the gap of~$\afoneLP$ decreases slightly from~$6.9\%$ (no setup costs) to~$6.3\%$ (with setup costs), while~$\afthreeLP$ achieves a gap of approximately $1.5\%$ in both cases.} This behavior {for the natural formulations} can be attributed to the fundamental difference in how the formulations handle class setup costs within their LP relaxation optimal objective function value. The optimal objective function value of~$\foneLP$ accounts for each class exactly once (see Equation~\eqref{eq:LPoptval}), regardless of how many bins actually contain items of that class. In contrast, formulations~$\ftwoLP$ and~$\fthreeLP$, incorporating the MCIs, include an estimate of the number of times each class $c \in \classes$ is activated across different bins, namely the $\binnumberLBclass_c$ coefficient (see Equation~\eqref{eq:LPoptval_2}).
Consequently, when setup costs are present, the LP relaxation of~$\fone$ fails to capture this term of the objective function, as it underestimates the total setup cost by not accounting for multiple activations of the same class.
Finally, item weights exhibit a notable interaction with the quality of the bound provided by the linear relaxations of both formulations. For instances with large item weights, the gap of~$\fthreeLP$ rises to~$6.1\%$, noticeably higher than for other weight categories. As far as setup weights are concerned, small setup weights tend to reduce the gap for all the variants of $\foneLP$. The arc-flow formulation variants behave differently: setup weights do not seem to significantly influence the integrality gap, while all the arc-flow variants achieve small gaps on the instances with large-sized items, below $3.5\%$ on average.

The second part of the table includes the average values of the lower bound $\binnumberLB$ and the upper bound $\hat{\binnumberUB}$, computed as detailed at the beginning of Section~\ref{sec:performance}, as well as the structural
characteristics of the optimal solutions found (\emph{Optimal BPPS Solution Features}): these last columns report the average number of bins used in the optimal solution, the average number of items per bin, the average number of active classes per bin, and the
average bin fill percentage.
The effectiveness of the bounds can be assessed by comparing their average values to the actual average number of bins in optimal \BPPS{} solutions. The lower bound~$\binnumberLB$ demonstrates high accuracy across all parameter configurations, with average values differing from the actual bin count by at most~$1.6$ for instances with large-sized items, and approaching the optimal bin count in many subsets of instances.
The upper bound~$\hat{\binnumberUB}$, on the other hand, provides a looser estimate: for instances with $m=10$, its average value is~$19.2$ while the actual average bin count is~$15.1$, a difference exceeding~$25\%$. Despite this imprecision,
$\hat{\binnumberUB}$ remains substantially smaller than the trivial bound~$n$ across all instance sizes, with the most significant reduction observed for larger instances (average $\hat{\binnumberUB} = 45.1$ vs.\ $n=200$). This reduction enables a substantial decrease in the number of variables and constraints within formulation~$\fthree$, thereby contributing to its superior computational performance relative to the other $\fone$ variants.

The solution statistics on random instances reveal that optimal solutions typically contain between~$5$ and~$7$ items per bin, involve a small number of active classes (around~$1.6$ on average), and achieve a high bin utilization rate ($92.4\%$ on average) across all parameter configurations. When the number of items~$n$ increases, the average number of items per bin grows while the number of active classes per bin decreases, reflecting a tendency towards more homogeneous bin composition in larger instances. Increasing the number of classes~$m$ results in slightly more diverse bins, whereas variations in bin capacity~$d$ have minimal impact on these statistics. The presence of setup costs is associated with marginally fewer active classes per bin, suggesting a packing strategy that limits costly setups. Instances with small item weights produce solutions with more items per bin ($11.6$ on average) and a higher class diversity ($2.3$ active classes per bin). In comparison, large item weights lead to fewer items per bin ($3.1$) and low class diversity ($1.3$).

For the real-world instances, the integrality gaps of the natural formulations remain large ($21.6\%$ for~$\foneLP$ and $12.1\%$ for~$\fthreeLP$ on average), whereas the arc-flow formulations achieve gaps below $1\%$ across all three instance groups.
This contrast suggests that the arc-flow relaxation is particularly well-suited to the structure of real-world \BPPS{} instances. The solution statistics reflect the different nature of these instances: optimal solutions use far more bins on average (up to~$101$ for family~\texttt{rwC}), contain fewer items per bin ($1.7$ on average), and involve exactly one active class per bin in all groups of instances.

\end{document}